%% file: FullPaper.tex
\documentclass[11pt,reqno]{amsart}

\usepackage{tikz-cd} 

\usepackage{amsmath,amsthm,amssymb}
\numberwithin{equation}{section}

\usepackage{graphicx, psfrag, amsmath, amscd, amssymb,mathtools}
\usepackage{lipsum}
\usepackage{subcaption} 

\usepackage[bookmarksnumbered, colorlinks, plainpages]{hyperref}
\hypersetup{colorlinks=true,linkcolor=red, anchorcolor=green, citecolor=cyan, urlcolor=red, filecolor=magenta, pdftoolbar=true}

\usepackage{shuffle} 

\usepackage{titletoc}

\newtheorem{thm}{Theorem}[section]
\newtheorem{prop}[thm]{Proposition}
\newtheorem{lem}[thm]{Lemma}
\newtheorem{cor}[thm]{Corollary}

\theoremstyle{definition}
\newtheorem{defi}[thm]{Definition}
\newtheorem{exa}[thm]{Example}
\newtheorem{prob}{Problem}
\newtheorem{ques}{Question}
\newtheorem{conj}[thm]{Conjecture}

\newtheorem{rem}{Remark}[section]

\newcommand{\B}{{\mathcal{B}}}

\newcommand{\F}{{\mathcal{F}}}
\newcommand{\G}{{\mathcal{G}}}
\newcommand{\I}{{\mathscr{I}}}

\newcommand{\Code}{{\mathtt{Code}}}

\newcommand{\DES}{{\mathsf{DES}}}
\newcommand{\des}{{\mathsf{des}}}

\newcommand{\inv}{{\mathsf{inv}}}
\newcommand{\exc}{{\mathsf{exc}}}
\newcommand{\EXC}{{\mathsf{EXC}}}

\newcommand{\fixTwo}{{\mathsf{fix_2}}}
\newcommand{\maj}{{\mathsf{maj}}}

\newcommand{\DEX}{{\mathsf{DEX}}}

\newcommand{\ch}{{\mathrm{ch}}}
\newcommand{\ps}{{\mathrm{ps}_q}}
\newcommand{\conv}{{\mathrm{conv}}}
\newcommand{\ind}{{\mathrm{ind}}}
\newcommand{\rk}{{\mathrm{rk}}}
\newcommand{\Conj}{{\mathsf{Conj}}}

\newcommand{\N}{{\mathcal{N}}}

\renewcommand{\I}{{\mathcal{I}}}
\renewcommand{\S}{{\mathfrak{S}}}

\newcommand{\x}{{\mathbf{x}}}
\renewcommand{\L}{{\mathcal{L}}}

\newcommand{\qbin}[2]{\begin{bmatrix}{#1}\\ {#2}\end{bmatrix}_q} 

\title[Stembridge codes, permutahedral varieties, and their extensions]{Stembridge codes, permutahedral varieties, and their extensions}

\author{Hsin-Chieh Liao}
\address{Department of Mathematics, University of Miami, USA}
\email{h.liao@math.miami.edu}

\date{}

\begin{document}

\begin{abstract}
It is well known that the Eulerian polynomial is the Hilbert series of the cohomology of the permutahedral variety. Stanley obtained a formula showing that the cohomology carries a permutation representation of $\mathfrak{S}_n$. We answer a question of Stembridge on finding an explicit permutation basis of this cohomology. We observe that the Feichtner-Yuzvinsky basis for the Chow ring of the Boolean matroid is such a permutation basis, and then we construct an $\mathfrak{S}_n$-equivariant bijection between this basis and codes introduced by Stembridge, thereby giving a combinatorial proof of Stanley's formula. We obtain an analogous result for the stellahedral variety. We find a permutation basis of the permutation representation its cohomology carries. This involves the augmented Chow ring of a matroid introduced by Braden, Huh, Matherne, Proudfoot and Wang. Along the way, we obtain a general result on augmented Chow rings (which was also independently obtained by Eur) asserting that augmented Chow rings of matroids are actually Chow rings in the sense of Feichtner and Yuzvinsky. 
In the last part of the paper, we study enumerative aspects of the permutahedra and the stellohedra related to these permutation bases.
\end{abstract}

\maketitle

\dottedcontents{section}[0em]{}{0em}{0.5pc}
\dottedcontents{subsection}[0.5em]{}{0em}{0.5pc}

\setcounter{tocdepth}{2} 
\tableofcontents

\section{Introduction}
A lattice polytope in $\mathbb{R}^n$ is a polytope whose vertices all lie in some lattice of rank $n$ in $\mathbb{R}^n$. Given a lattice polytope $P$ containing the origin, let $\Sigma(P)$ be its normal fan. Then the fan $\Sigma(P)$ can be associated with a toric variety $X_{\Sigma(P)}$.
Consider the $(n-1)$-dimensional permutahedron 
\[
    \Pi_n\coloneqq\conv\{(\sigma_1,\ldots,\sigma_n)\in\mathbb{R}^n : \sigma=\sigma_1\sigma_2\ldots\sigma_n\in\S_n\}.
\] 
Its normal fan $\Sigma_n=\Sigma(\Pi_n)$ can be obtained from the \emph{braid arrangement} $H_{i,j}\coloneqq\{x\in\mathbb{R}^n : x_i=x_j\}$ for $1\le i<j\le n$. The toric variety $X_{\Sigma_n}$ associated to $\Sigma_n$ is called the \emph{permutahedral variety}. An intriguing fact about $X_{\Sigma_n}$ is that the Hilbert series of its cohomology $H^*(X_{\Sigma_n})$ is the \emph{Eulerian polynomial} (see section \ref{subsec:Permutations and Symmetric functions}).

The cohomology $H^*(X_{\Sigma_n})$ carries a representation of $\S_n$ induced by the $\S_n$-action on $\Sigma_n$. Using a recurrence of Procesi \cite{Procesi1990}, Stanley \cite{Stanley1989} computed its Frobenius series, which shows this representation is a permutation representation. 
Stembridge \cite{Stembridge1992} introduced combinatorial objects, which we call \emph{Stembridge codes}, and showed that the representation of $\S_n$ on the linear span of Stembridge codes has the same Frobenius series as $H^*(X_{\Sigma_n})$. He then asked if there is a geometric explanation of the permutation representation on $H^*(X_{\Sigma_n})$, or more explicitly, can we find a basis of $H^*(X_{\Sigma_n})$ on which the $\S_n$-action induces the permutation representation of $\S_n$ on $H^*(X_{\Sigma_n})$ (see \cite[p.317]{Stembridge1992}, \cite[p.296, Problem 11.2]{Stembridge1994}).
    
This question recently attracted much interest due to its connection with the well-known Stanley-Stembridge Conjecture \cite{StanleyStembridge1993} and its geometric counterpart. 
The conjecture states that the chromatic symmetric function of the incomparability graph of any $(3+1)$-free poset is $e$-positive, and was reduced to the incomparability graph of any $(3+1)$ and $(2+2)$-free poset, also known as the natural unit interval order, by Guay-Paquet \cite{Guay2013modular}.   
Shareshian and Wachs \cite{ShareshianWachs2016Chromatic} then conjectured that such chromatic symmetric functions are the Frobenius characteristic of the $\S_n$-representation induced by  Tymoczko's ``dot action'' \cite{TymoczkoEqCoho2007} on the cohomology ring of a regular semisimple Hessenberg variety. This link to geometry was proved by Brosnan and Chow \cite{BrosnanChow2018} and Guay-Paquet \cite{Guay2016secondproof}. As a result, one way to settle the Stanley-Stembridge conjecture is to prove the $\S_n$-representation on the cohomology of the Hessenberg variety is a permutation representation; or more concretely, find a permutation basis for this $\S_n$-represnetation. The permutahedral variety $X_{\Sigma_n}$ is a special case of the regular semisimple Hessenberg variety and the representation on $H^*\left(\Sigma_n\right)$ induced from the natural $\S_n$-action on $\Sigma_n$ coincides with that induced from the dot action. Therefore Stembridge's question relates to a special case of the Stanley-Stembridge conjecture.

In 2015, Chow \cite{ChowErasing} proposed his \emph{erasing mark conjecutre} claiming that a set of classes in the equivariant cohomology of $X_{\Sigma_n}$, in terms of GKM theory, forms a permutation basis under the dot action and descends to a permutation basis for the usual cohomology. This conjecture was recently proved by Cho, Hong, and Lee \cite{ChoHongLee2020} using connections to flag and Schubert varieties.  Another work regarding Stembridge's question is by Lin \cite{Lin2022PermVar}. Lin considered the construction of the permutahedral variety by consecutively blowing up subvarieties in $\mathbb{P}^{n-1}$. By careful study of the geometry after each blow-up, Lin obtained a desired permutation basis compatible with the structure of Stembridge codes.

The results in this paper can be divided into three parts. 
In the first part, we take a more combinatorial path to answer Stembridge's question by considering the \emph{Chow ring of a matroid}, which was introduced by Feichtner and Yuzvinsky \cite{FeichtnerYuzvinsky2004} and played an important role in Adiprasito, Huh, and Katz's seminal paper \cite{AHK2018} that set the foundation of the Combinatorial Hodge theory and solved the Heron-Rota-Welsh Conjecture. We answer Stembridge's question by identifying $H^*(X_{\Sigma_n})$ with the Chow ring of the Boolean matroid and then giving a permutation basis for the induced action of $\S_n$ on the Chow ring. We show that a basis of Feichtner and Yuzvinsky for general matroids serves the purpose when we apply it to Boolean matroids. Furthermore, we construct an $\S_n$-equivariant bijection between this basis and Stembridge codes, which enables us to give a combinatorial proof of Stanley's formula (\ref{eq:FrobPerm}).

In the second part, we study a parallel story involving another simple polytope called the \emph{stellohedron} $\widetilde{\Pi}_n$. The story begins with the \emph{binomial Eulerian polynomial}, which Postnikov, Reiner, Williams \cite{PRW2008} show, is equal to the Hilbert series of the cohomology of the toric variety $X_{\widetilde{\Sigma}_n}$ associated to $\widetilde{\Pi}_n$. 
Shareshian and Wachs~\cite{ShareshianWachs2020} show that the representation of $\S_n$ on this cohomology is also a permutation representation. 
We apply similar ideas to the stellohedron $\widetilde{\Pi}_n$  and successfully give a permutation basis for $H^*(X_{\widetilde{\Sigma}_n})$.
This basis leads to a combinatorial proof of Shareshian and Wachs' analog of Stanley's formula for $H^*(X_{\widetilde{\Sigma}_n})$.
We obtain the permutation basis by proving for any matroid $M$, the \emph{augmented Chow ring $\widetilde{A}(M)$ of $M$}, introduced by Braden, Huh, Matherne, Proudfoot, and Wang \cite{BHMPW2020+}, is actually a special case of the Feichtner-Yuzvinsky's Chow ring \cite{FeichtnerYuzvinsky2004}, which leads to a basis for $\widetilde{A}(M)$. 
After this work was completed, we learned from \cite[Section 5.1]{Mastroeni2022Koszul} that the construction of the augmented Chow ring as the Feichtner-Yuzvinsky's Chow ring was also found independently by Eur. It was later included in Eur, Huh, Larson \cite{EHL2022stellahedral}.

In the third part, we study the combinatorial side of the two stories. The \emph{Eulerian quasisymmetric functions} and the $q$-Eulerian polynomials of Shareshian and Wachs \cite{ShareshianWachs2010} are considered. 
We give an interpretation of the $q$-Eulerian polynomials in terms of Stembridge codes. This is a $q$-analogue of Stembridge's interpretation of the Eulerian polynomials in \cite{Stembridge1992}.
We further extend the work of Shareshian and Wachs on the stellohedron-analogs of the Eulerian quasisymmetric functions and the $q$-Eulerian polynomials by using the extended Stembridge codes and Postnikov's \emph{decorated permutations} \cite{Postnikov2006Positroid}.
We also exhibit two new instances of Reiner, Stanton, White's \emph{cyclic sieving phenomenon} that naturally arise from the two stories.


The paper is organized as follows:

In Section \ref{Sec:Prelim}, we provide some background on matroids and some background on Feichtner and Yuzvinsky's theory on building sets of atomic lattices and the corresponding Chow rings, which shall be used in Sections \ref{Sec:EulerianStory}, \ref{Sec:BinEulerianStory} and \ref{Sec:AugIsChow}.

In Section \ref{Sec:EulerianStory}, we first briefly review the connection between the Eulerian polynomial, the permutahedron, and Stembridge's codes. Then we present the first part of our results related to the permutation basis for the representation of $\S_n$ on the cohomology of the permutahedral variety.

In Section \ref{Sec:BinEulerianStory}, we recall some connections between the binomial Eulerian polynomial and the stellohedron, and some known results on the $\S_n$-representation that the cohomology of the stellahedral variety carries. Then we present the second part of our results, which aimed at building a story parallel to Section \ref{Sec:EulerianStory}, on the permutation basis for the representation of $\S_n$ on the cohomology of the stellahedral variety. Here we present such a basis directly and leave the proof that the expression is indeed a basis for the next section.

In Section \ref{Sec:AugIsChow}, we present how the basis is constructed. We first introduce another construction of the stellohedron as the graph associahedron of the $n$-star graph $K_{1,n}$ and compare it with the construction in Section \ref{Sec:BinEulerianStory}. Then we generalize the graph associahedron construction to any matroid. This leads us to show that the augmented Chow ring of a matroid is actually a Chow ring in the sense of Feichtner and Yuzvinsky. This result allows us to construct a basis for the augmented Chow ring of a matroid.  

In Section \ref{Sec:CombTwoStories}, we first introduce some background on permutations and symmetric functions. Then we present the third part of our results on the exploration of the combinatorial side of the Eulerian and binomial Eulerian stories. 

\begin{rem}
Some of the results in this paper were announced in our extended abstract \cite{Liao2023FPSAC}. Here we provide the details and additional results.   
\end{rem}

\section{Background on matroids and building sets}
\label{Sec:Prelim}

For a positive integer $n$, we write $[n]\coloneqq\{1,2,\ldots, n\}$. \label{Sec:Prelim}

\subsection{Matroids}
Matroids are combinatorial objects that unify many notions of independence in different contexts. There are many equivalent ways to define a matroid. In this paper, we define a matroid in terms of its \emph{independednt sets}. For a more thorough introduction to matroid theory, we refer the reader to Oxley \cite{Oxley2006matroid}. See \cite[Section 2]{StanleyEC1} for background on posets.

\begin{defi} \label{def:matroid}
A matroid $M$ on a finite set $E$ is a pair $(E,\I$ such that $\I$ is a subset of $2^E$ satisfying the following three axioms:
\begin{itemize}
    \item[(I1)] $\emptyset\in \I$.
    \item[(I2)] If $I\in\I$, then $J\in\I$ for any $J\subseteq I$.
    \item[(I3)] If $I,~J\in \I$ with $|I|>|J|$, then there exists some $i\in I\setminus J$ such that $J\cup\{i\}\in\I$.
\end{itemize}
\end{defi}

An important motivating example of a matroid comes from linear algebra. Let $E$ be a finite set of vectors in a fixed vector space and $\I$ be the collection of all linearly independent subsets of $E$, then $(E,\I)$ forms a matroid. Such a matroid is called \emph{linear} or \emph{representable}. Not all the matroids are representable; in fact, most of the matroids are not representable, see \cite{NelsonMatroids2018}. 

For a matroid $M=(E,\I)$, it is customary to call the sets in $\I$ an \emph{independent set} of $M$. When necessary, we write $\I(M)$ instead of $\I$ to specify $M$. A maximal independent set in $\I$ is called a \emph{basis} of $M$. Every basis of $M$ has the same cardinality, which we refer to as the \emph{rank} of $M$, denoted by $\rk(M)$. We may associate a rank to each subset $S\subseteq E$ as well by letting $\rk_M(S)\coloneqq \max\left(|I|: I\subseteq S,~I\in\I\right)$. In particular, $\rk_M(E)=\rk(M)$.

\begin{exa}
The uniform matroid $U_{k,n}$ is a matroid with the ground set $E=[n]$. The independent sets of $U_{k,n}$ are the subsets of $[n]$ of cardinality at most $k$. The bases of $U_{k,n}$ are the subsets of $[n]$ of cardinality $k$. The rank of $U_{k,n}$ is $k$ and for any $S\subseteq [n]$,
\[
    \rk_{U_{k,n}}(S)=\begin{cases}
            |S| & \mbox{, if }|S|<k;\\
            k   & \mbox{, if }|S|\ge k.
        \end{cases}
\]
In this paper, we are concerned mostly with the special case of $k=n$, that is the matroid $U_{n,n}$, which is also called the \emph{Boolean matroid} $\mathsf{B}_n$. 
\end{exa}

A subset $F\subseteq E$ such that $\rk_M(F\cup\{i\})>\rk_M(F)$ for any $i\in E\setminus F$ is called a \emph{flat} of $M$. Let $\F=\F(M)$ denote the collection of flats of $M$. Under the inclusion partial order ``$\subseteq$''  inherited from $2^{E}$, the poset $\L(M)=(\F(M),\subseteq)$ forms a geometric lattice called the \emph{lattice of flats of $M$}. 

On the other hand, consider the collection $\I=\I(M)$ of independent sets, the axioms (I1) and (I2) in Definition \ref{def:matroid} imply that $\I(M)$ forms an abstract simplicial complex, called the \emph{independence complex} of $M$. 

For convenience we sometimes write a subset $\{i_1,i_2,\ldots,i_k\}$ of $[n]$ as $i_1i_2\ldots i_n$.
\begin{exa}
For the uniform matroid $U_{2,3}$, we have
\[
    \I(U_{2,3})=\{\emptyset,1,2,3,12,13,23\} \mbox{ and }
    \F(U_{2,3})=\{\emptyset,1,2,3,123\}.
\]
The corresponding independence complex and the lattice of flats are shown in figure \ref{fig:indepComplex} and \ref{subfig:lattice_flat} respectively.

\begin{figure}[h]
    \centering
    \begin{subfigure}{0.35\textwidth}
        \centering
        \include{Tikzpictures/IndCplxExample}
        \caption{Independence complex $\mathcal{I}(U_{2,3})$}    
        \label{fig:indepComplex}
    \end{subfigure}
    \begin{subfigure}{0.35 \textwidth}
        \centering
       \include{Tikzpictures/LatticeFlatsExample}
        \caption{Lattice of flats $\L(U_{2,3})$}
        \label{subfig:lattice_flat}
    \end{subfigure}
    \caption{}
\end{figure}
\end{exa}

We say that an element $e\in E$ is a \emph{loop} of $M$ if it is contained in no basis; dually, an element $e\in E$ is a \emph{coloop} of $M$ if it is contained in every basis.

\subsection{Building sets and Chow rings of atomic lattices} \label{sec:BS}
Next we recall some background about abstract building sets, nested set complexes, and Chow rings of the atomic lattice introduced in Feichtner-Yuzvinsky \cite{FeichtnerYuzvinsky2004} and Feichtner-Kozlov \cite{FeichtnerKozlov2004}. They constructed this purely order-theoretic framework as an abstraction of the incidence combinatorics occurring in De~Concini and Procesi's wonderful compactification \cite{DeCociniProcesiWonderful} of hyperplane arrangement complements.

For a poset $P$, we denote $\hat{0}$ the unique minimal element and $\hat{1}$ the unique maximal element of $P$ if they exist. 
If $X, Y \in P$ with $X\le Y$, then we denote $[X,Y]=\{Z\in P: X\le Z\le Y\}$ the interval between $X$ and $Y$.
For a poset $P$, a subset $Q$ of $P$, and $X\in P$, we write $Q_{\le X}=\{Y\in Q: Y\le X\}$. Similarly, we can define $Q_{<X}$, $Q_{\ge X}$, $Q_{>X}$. 

Let $\L$ be an atomic lattice. A subset $\G\subseteq \mathcal{L}-\{\widehat{0}\}$ is a \emph{building set} of $\mathcal{L}$ if for any $X\in\L-\{\widehat{0}\}$ with subset $\max(\G_{\le X})=\{G_1,\ldots,G_k\}$, there is a poset isomorphism 
\[
   \varphi_X: \prod_{i=1}^k[\hat{0},G_i]\longrightarrow [\hat{0},X]
\]
with $\varphi_X(\widehat{0},\ldots,\widehat{0},G_i,\widehat{0},\ldots,\widehat{0})=G_i$ for $i=1,\ldots,k$.

\begin{exa}
Consider the Boolean lattice $B_3$ which consists of all subsets of $[3]$ ordered by inclusion.
\begin{itemize}
    \item The set $\G=\{1,2,3,23\}$ is a building set of $B_3$. For example, take $X=123$ then $\max(\G_{\le 123})=\{1,23\}$ and $[\emptyset,123]\cong[\emptyset,1]\times[\emptyset,23]$. This works for all $X\in B_3-\{\emptyset\}$.
    \item The set $\mathcal{H}=\{1,2,3,12,13,23\}$ is not a building set. Let's look at $X=123$ again. Then $\max(\mathcal{H}_{\le 123})=\{12,13,23\}$ while $[\emptyset,123]\ncong [\emptyset, 12]\times[\emptyset, 13]\times[\emptyset, 23]$.
\end{itemize}
\begin{figure}[h]

\begin{subfigure}{0.2\textwidth}
\input{Tikzpictures/BuildingSetEx1.tex}
\caption{Building set}
\end{subfigure}
\hspace{2 cm}
\begin{subfigure}{0.25\textwidth}
\input{Tikzpictures/BuildingSetEX2.tex}
\caption{Not a building set}
\end{subfigure}

\end{figure}
One can easily observe that for any atomic lattice $\L$, the set $\L_{>\hat{0}}$ is always a building set and is the maximal one.

\end{exa}

Given a building set $\G$ of $\L$, we say a subset $N\subseteq\G$ is \emph{nested} (or \emph{$\G$-nested}) if for any pairwise incomparable elements $G_1,\ldots, G_t\in N$ ($t\ge 2$), their join $G_1\vee\ldots\vee G_t\notin~\G$. Notice that the collection of all $\G$-nested sets forms an abstract simplicial complex $\mathcal{N}(\L,\G)$ on the vertex set $\G$, which is called the \emph{nested set complex}. If $\G$ contains the maximal element $\widehat{1}$ then $\mathcal{N}(\L,\G)$ is a cone with apex $\{\widehat{1}\}$; in this case, the base of the cone is called the \emph{reduced nested set complex} $\widetilde{\mathcal{N}}(\L,\G)$.

\begin{exa}
The set $\G=\{1,2,3,23,123\}$ is a building set of $B_3$ containing the maximal element $123$. The corresponding $\N(\L,\G)$ and $\widetilde{N}(\L,\G)$ are the following:

\begin{figure}[h]
\begin{subfigure}{0.2 \textwidth}
    \begin{tikzpicture}[scale=0.7]
        \node [circle, fill, inner sep=2pt] (1) at (0,2) {};
        \node [left] at (0,2.5) {\footnotesize 1};
        \node [circle, fill, inner sep=2pt] (2) at (2.5,2.5) {};
        \node [right] at (2.5,3) {\footnotesize 2};
        \node [circle, fill, inner sep=2pt] (3) at (0,0) {};
        \node [left] at (0,-.5) {\footnotesize 3};
        \node [circle, fill, inner sep=2pt] (23) at (3,0) {};
        \node [right] at (3,-.5) {\footnotesize 23};
        \node [circle, fill, inner sep=2pt] (123) at (1.5,1) {};
        \node [above] at (1.5,1.2) {\footnotesize 123};
    
        \draw (1)--(123)--(2)--(23)--(123)--(3)--(23);
        \draw (3)--(1)--(2);
    
        \fill[gray,opacity=0.3] (0,2)--(1.5,1)--(0,0)--cycle
                            (0,2)--(2.5,2.5)--(1.5,1)--cycle
                            (2.5,2.5)--(3,0)--(1.5,1)--cycle
                            (3,0)--(0,0)--(1.5,1)--cycle;
        
    \end{tikzpicture}
\end{subfigure}
\hspace{2 cm}
\begin{subfigure}{0.2 \textwidth}
    \begin{tikzpicture}[scale=0.7]
        \node [circle, fill, inner sep=2pt] (1) at (0,2) {};
        \node [left] at (0,2.5) {\footnotesize 1};
        \node [circle, fill, inner sep=2pt] (2) at (2.5,2.5) {};
        \node [right] at (2.5,3) {\footnotesize 2};
        \node [circle, fill, inner sep=2pt] (3) at (0,0) {};
        \node [left] at (0,-.5) {\footnotesize 3};
        \node [circle, fill, inner sep=2pt] (23) at (3,0) {};
        \node [right] at (3,-.5) {\footnotesize 23};
    
        \draw (1)--(2)--(23)--(3)--(1);
        
    \end{tikzpicture}
\end{subfigure}
\end{figure}
\end{exa}

The nested set complex $\mathcal{N}(\L,\G)$ can be viewed as a generalization of the order complex $\Delta(\L-\{\widehat{0}\})$, which consists of all chains in $\L-\{\hat{0}\}$. Indeed, taking the maximal building set $\L-\{\hat{0}\}$ in $\L$, the nested sets are chains in $\L-\{\hat{0}\}$. Therefore, the nested set complex $\mathcal{N}(\L,\L_{>\hat{0}})$ coincides with the order complex $\Delta(\L-\{\widehat{0}\})$ and reduced nested set complex $\widetilde{\mathcal{N}}(\L,\L_{>\hat{0}})$ coincides with $ \Delta(\L-\{\widehat{0},\widehat{1}\})$.

Inspired by the presentation of the cohomology ring of De~Concini and Procesi's wonderful compactification model of a hyperplane arrangement in \cite{DeCociniProcesiCohomology}, Feichtner and Yuzvinsky introduced a combinatorially defined $\mathbb{Z}$-algebra in terms of an atomic lattice $\L$ and its building set $\G$, which is now referred to as the \emph{Chow ring $D(\L,\G)$ of an atmoic lattice $\L$ with repsect to a builiding set $\G$}. 

\begin{defi}[\cite{FeichtnerYuzvinsky2004}]\label{ChowRingLattice}
Let $\L$ be an atomic lattice with the set of atoms $\mathfrak{A}(L)$ and a building set $\G$ of $\L$. Then the \emph{Chow ring of $\L$ with respect to $\G$} is the $\mathbb{Z}$-algebra
\[
	D(\L,\G)\coloneqq \mathbb{Z}[x_{G}]_{G\in\G}/(I+J)
\] 
where 
	$I=\left\langle\prod_{i=1}^t x_{G_i}: \{G_1,\ldots,G_t\}\notin\mathcal{N}(\L,\G)\right\rangle$  and $J=\left\langle\sum_{G\ge H}x_G: H\in\mathfrak{A}(\L)\right\rangle$.

\end{defi}
Note that $\mathbb{Z}[x_{G}]_{G\in\G}/I$ is the \emph{Stanley-Reisner ring} associated with the nested set complex $\mathcal{N}(\L,\G)$. 
For details about the Stanley-Reisner ring associated with a simplicial complex, see Stanley \cite{Stanley2007comutative}, Miller and Sturmfels \cite{SturmfelsCommtative2005}. 
\begin{rem}\label{rem:FYChowRing}
Moreover, 
Feichtner and Yuzvinsky \cite{FeichtnerYuzvinsky2004} showed that $D(\L,\G)$ is also the Chow ring of the toric variety associated with a unimodular polyhedral fan $\Sigma(\L,\G)$ defined via the structure of $\N(\L,\G)$. For the definition of the Chow ring of the toric variety defined by a polyhedral fan, see Cox, Little, Schenck \cite[chapter 12]{Cox2011toric}.
\end{rem}

Feichtner and Yuzvinsky also found a monic Gr\"{o}bner basis for the ideal $I+J$. As a result, the algebra $D(\L,\G)$ is a free $\mathbb{Z}$-module of finite rank coming with the following $\mathbb{Z}$-basis. Notice that for simplicity, the $\mathbb{Z}$-basis we present here is for the special case that $\L$ is a geometric lattice. A general statement for any atomic lattice can be found in \cite{FeichtnerYuzvinsky2004}.

\begin{thm}[\cite{FeichtnerYuzvinsky2004}]\label{FYbasis} Let $\L$ be a geometric lattice. Then the following monomials form a basis for $D(\L,\G)$
\[
	\left\{\prod_{G\in N}x_G^{a_G}: N\text{ is nested, }a_G<\rk(G)-\rk(G')\right\}
\]
where $G'$ is the join of elements in $N\cap\L_{<G}$.
\end{thm}

The special case that $\L$ is the lattice of flats of a matroid $M$ and $\G$ is the maximal building set was first considered by Adiprasito, Huh, and Katz \cite{AHK2018}. They refered to $D(\L(M),\L(M)-\{\emptyset\})$ as the \emph{Chow ring of the matroid $M$}.

\section{Eulerian story: permutahedra $\Pi_n$}
\label{Sec:EulerianStory}

\subsection{The $\S_n$-module structure on $H^*(X_{\Sigma_n})$} \label{Sec:RepCohPerm}
Given a $d$-dimensional simple polytope $P$ with normal fan $\Sigma(P)$, the \emph{f-polynomial} of $P$, $f_P(t)=\sum_{i=0}^d f_i(P)t^i$, records the number of $i$-dimensional faces $f_i(P)$ of $P$. Under this setting, the Hilbert series of the cohomology of the toric variety $X_{\Sigma(P)}$ can be computed combinatorially and agrees with the \emph{h-polynomial} $h_P(t)\coloneqq f_P(t-1)$ of $P$ (see e.g. \cite[Sec. 12.5]{Cox2011toric} ). 
It is well-known that the $h$-polynomial of the permutahedron $\Pi_n$ is the Eulerian polynomial $A_n(t)$ (see e.g. \cite[eq. (4.2)]{PRW2008}), hence one has
\[
    A_n(t)=h_{\Pi_n}(t)=\sum_{j=0}^{n-1}\dim H^{2j}(X_{\Sigma_n})t^j.
\]
The cohomology $H^*(X_{\Sigma_n})$ carries an $\mathfrak{S}_n$-representation induced by the $\S_n$-action on $\Sigma_n$. Using a recurrence of Procesi \cite{Procesi1990}, Stanley \cite{Stanley1989} computed the Frobenius series of this representation:
\begin{equation} \label{eq:FrobPerm}
    \sum_{n\ge 0}\sum_{j=0}^{n-1}\ch\left(H^{2j}(X_{\Sigma_n})\right) t^jz^n=\frac{(1-t)H(z)}{H(tz)-tH(z)},
\end{equation}
where $\ch$ is the \emph{Frobenius characteristic map}, $H(z)=\sum_{n\ge 0}h_n(\x)z^n$ and $h_n(\x)$ is the complete homoegeneous symmetric function of degree $n$. It follows from (\ref{eq:FrobPerm}) that the representation on $H^*(X_{\Sigma_n})$ is a permutation representation. 

Stembridge \cite{Stembridge1992} constructed this permutation representation abstractly by introducing a combinatorial object that we call a \emph{Stembridge code}. A \emph{Stembridge code} is a finite sequence $\alpha$ over $\{0,1,2,\ldots\}$ with marks such that if $m(\alpha)$ is the maximum number appearing in $\alpha$ then for each $k=1,2,\ldots, m(\alpha)$
\begin{itemize}
\item $k$ occurs at least twice in $\alpha$;
\item a mark is assigned to the $i$th occurence of $k$ for a unique $i\ge 2$.
\end{itemize}
Note that $0$ can occur any number of times in $\alpha$. 
For all $1\le k\le m(\alpha)$, let $f(k)=i-1$ the number of occurrences of $k$ in $\alpha$ to the left of the marked $k$. Let $(\alpha,f)$ denote a Stembridge code, then the \emph{index} of $(\alpha,f)$ is 
\[
    \ind(\alpha,f)\coloneqq\sum_{k=1}^{m(\alpha)}f(k).
\]
The \emph{content} of $(\alpha,f)$, denoted by $\lambda(\alpha,f)$, is the partition obtained by reordering in decreasing order the weak composition $(\mu_0,\mu_1,\ldots,\mu_{m(\alpha)})$, where $\mu_i$ equals the number of occurrences of $i$ in $\alpha$, and then removing all zeros. 
Let $\Code_n=\bigcup_{j=0}^{n-1}\Code_{n,j}$ where $\Code_{n,j}$ is the set of Stembridge codes of length $n$ with index $j$.
\begin{exa}
A Stembridge code $(\alpha,f)=11320\hat{2}\hat{3}\hat{1}2$ consists of $\alpha=113202312$ with $f(1)=2$, $f(2)=1$, $f(3)=1$ and $\ind(\alpha,f)=4$. The content of $(\alpha,f)$ is obtained from reordering $(1,3,3,2)$, hence $\lambda(\alpha,f)=(3,3,2,1)$. On the other hand, if $(\beta,g)=1132\hat{2}\hat{3}\hat{1}2$, then $\ind(\beta,g)=4$, $\lambda(\beta,g)=(3,3,2)$.
\noindent There are $6$ codes in $\Code_3$: 
\[
\begin{array}{c|cccccc}
  (\alpha,f)    & 000 & 01\hat{1}&~10\hat{1}&~1\hat{1}0& 1\hat{1}1 & 11\hat{1}\\
  \hline
 \ind(\alpha,f) &   0  &    1     &     1    &    1     &    1    &  2                
\end{array}
\]
\end{exa}
For $\sigma\in\mathfrak{S}_n$, define $\sigma\cdot (\alpha_1\alpha_2,\ldots\alpha_n, f)=(\alpha_{\sigma(1)}\alpha_{\sigma(2)}\ldots\alpha_{\sigma(n)},f)$.
This action induces a graded representation of $\mathfrak{S}_n$ on the linear span $\mathbb{C}\Code_n=\bigoplus_{j=0}^{n-1}\mathbb{C}\Code_{n,j}$. Stembridge showed that its graded Frobenius series is equal to the right-hand side of (\ref{eq:FrobPerm}). Therefore
\begin{equation}\label{CodeCohoQ}
	\mathbb{C}\Code_{n,j}\cong_{\S_n} H^{2j}\left(X_{\Sigma_n},\mathbb{C}\right) \text{ for all }0\le j\le n-1.
\end{equation}
The following problem is proposed by Stembridge in \cite[p.317]{Stembridge1992} and \cite[p.296, Problem 11.2]{Stembridge1994}.
\begin{prob}
Find a geometric explanation of $H^*(X_{\Sigma_n})$ being a permutation representation, or more explicitly, find a basis of $H^*(X_{\Sigma_n})$ permuted by $\S_n$ that induces the representation we are looking at.    
\end{prob}
Although there is no obvious direct connection between Stembridge codes and $H^*(X_{\Sigma_n})$, it is natural that we expect such basis to have similar combinatorial structure as Stembridge codes. We will present such a basis in section \ref{CodeBasis}.

\subsection{Stembridge codes and FY-basis for Chow ring of $\mathsf{B}_n$} \label{CodeBasis}

In this section, we shall answer Stembridge's question. We first observe that $H^*(X_{\Sigma_n})$ can be identified with the Chow ring of the Boolean matroid, and then show that the Feichtner-Yuzvinsky basis for the Chow ring introduced in section \ref{sec:BS} is the explicit permutation basis that Stembridge wanted. Moreover, we construct an $\S_n$-equivariant bijection between this basis and Stembridge codes.

The following lemma from \cite[p.251(1.5)]{Stembridge1994} allows us to identify  $H^*(X_{\Sigma_n})$ with a quotient of the Stanley-Reisner ring of the boundary complex of the dual permutahedron $\Pi_n^*$. 
Let $P$ be an $n$-dimensional polytope in a Euclidean space $V\cong\mathbb{R}^n$. 
Recall that its dual polytope is denoted by $P^*$ and its normal fan is denoted by $\Sigma(P)$. 
Let $K[\partial P^*]$ be the Stanley-Reisner ring of the boundary complex of $P^*$ over a field $K$ of characteristic $0$. 

\begin{lem}[\cite{danilov1978},\cite{Stembridge1994}]\label{CohomologyFaceRing}
Let $P$ be a $n$-dimensional simple lattice polytope in a Euclidean space $V\cong\mathbb{R}^n$. 
If a finite group $G$ acts on $\Sigma(P)$ simplicially and freely, then the cohomology of the toric variety associated with the fan $\Sigma(P)$ is
\[
	H^*(X_{\Sigma(P)},K)\cong K[\partial P^*]/\langle\theta_1,\ldots,\theta_n\rangle ~~\text{ as $K[G]$-modules},
\]
where $\theta_i=\sum_{v\in V(P^*)}\langle v,e_i \rangle x_v$ for $i=1,\ldots,n$, the vectors $e_1,\ldots,e_n$ are standard basis vectors in $\mathbb{Z}^n\subset V$, and $V(P^*)$ is the set of vertices of $P^*$.
\end{lem}

\begin{rem}
The quotient ring on the right-hand-side of the isomorphism is also known as the \emph{Chow ring $A(X_{\Sigma(P)})$ of the toric variety $X_{\Sigma(P)}$}.
\end{rem}
Let $M$ be a loopless matroid on ground set $[n]$ whose lattice of flats is $\L(M)$. For each $S\subseteq [n]$, write $e_S=\sum_{i\in S}e_i$. The \emph{Bergman fan} $\Sigma_M$ of $M$ is a fan in  $\mathbb{R}^n/\langle e_{[n]}\rangle$ consisting of cones $\sigma_{\F}$ indexed by all flags $\F=\{F_1\subsetneq\ldots\subsetneq F_k\}$ in $\L(M)-\{\emptyset,[n]\}$, where
\[
	\sigma_{\F}=\mathbb{R}_{\ge 0}\{e_{F_1},\ldots,e_{F_k}\}.
\]
The \emph{Bergman complex} $\Delta_{\Sigma_M}$ of $M$ is the simplicial complex obtained by intersecting $\Sigma_M$ with the unit sphere centered at $0$. Obviously, the Bergman complex is a geometric realization of the order complex $\Delta(\L(M)-\{\emptyset,[n]\})$.
 
The \emph{Chow ring of $M$} encodes information from the Bergman fan $\Sigma_M$ and is given in two different presentations as the following,
\begin{align}
    A(M) &\coloneqq\frac{\mathbb{Q}[x_F]_{F\in\mathcal{\L}(M)-\{\emptyset\}}/\left\langle x_F x_G: F,G\text{ are incomparable in }\mathcal{L}(M)\right\rangle}{\left\langle\sum_{F: i\in F}x_F: 1\le i \le n \right\rangle} \label{FYPresentation}\\
    \nonumber\\
    	&=\frac{\mathbb{Q}[x_F]_{F\in\mathcal{\L}(M)-\{\emptyset,[n]\}}/\left\langle x_F x_G: F,G\text{ are incomparable in }\mathcal{L}(M)\right\rangle}{\left\langle\sum_{F: i\in F}x_F-\sum_{F: j\in F}x_F: i\neq j \right\rangle}. \label{HuhPresentation}
\end{align}
The presentation (\ref{FYPresentation}) is a special case of Definition \ref{ChowRingLattice} given by Feichtner and Yuzvnisky \cite{FeichtnerYuzvinsky2004}. The presentation (\ref{HuhPresentation}) is obtained from (\ref{FYPresentation}) by eliminating $x_E$ and was first considered by Adiprasito, Huh, and Katz in \cite{AHK2018}. Note that the numerator in (\ref{HuhPresentation}) is the Stanley-Reisner ring of $\Sigma_M$, or equivalently of the Bergman complex $\Delta_{\Sigma_M}$. 

Now let $M$ be the Boolean matroid $\mathsf{B}_n$; the flats of $\mathsf{B}_n$ are all subsets of $[n]$ and $\L(\mathsf{B}_n)$ is the Boolean lattice. It is well-known that the order complex $\Delta(\L(\mathsf{B}_n)-\{\emptyset,[n]\})$ is the barycentric subdivision of the boundary of the $(n-1)$-simplex, which is exactly the boundary of the dual permutahedron  $\partial\Pi_n^*$.

\begin{exa}
The flats of $\mathsf{B}_3$ consists of all subsets of $[3]$ and the lattice of flats is $\L(\mathsf{B}_3)=(2^{[3]},\subseteq)$. The Bergman fan and the Bergman complex of $\mathsf{B}_3$  are the following: 


\begin{figure}[h]

\begin{subfigure}{0.3 \textwidth}
\tikzset{every picture/.style={line width=0.75pt}} 
\begin{tikzpicture}[x=0.6pt,y=0.6pt,yscale=-1,xscale=1]

\draw   (122.63,154.48) .. controls (122.49,152.44) and (124.04,150.69) .. (126.07,150.55) .. controls (128.11,150.42) and (129.87,151.96) .. (130,154) .. controls (130.13,156.04) and (128.59,157.79) .. (126.55,157.93) .. controls (124.52,158.06) and (122.76,156.52) .. (122.63,154.48) -- cycle ;
\draw    (129.8,156) -- (199.14,202.88) ;
\draw [shift={(200.8,204)}, rotate = 214.06] [color={rgb, 255:red, 0; green, 0; blue, 0 }  ][line width=0.75]    (10.93,-3.29) .. controls (6.95,-1.4) and (3.31,-0.3) .. (0,0) .. controls (3.31,0.3) and (6.95,1.4) .. (10.93,3.29)   ;
\draw    (122.63,155.48) -- (53.05,203.46) ;
\draw [shift={(51.4,204.6)}, rotate = 325.41] [color={rgb, 255:red, 0; green, 0; blue, 0 }  ][line width=0.75]    (10.93,-3.29) .. controls (6.95,-1.4) and (3.31,-0.3) .. (0,0) .. controls (3.31,0.3) and (6.95,1.4) .. (10.93,3.29)   ;
\draw    (126.07,150.55) -- (125.98,66.05) ;
\draw [shift={(125.97,64.05)}, rotate = 89.93] [color={rgb, 255:red, 0; green, 0; blue, 0 }  ][line width=0.75]    (10.93,-3.29) .. controls (6.95,-1.4) and (3.31,-0.3) .. (0,0) .. controls (3.31,0.3) and (6.95,1.4) .. (10.93,3.29)   ;
\draw    (122.63,151.48) -- (51.81,106.06) ;
\draw [shift={(50.13,104.98)}, rotate = 32.68] [color={rgb, 255:red, 0; green, 0; blue, 0 }  ][line width=0.75]    (10.93,-3.29) .. controls (6.95,-1.4) and (3.31,-0.3) .. (0,0) .. controls (3.31,0.3) and (6.95,1.4) .. (10.93,3.29)   ;
\draw    (129,152) -- (200.82,106.08) ;
\draw [shift={(202.5,105)}, rotate = 147.4] [color={rgb, 255:red, 0; green, 0; blue, 0 }  ][line width=0.75]    (10.93,-3.29) .. controls (6.95,-1.4) and (3.31,-0.3) .. (0,0) .. controls (3.31,0.3) and (6.95,1.4) .. (10.93,3.29)   ;
\draw    (126.55,157.93) -- (126.14,244.8) ;
\draw [shift={(126.13,246.8)}, rotate = 270.27] [color={rgb, 255:red, 0; green, 0; blue, 0 }  ][line width=0.75]    (10.93,-3.29) .. controls (6.95,-1.4) and (3.31,-0.3) .. (0,0) .. controls (3.31,0.3) and (6.95,1.4) .. (10.93,3.29)   ;


\draw (34,202) node [anchor=north west][inner sep=0.75pt]   [align=left] {$\displaystyle e_{1}$};
\draw (202,202) node [anchor=north west][inner sep=0.75pt]   [align=left] {$\displaystyle e_{2}$};
\draw (119,46) node [anchor=north west][inner sep=0.75pt]   [align=left] {$\displaystyle e_{3}$};
\draw (118,245) node [anchor=north west][inner sep=0.75pt]   [align=left] {$\displaystyle e_{12}$};
\draw (206,95) node [anchor=north west][inner sep=0.75pt]   [align=left] {$\displaystyle e_{23}$};
\draw (26,95) node [anchor=north west][inner sep=0.75pt]   [align=left] {$\displaystyle e_{13}$};
\draw (140,145) node [anchor=north west][inner sep=0.75pt]   [align=left] {$\displaystyle e_{123}$};

\end{tikzpicture}
\end{subfigure}
\hspace{3 cm}
\begin{subfigure}{0.2 \textwidth}
\tikzset{every picture/.style={line width=0.75pt}} 

\begin{tikzpicture}[x=0.6pt,y=0.6pt,yscale=-1,xscale=1]

\draw   (86.58,28.23) .. controls (86.44,26.19) and (87.99,24.44) .. (90.02,24.3) .. controls (92.06,24.17) and (93.82,25.71) .. (93.95,27.75) .. controls (94.08,29.79) and (92.54,31.54) .. (90.5,31.68) .. controls (88.47,31.81) and (86.71,30.27) .. (86.58,28.23) -- cycle ;
\draw   (146.3,58.43) .. controls (146.17,56.39) and (147.71,54.63) .. (149.75,54.5) .. controls (151.79,54.37) and (153.54,55.91) .. (153.68,57.95) .. controls (153.81,59.98) and (152.27,61.74) .. (150.23,61.87) .. controls (148.19,62.01) and (146.44,60.46) .. (146.3,58.43) -- cycle ;
\draw   (86.3,158.43) .. controls (86.17,156.39) and (87.71,154.63) .. (89.75,154.5) .. controls (91.79,154.37) and (93.54,155.91) .. (93.68,157.95) .. controls (93.81,159.98) and (92.27,161.74) .. (90.23,161.87) .. controls (88.19,162.01) and (86.44,160.46) .. (86.3,158.43) -- cycle ;
\draw   (26.3,128.43) .. controls (26.17,126.39) and (27.71,124.63) .. (29.75,124.5) .. controls (31.79,124.37) and (33.54,125.91) .. (33.68,127.95) .. controls (33.81,129.98) and (32.27,131.74) .. (30.23,131.87) .. controls (28.19,132.01) and (26.44,130.46) .. (26.3,128.43) -- cycle ;
\draw   (26.8,58.93) .. controls (26.67,56.89) and (28.21,55.13) .. (30.25,55) .. controls (32.29,54.87) and (34.04,56.41) .. (34.18,58.45) .. controls (34.31,60.48) and (32.77,62.24) .. (30.73,62.37) .. controls (28.69,62.51) and (26.94,60.96) .. (26.8,58.93) -- cycle ;
\draw   (146.3,128.43) .. controls (146.17,126.39) and (147.71,124.63) .. (149.75,124.5) .. controls (151.79,124.37) and (153.54,125.91) .. (153.68,127.95) .. controls (153.81,129.98) and (152.27,131.74) .. (150.23,131.87) .. controls (148.19,132.01) and (146.44,130.46) .. (146.3,128.43) -- cycle ;
\draw    (93.5,29) -- (146.6,55.9) ;
\draw    (33.2,130.53) -- (86.3,157.43) ;
\draw    (33.6,56.9) -- (86.58,29.23) ;
\draw    (93.68,157.95) -- (146.65,130.28) ;
\draw    (29.75,124.5) -- (30.73,62.37) ;
\draw    (149.75,124.5) -- (150.73,62.37) ;

\draw (12.5,125) node [anchor=north west][inner sep=0.75pt]   [align=left] {$\displaystyle e_{1}$};
\draw (80.5,160.5) node [anchor=north west][inner sep=0.75pt]   [align=left] {$\displaystyle e_{12}$};
\draw (152.5,124) node [anchor=north west][inner sep=0.75pt]   [align=left] {$\displaystyle e_{2}$};
\draw (156,42.5) node [anchor=north west][inner sep=0.75pt]   [align=left] {$\displaystyle e_{23}$};
\draw (82,4.5) node [anchor=north west][inner sep=0.75pt]   [align=left] {$\displaystyle e_{3}$};
\draw (6.5,40.5) node [anchor=north west][inner sep=0.75pt]   [align=left] {$\displaystyle e_{13}$};

\end{tikzpicture}
\end{subfigure}

\end{figure}

\end{exa}

Since each cone in $\Sigma_{\mathsf{B}_n}=\Sigma_n$ corresponds to a chain in $\L(\mathsf{B}_n-\{\emptyset,[n]\})$, the $\mathfrak{S}_n$-action on $\Sigma_n$ coincides with the action on the chains which is induced from the 
$\S_n$-action on subsets of $[n]$. It is clear that this action acts on $\Sigma_n$ (and the chains) simplicially and freely. Applying Lemma \ref{CohomologyFaceRing}, we see that  
\[
	\theta_i=\sum_{F\in\L(M)-\{\emptyset\}}\langle e_F,e_i-e_{i+1}\rangle x_F=\sum_{i\in F}x_F-\sum_{i+1\in F}x_F.
\]
Obviously $\theta_1,\ldots,\theta_{n-1}$ generate the ideal in the denominator of (\ref{HuhPresentation}). Therefore, we have
\[
    A(\mathsf{B}_n)=A(X_{\Sigma_n})\coloneqq\frac{\mathbb{Q}[\partial\Pi_n^*]}{\langle\theta_1,\ldots,\theta_{n-1}\rangle}\cong_{\mathfrak{S}_n} H^*(X_{\Sigma_n},\mathbb{Q}).
\]
Applying Proposition \ref{FYbasis} to (\ref{FYPresentation}), the Feichtner--Yuzvinsky basis of $A(M)$ is given by
\[
    FY(M)=\left\{x_{F_1}^{a_1}x_{F_2}^{a_2}\ldots x_{F_\ell}^{a_\ell}:\substack{~\emptyset=F_0\subsetneq F_1\subsetneq F_2\subsetneq\ldots\subsetneq F_\ell\subseteq \L(M) \text{ for } 0\le \ell\le n \\
    1\le a_i\le \rk_M(F_i)-\rk_M(F_{i-1})-1}\right\}.
\]
In particular, when $M$ is the Boolean matroid $\mathsf{B}_n$, the basis is
\[
    FY(\mathsf{B}_n)=\left\{x_{F_1}^{a_1}x_{F_2}^{a_2}\ldots x_{F_\ell}^{a_\ell}:\substack{~\emptyset=F_0\subsetneq F_1\subsetneq F_2\subsetneq\ldots\subsetneq F_\ell\subseteq [n] \text{ for } 0\le \ell\le n \\
    1\le a_i\le |F_i|-|F_{i-1}|-1}\right\}.    
\]
Note that $|F_i|-|F_{i-1}|\ge 2$ for all $1\le i\le n$. Let $FY^k(\mathsf{B}_n)$ be the subset of  $FY(\mathsf{B}_n)$ consisting of degree $k$ monomials, for $0\le k\le n-1$. This forms a basis for the homogeneous component $A^k(\mathsf{B}_n)$ of $A(\mathsf{B}_n)$. Clearly, $FY(\mathsf{B}_n)=\uplus_{k=0}^{n-1}FY^k(\mathsf{B}_n)$.

\begin{exa} The FY-basis of $A^k(\mathsf{B}_4)$:\\
$FY^0(\mathsf{B}_4)=\{  1 \}$\\
$FY^1(\mathsf{B}_4)=\{x_{12}, x_{13}, x_{14}, x_{23}, x_{24}, x_{34}, x_{123}, x_{124}, x_{134}, x_{234}, x_{1234}\}$\\
$FY^2(\mathsf{B}_4)=\{ x_{12}x_{1234}, x_{13}x_{1234}, x_{14}x_{1234}, x_{23}x_{1234}, x_{24}x_{1234}, x_{34}x_{1234}, x_{123}^2, x_{124}^2, x_{134}^2, x_{234}^2, x_{1234}^2\}$\\
$FY^3(\mathsf{B}_4)=\{x_{1234}^3\}$    
\end{exa}
Observe that the $\mathfrak{S}_n$-action on $A(X_{\Sigma_n})$, when restricted to $FY(\mathsf{B}_n)$, permutes the indices and preserves the degree of the monomials in $FY(\mathsf{B}_n)$.
\begin{prop}\label{prop:FY(B_n)Perm}
The $\S_n$-module $A(X_{\Sigma_n})=A(\mathsf{B}_n)$ and its homogeneous components $A^k(\mathsf{B}_n)$ have permutation bases $FY(\mathsf{B}_n)$ and $FY^k(\mathsf{B}_n)$, respectively. 
\end{prop}

We now provide a natural $\S_n$-equivariant bijection between Stembridge codes of length $n$ and $FY(\mathsf{B}_n)$. We don't know of any other basis for $H^*(X_{\Sigma_n})$ that bears such a natural and explicit resemblance to Stembridge codes. See \cite{ChoHongLee2020} and \cite{lin2019around} for other bases for $H^*(X_{\Sigma_n})$.

\begin{thm}\label{thm:bijectionBasis1}
For $u=x_{F_1}^{a_1}\ldots x_{F_\ell}^{a_\ell}\in FY(\mathsf{B}_n)$, define $\phi(u)=(\alpha_1\ldots\alpha_n,f)$ by
\[
    \alpha_i=\begin{cases} j & \text{ if }i\in F_j-F_{j-1}\\
0 & \text{ if }i\in [n]-F_\ell
\end{cases} \quad\mbox{ for }i\in [n]
\]
and
\[
    f(j)=a_j \quad\text{ for }j\in [\ell]. 
\]
Then $\phi:FY(\mathsf{B}_n)\rightarrow\Code_n$ is a bijection that satsifies $\deg(u)=\ind(\phi(u))$ and respects the $\S_n$-action on both sets.
\end{thm}

\begin{exa}
The map $\phi$ sends $x_{13}x_{1235}x_{1234568}^{2}\in FY(\mathsf{B}_8)$  to a code as follows:
\begin{align*}
    1\underline{~~}\hat{1}~\underline{~~}~\underline{~~}~\underline{~~}~\underline{~~}~\underline{~~}  \longrightarrow    12\hat{1}~\underline{~~}\hat{2}\underline{~~}~ \underline{~~} ~\underline{~~} \rightarrow 12\hat{1}3\hat{2}3\underline{~~}\hat{3} \rightarrow  12\hat{1}3\hat{2}30\hat{3}.
\end{align*}
The basis for $A^2(B_4)$ and the corresponding codes in $\Code_{4,2}$:\\
\begin{center}
\begin{tikzcd}[ ampersand replacement=\&, column sep=tiny]
x_{12}x_{1234} \& x_{13}x_{1234} \& x_{14}x_{1234} \& x_{23}x_{1234} \& x_{24}x_{1234} \& x_{34}x_{1234}\\
1\hat{1}2\hat{2} \arrow[u,leftrightarrow] \& 12\hat{1}\hat{2} \arrow[u,leftrightarrow] \& 12\hat{2}\hat{1} \arrow[u,leftrightarrow] \& 21\hat{1}\hat{2} \arrow[u,leftrightarrow] \& 21\hat{2}\hat{1} \arrow[u,leftrightarrow] \& 2\hat{2}1\hat{1} \arrow[u,leftrightarrow]
\end{tikzcd}

\begin{tikzcd}[ ampersand replacement=\&, column sep=tiny]
x_{123}^2 \& x_{124}^2 \& x_{134}^2 \&  x_{234}^2 \& x_{1234}^2 \\
11\hat{1}0 \arrow[u,leftrightarrow] \& 110\hat{1} \arrow[u,leftrightarrow] \& 101\hat{1} \arrow[u,leftrightarrow] \& 011\hat{1} \arrow[u,leftrightarrow] \& 11\hat{1}1 \arrow[u,leftrightarrow]
\end{tikzcd}
\end{center}
It is easy to see that the bijection respects the $\mathfrak{S}_4$-action on both sets. 
\end{exa}

\begin{proof}[Proof of Theorem \ref{thm:bijectionBasis1}]
Let $x_{F_1}^{a_1}x_{F_2}^{a_2}\ldots x_{F_\ell}^{a_\ell}$ be a monomial in  $FY(\mathsf{B}_n)$. If $\ell=0$, the monomial is $1$, and therefore we set $\phi(1)=00\ldots 0$.\\
If $\ell\ge 1$ then $a_i\ge 1$ for $i=1,\ldots,\ell$. Therefore, according to the definition of $\phi$, we can associate a sequence $\alpha=\alpha_1\alpha_2\ldots\alpha_n$ to the monomial by letting $\alpha_{i_1}=1$ for all $i_1\in F_1$, $\alpha_{i_2}=2$ for all $i_2\in F_2\setminus F_1$, \ldots, $\alpha_{i_\ell}=\ell$ for all $i_\ell\in F_{\ell}\setminus F_{\ell-1}$, $\alpha_i=0$ for all $i\in [n]\setminus F_\ell$.
The sequence satisfies $m(\alpha)=\ell$. Since $|F_i|-|F_{i-1}|\ge 2$ for all $i$, each $1\le j\le \ell$ occurs at least twice in $\alpha$.
Since the mark $f(j)$ is given by the exponent $a_j$ for $j\in [\ell]$, we have $1\le f(j)=a_j\le |F_j|-|F_{j-1}|-1$. Therefore $\phi(x_{F_1}^{a_1}x_{F_2}^{a_2}\ldots x_{F_\ell}^{a_\ell})=(\alpha,f)$ forms a code of length $n$.

Conversely, given a Stembridge code $(\alpha,f)$ with $m(\alpha)=\ell$. Say $\alpha=\alpha_1\ldots\alpha_n$. We can view $\alpha$ as a map such that $\alpha(i)=\alpha_i$ and then set $F_1=\alpha^{-1}(\{1\})$, $F_2=\alpha^{-1}(\{1,2\}),\ldots, F_\ell=\alpha^{-1}(\{1,2,\ldots,\ell\})$ so that it forms a flag $F_1\subsetneq F_2\subsetneq\ldots\subsetneq F_\ell$ in $[n]$. This implies $\alpha^{-1}(\{j\})=F_j\setminus F_{j-1}$. Since each $j$ in $[\ell]$ occurs at least twice in $\alpha$, we have $|F_j\setminus F_{j-1}|\ge 2$. Recall that $f(j)$ is the number of occurrences of $j$ in front of the marked $j$, for $1\le j\le \ell$. Therefore, we have 
\[
    1\le f(j)\le |\alpha^{-1}(\{j\})|-1=|F_j|-|F_{j-1}|-1.
\]
Now set the exponent $a_j$ equal to $f(j)$.
Therefore the corresponding monomial $x_{F_1}^{a_1}x_{F_2}^{a_2}\ldots x_{F_\ell}^{a_\ell}$ is in $FY(\mathsf{B}_n)$. Obviously, this procedure is the inverse operation of $\phi$.

To prove $\phi$ respects the $\mathfrak{S}_n$-action, it is sufficient to check that a transposition $(i,j)$ commutes with $\phi$. Let $\phi(x_{F_1}^{a_1}\ldots x_{F_\ell}^{a_\ell})=(\alpha,f)$. Observe that $(i,j)\cdot x_{F_1}^{a_1}\ldots x_{F_\ell}^{a_\ell}=x_{F_1'}^{a_1}\ldots x_{F_\ell'}^{a_\ell}$ where $F_1'\subsetneq\ldots\subsetneq F_\ell'$ is the flag obtained from the original flag by exchanging the $i$ and $j$ in each subset. Then the first occurrence of $i$ and the first occurrence of $j$ in the flag are also exchanged. Thus $\phi(x_{F_1'}^{a_1}\ldots x_{F_\ell'}^{a_\ell})=(\alpha', f)$ where the sequence $\alpha'$ is obtained from $\alpha$ by exchanging $\alpha_i$ and $\alpha_j$. That is $(\alpha',f)=(i,j)(\alpha,f)$. Therefore
\[
    \phi((i,j)\cdot x_{F_1}^{a_1}\ldots x_{F_\ell}^{a_\ell})=(\alpha',f)=(i,j)\cdot(\alpha,f)=(i,j)\cdot\phi(x_{F_1}^{a_1}\ldots x_{F_\ell}^{a_\ell}).
\]
\end{proof}
Combining Theorem \ref{thm:bijectionBasis1} with Stembridge's computation on the Frobenius series of $\mathbb{C}\Code_n$ in \cite[Lemma 3.1]{Stembridge1992}, we reprove Stanley and Procesi's result (\ref{eq:FrobPerm}) on the graded Frobenius series of $H^{2j}(X_{\Sigma_n})$ in a more combinatorial way. The original proof involves some algebraic geometry machinery. 

It might be interesting to compare our basis with the bases in \cite{Lin2022PermVar} and \cite{ChoHongLee2020}.

\section{Binomial Eulerian story: stellohedra $\widetilde{\Pi}_n$} \label{Sec:BinEulerianStory}
\subsection{The $\S_n$-module structure on cohomology of the stellahedral variety}
\subsubsection{Binomial Eulerian polynomials and Stellahedra}

We now switch track to a parallel story. Postnikov, Reiner, and Williams \cite{PRW2008} introduced a simple polytope called the \emph{stellohedron} $\widetilde{\Pi}_n$ as the graph associahedron of the $n$-star graph $K_{1,n}$ (see Section \ref{Sec:GraphAss}). A more concrete way to construct the stellohedron $\widetilde{\Pi}_n$ is to start with the standard simplex $\Delta_n=\conv(0,e_i:i\in[n])$, where $e_i$ is the standard basis vector in $\mathbb{R}^n$, then truncate all faces of $\Delta_n$ not containing $0$ in the inclusion order (This construction follows from Carr and Devadoss \cite{CarrGraphAss}). One can see that it contains the permutahedron $\Pi_n$ as a facet. See figure (\ref{fig:stell_3}) and figure (\ref{fig:dualstell_3}) for an example of the stellohedron and its dual when $n=3$. 
\begin{figure}[h]
    \begin{subfigure}{0.3\textwidth}
        \includegraphics{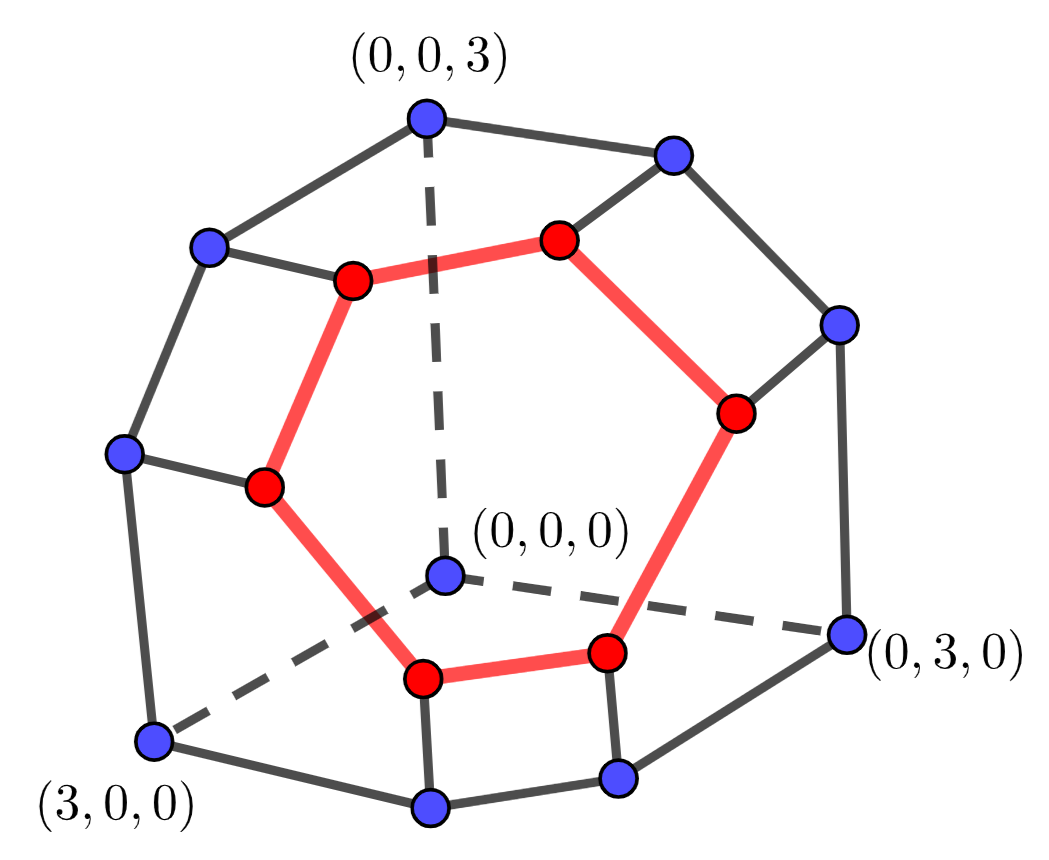}
        \caption{Stellohedron $\widetilde{\Pi}_3$}
        \label{fig:stell_3}
    \end{subfigure}
    \hspace{2 cm}
    \begin{subfigure}{0.3\textwidth}
        \includegraphics{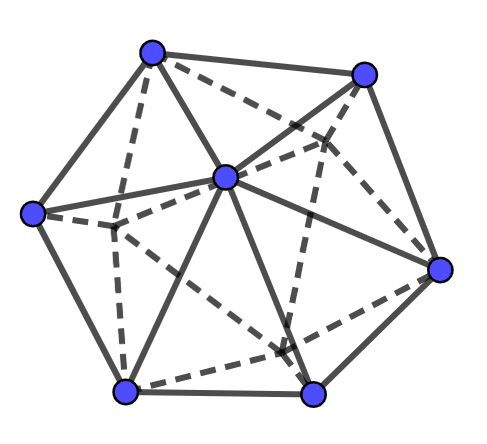}
        \caption{Dual stellohedron $\widetilde{\Pi}_3^*$}
        \label{fig:dualstell_3}
    \end{subfigure}
    \caption{}
\end{figure}

The toric variety defined by the normal fan $\widetilde{\Sigma}_n$ of $\widetilde{\Pi}_n$ is called the \emph{stellahedral variety} $X_{\widetilde{\Sigma}_n}$. 
Postnikov, Reiner, and Williams showed that the Hilbert series of $H^*(X_{\widetilde{\Sigma}_n})$, which agrees with the $h$-polynomial of $\widetilde{\Pi}_n$, is given by
\[ 
    \sum_{j=0}^n\dim H^{2j}(X_{\widetilde{\Sigma}_n})t^j=\widetilde{A}_n(t)\coloneqq 1+t\sum_{k=1}^n\binom{n}{k}A_k(t).
\]
Following Shareshian and Wachs \cite{ShareshianWachs2020}, the polynomial $\widetilde{A}_n(t)$ is now known as the \emph{binomial Eulerian polynomial}.

The cohomology $H^*(X_{\widetilde{\Sigma}_n})$ also carries a representation of $\S_n$ induced by the simiplicial action of $\mathfrak{S}_n$ on $\widetilde{\Sigma}_n$. This representation was studied by Shareshian and Wachs in \cite{ShareshianWachs2020}. 
They introduced a symmetric function analogue of $\widetilde{A}_n(t)$,
\begin{equation} \label{defBinomQ}
	\widetilde{Q}_n(\x,t)\coloneqq h_n(\x)+t\sum_{k=1}^n h_{n-k}(\x)Q_k(\x,t),
\end{equation}
where 
\begin{equation} \label{eq:GeoDefQ}
    Q_k(\x,t)=\sum_{j=0}^{k-1}\ch\left(H^{2j}(X_{\Sigma_n})\right)t^j. 
\end{equation}
Using algebraic geometry machinery similar to that in Procesi's proof of (\ref{eq:FrobPerm}), they showed that the graded Frobenius series of $H^*(X_{\widetilde{\Sigma}_n})$ is exactly $\widetilde{Q}_n(\x,t)$ as stated in the following theorem. Let $\widetilde{Q}_{n,j}(\x)$ be the symmetric function coefficient of $t^j$ of $\widetilde{Q}_n(\x,t)$.  
\begin{thm}[\cite{ShareshianWachs2020}]\label{FrobStell}
For all $n\ge 1$, we have $\sum_{j=0}^{n}\ch\left(H^{2j}(X_{\widetilde{\Sigma}_n})\right)t^j=\widetilde{Q}_n(\x,t)$, or equivalently, $\ch\left(H^{2j}(X_{\widetilde{\Sigma}_n})\right)=\widetilde{Q}_{n,j}(\x)$ for all $0\le j\le n$.
\end{thm} 
From (\ref{defBinomQ}), since $Q_k(\x,t)$ is $h$-positive, so is $\widetilde{Q}_n(\x,t)$. Therefore, the representation of $\S_n$ on $H^*(X_{\widetilde{\Sigma}_n})$ is also a permutation representation.

\subsubsection{Extended codes} \label{Sec:ExtCodes}

We introduce an analog of Stembridge codes which we call \emph{extended codes}.
 
 An \emph{extended code}  $(\alpha,f)$ is a marked sequence like a Stembridge code. The sequence $\alpha$ is over $\{0,1,\ldots\}\cup\{\infty\}$ with $\infty$'s working as $0$'s in Stembridge codes; $m(\alpha)$ and $f$ are defined as in Stembridge codes. We define the \emph{index} of the extended code $(\alpha, f)$ to be 
\[
    \ind(\alpha,f)=\begin{cases}
        -1 & \text{ if }\alpha=\infty\ldots\infty;\\
        \sum_{k=1}^{m(\alpha)}f(k) & \text{ otherwise.}
    \end{cases}
\]
The \emph{content} of $(\alpha, f)$, also denoted by $\lambda(\alpha,f)$, is the partition obtained by reordering in decreasing order the weak composition $(\mu_{\infty},\mu_0,\ldots,\mu_{m(\alpha)})$, where $\mu_i$ is the number of occurrences of $i$ in $\alpha$, and then removing all zeros. 
Let $\widetilde{\Code}_{n,j}$ be the set of extended codes of length $n$ with index $j$ and $\widetilde{\Code}_n=\bigcup_{j=-1}^{n-1}\widetilde{\Code}_{n,j}$.
\begin{exa}
The marked sequence $(\alpha,f)=11\infty23\hat{3}\hat{2}0\hat{1}3\infty$ is an extended code whose index is $\ind(\alpha,f)=f(1)+f(2)+f(3)=2+1+1=4$. The content of $(\alpha,f)$ is obtained by reordering $(2,1,3,2,3)$; hence $\lambda(\alpha,f)=(3,3,2,2,1)$. While $(\beta,g)=1123\hat{3}\hat{2}\hat{1}3$ is an extended code with the same index. 
The content can be obtained from the weak composition $(0,0,3,2,3)$ and is $\lambda(\beta,g)=(3,3,2)$.

The following are all the extended codes of length $3$:\\
$\widetilde{\Code}_{3,-1}=\{\infty\infty\infty\}$,~ 
$\widetilde{\Code}_{3,0}=\{0\infty\infty  ,\infty 0\infty ,\infty \infty 0,\infty 00,0\infty0 ,00\infty ,000\}$,~\\
$\widetilde{\Code}_{3,1}=\{1\hat{1}\infty,1\infty\hat{1},\infty 1\hat{1},01\hat{1},10\hat{1},1\hat{1}0, 1\hat{1}1\}$,~
$\widetilde{\Code}_{3,2}=\{11\hat{1}\}$.
\end{exa}
For $\sigma\in\mathfrak{S}_n$, consider the $\S_n$-action $\sigma\cdot (\alpha_1\alpha_2,\ldots\alpha_n, f)=(\alpha_{\sigma(1)}\alpha_{\sigma(2)}\ldots\alpha_{\sigma(n)},f)$. This induces a graded $\mathfrak{S}_n$-representation on $\mathbb{C}\widetilde{\Code}_n=\bigoplus_{j=0}^{n}\mathbb{C}\widetilde{\Code}_{n,j-1}$. We compute its Frobenius series and obtain a result parallel to Stembrige's result (\ref{CodeCohoQ}), which is equivalent to
\begin{equation} \label{eq:Qcode}
    \sum_{j=0}^{n-1}\ch\left(\mathbb{C}\Code_{n,j}\right)t^j=Q_n(\x,t).
\end{equation}
\begin{thm} \label{thm:FcharExtCodes} 
For $n\ge 1$, we have  $\sum_{j=0}^n\ch(\mathbb{C}\widetilde{\Code}_{n,j-1})t^j =\widetilde{Q}_n(\x,t)$.
\end{thm}
\begin{proof}
Note that the representation of $\S_n$ generated by an extended code $(\alpha,f)$ is isomorphic to the permutation module $M_{\lambda(\alpha,f)}$ whose basis is given by the $\mathfrak{S}_n$-orbit containing $(\alpha,f)$. 
Let $\widetilde{U}_n$ be a set of  representatives from each $\mathfrak{S}_n$-orbit in $\widetilde{\Code}_n$ and $U_n$ be a set of representatives from each $\mathfrak{S}_n$-orbit of $\Code_n$ for all $n$. Then
\begin{align*}
    \sum_{j\ge 0}\ch(\mathbb{C}\widetilde{\Code}_{n,j-1})t^j&=
    \sum_{j\ge 0}\sum_{\substack{(\alpha,f)\in \widetilde{U}_n\\ \ind(\alpha,f)=j-1}}\ch(M_{\lambda(\alpha,f)})t^j\\
    &=\sum_{j\ge 0}\sum_{\substack{(\alpha,f)\in \widetilde{U}_n\\ \ind(\alpha,f)=j-1\\ \lambda(\alpha,f)=(\mu_\infty,\mu_0,\ldots,\mu_\ell)}}
    h_{\mu_{\infty}}h_{\mu_0}\ldots h_{\mu_\ell}t^j\\
    &=h_n+\sum_{j\ge 1}\sum_{\mu_{\infty}=0}^{n-1}h_{\mu_\infty}\sum_{\substack{(\alpha,f)\in U_{n-\mu_{\infty}}\\ \ind(\alpha,f)=j-1}}h_{\lambda(\alpha,f)}t^j\\
    &=h_n+t\sum_{\mu_{\infty}=0}^{n-1}h_{\mu_\infty}\sum_{j
    \ge 1}\sum_{\substack{(\alpha,f)\in U_{n-\mu_{\infty}}\\ \ind(\alpha,f)=j-1}}h_{\lambda(\alpha,f)}t^{j-1}\\
    &=h_n+t\sum_{\mu_{\infty}=0}^{n-1}h_{\mu_{\infty}}Q_{n-\mu_{\infty}}(\x,t)~~(\text{by (\ref{eq:Qcode})}) \\
    &=\widetilde{Q}_n(\x,t) ~~(\text{by }(\ref{defBinomQ})).
\end{align*}
\end{proof}

Combining with Theorem \ref{FrobStell}, we show that $\mathbb{C}\widetilde{\Code}_n\cong_{\S_n} H^*(X_{\widetilde{\Sigma}_n})$ as permutation modules. One can also ask if there is a permutation basis for $H^*(X_{\widetilde{\Sigma}_n})$ which has a similar combinatorial structure as extended codes. We will show such a basis exists in what follows. And it will give us a combinatorial proof of Theorem \ref{FrobStell}.

\subsection{Extended codes and a basis for augmented Chow ring of $\mathsf{B}_n$}
\subsubsection{Augmented Chow ring of $\mathsf{B}_n$ and $H^*(X_{\widetilde{\Sigma}_n})$} \label{Sec:AugChowBn}

Braden, Huh, Matherne, Proudfoot, and Wang \cite{BHMPW2020+,BHMPW2020} recently introduced the \emph{augmented Bergman fan} and the \emph{augmented Chow ring} of a matroid in order to prove the \emph{ Dowling-Wilson Top Heavy Conjecture}. In \cite{BHMPW2020+}, they showed that the augmented Bergman fan of $\mathsf{B}_n$ is the normal fan $\widetilde{\Sigma}_n$ of the stellohedron $\widetilde{\Pi}_n$. Therefore the corresponding spherical complex is the boundary complex of the dual stellohedron $\partial\widetilde{\Pi}_n^*$. Below we use Lemma~ \ref{CohomologyFaceRing} to identify $H^*(X_{\widetilde{\Sigma}_n})$ with the augmented Chow ring of $\mathsf{B}_n$.

Let $M$ be a matroid on $[n]$ with lattice of flats $\L(M)$ and collection of independent subsets $\I(M)$. The \emph{augmented Chow ring of $M$} is defined as 
\begin{equation}\label{AugChow}
 \widetilde{A}(M):=\frac{\mathbb{Q}\left[\{x_F\}_{F\in\L(M)\setminus [n]}\cup\{y_1,y_2,\ldots,y_n\}\right]/(I_1+I_2)}{\langle y_i-\sum_{F:i\notin F}x_F\rangle_{i=1,2,\ldots,n}}
\end{equation}
where $I_1=\left\langle x_F x_G: F,G\text{ are incomparable in }\mathcal{L}(M)\right\rangle$, $I_2=\left\langle y_ix_F: i\notin F\right\rangle$. There is a simplicial fan associated with $\widetilde{A}(M)$ called the \emph{augmented Bergman fan of M}.

Let $I\in\I(M)$ and $\F=(F_1\subsetneq\ldots\subsetneq F_k)$ be a flag in $\L(M)$. We say $I$ is \emph{compatible with} $\F$, denoted by $I\le \F$, if $I\subseteq F_1$. Note that for the empty flag $\emptyset$, we have $I\le \emptyset$ for any $I\in\I(M)$. The \emph{augmented Bergman fan $\widetilde{\Sigma}_M$ of $M$} is a simplicial fan in $\mathbb{R}^n$ consisting of cones $\sigma_{I\le\F}$ indexed by compatible pairs $I\le\F$, where $\F$ is a flag in $\L(M)-\{[n]\}$ and
\[
	\sigma_{I\le\F}=\mathbb{R}_{\ge 0}\left(\{e_i\}_{i\in I}\cup\{-e_{[n]\setminus F}\}_{F\in\mathcal{F}}\right).
\]
More explicitly, the augmented Bergman fan of $M$ has rays $\sigma_{\{i\}\le\emptyset}$ generated by $e_i$ for any independent $\{i\}$ and rays $\sigma_{\emptyset\le\{F\}}$ generated by $-e_{[n]\setminus F}$ for any proper flat $F$. The rays $\{\mathbb{R}_{\ge 0}e_i\}_{i\in I}$ and $\{\mathbb{R}_{\ge 0}(-e_{[n]\setminus F})\}_{F\in\F}$ span a cone in $\Sigma_M$ if and only if $I$ and $\F$ forms a compatible pair $I\le\F$. 
The corresponding simplicial complex is called the \emph{augmented Bergman complex}.

\begin{exa} \label{ExAugBergB2}
The augmented Bergman fan $\widetilde{\Sigma}_{\mathsf{B}_2}$ of the Boolean matroid $\mathsf{B}_2$ in $\mathbb{R}^2$ and the corresponding augmented Bergman complex:

\tikzset{every picture/.style={line width=0.75pt}} 

\begin{center}
\begin{tikzpicture}[x=0.75pt,y=0.75pt,yscale=-1.1,xscale=1.1]

\draw    (130.09,119.65) -- (198.46,120.29) ;
\draw [shift={(200.46,120.31)}, rotate = 180.54] [color={rgb, 255:red, 0; green, 0; blue, 0 }  ][line width=0.75]    (10.93,-3.29) .. controls (6.95,-1.4) and (3.31,-0.3) .. (0,0) .. controls (3.31,0.3) and (6.95,1.4) .. (10.93,3.29)   ;
\draw   (124.84,120.05) .. controls (124.74,118.35) and (125.84,116.88) .. (127.29,116.77) .. controls (128.74,116.66) and (129.99,117.95) .. (130.09,119.65) .. controls (130.18,121.34) and (129.08,122.81) .. (127.63,122.92) .. controls (126.18,123.03) and (124.93,121.74) .. (124.84,120.05) -- cycle ;
\draw    (127.95,117.98) -- (129.2,36.63) ;
\draw [shift={(129.23,34.64)}, rotate = 90.88] [color={rgb, 255:red, 0; green, 0; blue, 0 }  ][line width=0.75]    (10.93,-3.29) .. controls (6.95,-1.4) and (3.31,-0.3) .. (0,0) .. controls (3.31,0.3) and (6.95,1.4) .. (10.93,3.29)   ;
\draw    (124.84,120.05) -- (54.18,120.05) ;
\draw [shift={(52.18,120.05)}, rotate = 360] [color={rgb, 255:red, 0; green, 0; blue, 0 }  ][line width=0.75]    (10.93,-3.29) .. controls (6.95,-1.4) and (3.31,-0.3) .. (0,0) .. controls (3.31,0.3) and (6.95,1.4) .. (10.93,3.29)   ;
\draw    (127.63,122.92) -- (127.11,203.32) ;
\draw [shift={(127.1,205.32)}, rotate = 270.37] [color={rgb, 255:red, 0; green, 0; blue, 0 }  ][line width=0.75]    (10.93,-3.29) .. controls (6.95,-1.4) and (3.31,-0.3) .. (0,0) .. controls (3.31,0.3) and (6.95,1.4) .. (10.93,3.29)   ;
\draw    (125.55,121.71) -- (59.36,198.44) ;
\draw [shift={(58.06,199.95)}, rotate = 310.78] [color={rgb, 255:red, 0; green, 0; blue, 0 }  ][line width=0.75]    (10.93,-3.29) .. controls (6.95,-1.4) and (3.31,-0.3) .. (0,0) .. controls (3.31,0.3) and (6.95,1.4) .. (10.93,3.29)   ;
\draw    (225.39,61.14) .. controls (212.13,95.44) and (171.32,96.33) .. (136.35,114.2) ;
\draw [shift={(134.22,115.31)}, rotate = 331.84] [fill={rgb, 255:red, 0; green, 0; blue, 0 }  ][line width=0.08]  [draw opacity=0] (8.93,-4.29) -- (0,0) -- (8.93,4.29) -- cycle    ;
\draw   (345.05,70.15) .. controls (344.95,68.55) and (346.13,67.17) .. (347.7,67.06) .. controls (349.26,66.96) and (350.61,68.17) .. (350.71,69.78) .. controls (350.81,71.38) and (349.63,72.76) .. (348.06,72.86) .. controls (346.5,72.97) and (345.15,71.75) .. (345.05,70.15) -- cycle ;
\draw   (393.44,119.7) .. controls (393.34,118.1) and (394.53,116.72) .. (396.09,116.62) .. controls (397.65,116.51) and (399,117.73) .. (399.11,119.33) .. controls (399.21,120.93) and (398.02,122.31) .. (396.46,122.42) .. controls (394.89,122.52) and (393.54,121.31) .. (393.44,119.7) -- cycle ;
\draw   (297.42,118.92) .. controls (297.32,117.32) and (298.51,115.93) .. (300.07,115.83) .. controls (301.64,115.73) and (302.99,116.94) .. (303.09,118.54) .. controls (303.19,120.14) and (302,121.52) .. (300.44,121.63) .. controls (298.88,121.73) and (297.53,120.52) .. (297.42,118.92) -- cycle ;
\draw   (298.19,167.68) .. controls (298.09,166.08) and (299.28,164.7) .. (300.84,164.59) .. controls (302.4,164.49) and (303.75,165.7) .. (303.86,167.31) .. controls (303.96,168.91) and (302.77,170.29) .. (301.21,170.39) .. controls (299.64,170.5) and (298.29,169.28) .. (298.19,167.68) -- cycle ;
\draw   (345.82,167.68) .. controls (345.72,166.08) and (346.9,164.7) .. (348.47,164.59) .. controls (350.03,164.49) and (351.38,165.7) .. (351.48,167.31) .. controls (351.58,168.91) and (350.4,170.29) .. (348.83,170.39) .. controls (347.27,170.5) and (345.92,169.28) .. (345.82,167.68) -- cycle ;
\draw    (345.82,71.73) -- (302.32,116.97) ;
\draw    (394.21,120.49) -- (350.71,165.73) ;
\draw    (349.23,71.78) -- (394.55,116.62) ;
\draw    (300.44,121.63) -- (300.84,164.59) ;
\draw    (345.82,167.68) -- (303.86,167.31) ;

\draw (205.12,113.9) node [anchor=north west][inner sep=0.75pt]  [font=\scriptsize] [align=left] {$\displaystyle \sigma _{\{1\} \leq \emptyset }$};
\draw (112.94,17.38) node [anchor=north west][inner sep=0.75pt]  [font=\scriptsize] [align=left] {$\displaystyle \sigma _{\{2\} \leq \emptyset }$};
\draw (103.09,208.74) node [anchor=north west][inner sep=0.75pt]  [font=\scriptsize] [align=left] {$\displaystyle \sigma _{\emptyset \leq \{\{1\}\}}$};
\draw (3.37,113.9) node [anchor=north west][inner sep=0.75pt]  [font=\scriptsize] [align=left] {$\displaystyle \sigma _{\emptyset \leq \{\{2\}\}}$};
\draw (19.33,197.74) node [anchor=north west][inner sep=0.75pt]  [font=\scriptsize] [align=left] {$\displaystyle \sigma _{\emptyset \leq \{\emptyset \}}$};
\draw (150.4,68.39) node [anchor=north west][inner sep=0.75pt]  [font=\scriptsize] [align=left] {$\displaystyle \sigma _{\{1,2\} \leq \emptyset }$};
\draw (213.66,44.05) node [anchor=north west][inner sep=0.75pt]  [font=\scriptsize] [align=left] {$\displaystyle \sigma _{\emptyset \leq \emptyset }$};
\draw (137.13,147.23) node [anchor=north west][inner sep=0.75pt]  [font=\scriptsize] [align=left] {$\displaystyle \sigma _{\{1\} \leq \{\{1\}\}}$};
\draw (60.2,76.56) node [anchor=north west][inner sep=0.75pt]  [font=\scriptsize] [align=left] {$\displaystyle \sigma _{\{2\} \leq \{\{2\}\}}$};
\draw (25.57,143.07) node [anchor=north west][inner sep=0.75pt]  [font=\scriptsize] [align=left] {$\displaystyle \sigma _{\emptyset \leq \{\emptyset ,\{2\}\}}$};
\draw (75.41,179.57) node [anchor=north west][inner sep=0.75pt]  [font=\scriptsize] [align=left] {$\displaystyle \sigma _{\emptyset \leq \{\emptyset ,\{1\}\}}$};
\draw (328.43,51.19) node [anchor=north west][inner sep=0.75pt]  [font=\tiny,xslant=0.02] [align=left] {$\displaystyle \{2\} \leq \emptyset $};
\draw (401.87,112.33) node [anchor=north west][inner sep=0.75pt]  [font=\tiny,xslant=0.02] [align=left] {$\displaystyle \{1\} \leq \emptyset $};
\draw (344.56,171.47) node [anchor=north west][inner sep=0.75pt]  [font=\tiny,xslant=0.02] [align=left] {$\displaystyle \emptyset \leq \{\{1\}\}$};
\draw (256.27,113.33) node [anchor=north west][inner sep=0.75pt]  [font=\tiny,xslant=0.02] [align=left] {$\displaystyle \emptyset \leq \{\{2\}\}$};
\draw (260.95,166.96) node [anchor=north west][inner sep=0.75pt]  [font=\tiny,xslant=0.02] [align=left] {$\displaystyle \emptyset \leq \{\emptyset \}$};
\draw (247.97,139.71) node [anchor=north west][inner sep=0.75pt]  [font=\tiny,xslant=0.02] [align=left] {$\displaystyle \emptyset \leq \{\emptyset ,\{2\}\}$};
\draw (300.81,180.97) node [anchor=north west][inner sep=0.75pt]  [font=\tiny,xslant=0.02] [align=left] {$\displaystyle \emptyset \leq \{\emptyset ,\{1\}\}$};
\draw (365.56,79.51) node [anchor=north west][inner sep=0.75pt]  [font=\tiny,xslant=0.02] [align=left] {$\displaystyle \{1,2\} \leq \emptyset $};
\draw (270,85.36) node [anchor=north west][inner sep=0.75pt]  [font=\tiny,xslant=0.02] [align=left] {$\displaystyle \{2\} \leq \{\{2\}\}$};
\draw (374.41,146.11) node [anchor=north west][inner sep=0.75pt]  [font=\tiny,xslant=0.02] [align=left] {$\displaystyle \{1\} \leq \{\{1\}\}$};
\end{tikzpicture}
\end{center}
One can show $\widetilde{\Sigma}_{\mathsf{B}_2}$ is the normal fan of $\widetilde{\Pi}_2$. Thus the augmented Bergman complex is $\partial\widetilde{\Pi}_2^*$.
\end{exa}
Note that the numerator in (\ref{AugChow}) is the Stanley-Reisner ring of $\widetilde{\Sigma}_M$, since for each $i\in [n]$, the ray $\sigma_{\{i\}\le\emptyset}$ corresponds to $y_i$; and for each $F\in\L(M)-\{[n]\}$, the ray $\sigma_{\emptyset\le\{F\}}$ corresponds to $x_F$.

Now consider the case that $M$ is $\mathsf{B}_n$. Recall that by \cite{BHMPW2020+} the augmented Bergman fan $\widetilde{\Sigma}_{\mathsf{B}_n}$ is the normal fan $\widetilde{\Sigma}_n$ of the stellohedron. The $\S_n$-action on $\widetilde{\Sigma}_{\mathsf{B}_n}$ is induced from the $\S_n$-action on the standard basis $\{e_1,\ldots,e_n\}$ and coincides with the action of $\S_n$ on the variables of $\widetilde{A}(\mathsf{B}_n)$. It is clear that this action of $\S_n$ acts on $\widetilde{\Sigma}_n$ simplicially and freely. To apply Lemma \ref{CohomologyFaceRing} on $\widetilde{\Sigma}_n$, observe that for $1\le i \le n$, 
\[
    \widetilde{\theta_i}
    =\sum_{j=1}^n\langle e_j,e_i\rangle y_j+\sum_{F\in\L(M)-\{[n]\}}\langle -e_{[n]\setminus F},e_i\rangle x_F
    =y_i-\sum_{F:i\notin F}x_F.
\]
Thus $\widetilde{\theta_1}, \ldots, \widetilde{\theta_n}$ are exactly the linear elements that generate the ideal in the denominator of (\ref{AugChow}), and the numerator of (\ref{AugChow}) is the Stanley-Reisner ring of $\partial\widetilde{\Pi}_n^*$. Therefore we have 
\[
	\widetilde{A}(\mathsf{B}_n)=A(X_{\widetilde{\Sigma}_n})\coloneqq\frac{\mathbb{Q}[\partial\widetilde{\Pi}_n^*]}{\langle\widetilde{\theta_1},\ldots, \widetilde{\theta_n} \rangle}\cong_{\S_n} H^*(X_{\widetilde{\Sigma}_n},\mathbb{Q}).
\]
However, unlike for Chow ring $A(M)$, there was no analogue of Feichtner--Yuzvinsky's result known for the augmented Chow ring $\widetilde{A}(M)$. Another main result in this paper is that we obtain a Feichtner-Yuzvinsky type basis for the augmented Chow ring of any matroid $M$. Next, we present this basis in the case of Boolean matroid $\mathsf{B}_n$ and show that it is the permutation basis we want. Moreover, we construct an $\S_n$-equivariant bijection between the basis and the extended codes. The proof that our general Feichtner-Yuzvinsky type basis is indeed a basis for the augmented Chow ring of any matroid will be postponed to Section 5. 

\subsubsection{Extended codes and FY-basis for augmented Chow ring of $\mathsf{B}_n$} \label{SecExtCodeBasis}

The basis for the augmented Chow ring of $\mathsf{B}_n$, which we denote as $\widetilde{FY}(\mathsf{B}_n)$, consists of monomials $x_{F_1}^{a_1}\ldots x_{F_\ell}^{a_\ell}$ indexed by a chain in the Boolean lattice and exponents $a_i$ satisfying
\begin{itemize}
	\item $1\le a_1\le |F_1|$ and $1\le a_i\le |F_i|-|F_{i-1}|-1$ for all $i\ge 2$.
\end{itemize}
Note that $|F_1|\ge 1$ and $|F_i|-|F_{i-1}|\ge 2$ for all $i\ge 2$.
See Corollary \ref{AugBasis} for the case of general matroids. It will follow from Corollary \ref{AugBasis} that $\widetilde{FY}(\mathsf{B}_n)$ is indeed a basis for $\widetilde{A}(\mathsf{B}_n)$. 

Let $\widetilde{FY}^k(\mathsf{B}_n)$ be the subset of $\widetilde{FY}(\mathsf{B}_n)$ consisting of monomials of degree $k$, for $0\le k\le n$. This forms a basis for the homogeneous component $\widetilde{A}^k(\mathsf{B}_n)$ of $\widetilde{A}(\mathsf{B}_n)$. Clearly, $\widetilde{FY}(\mathsf{B}_n)=\biguplus_{k=0}^n\widetilde{FY}(\mathsf{B}_n)$.
\begin{exa}
The FY-basis for $\widetilde{A}^k(B_3)$:\\
$\widetilde{FY}^0(\mathsf{B}_3)=\{1\}$,\\
$\widetilde{FY}^1(\mathsf{B}_3)=\{x_1,x_2,x_3, x_{12}, x_{13}, x_{23},x_{123}\}$,\\
$\widetilde{FY}^2(\mathsf{B}_3)=\{x_1x_{123},x_2x_{123},x_3x_{123},x_{12}^2,x_{13}^2,x_{23}^2,x_{123}^2\}$,\\
$\widetilde{FY}^3(\mathsf{B}_3)=\{x_{123}^3\}$.

\end{exa}
Similar to the case of the Chow ring $A(\mathsf{B}_n)$, the action of $\S_n$ on the augmented Chow ring $\widetilde{A}(\mathsf{B}_n)=A(X_{\widetilde{\Sigma}_n})$ can also be restricted to the basis $\widetilde{FY}(\mathsf{B}_n)$. Hence we obtain the following theorem analogous to Proposition \ref{prop:FY(B_n)Perm}.

\begin{thm}\label{thm:AugPermbasis}
The $\S_n$-module $A(X_{\widetilde{\Sigma}_n})=\widetilde{A}(\mathsf{B}_n)$ and its homogeneous components $\widetilde{A}^k(\mathsf{B}_n)$ have permutation bases $\widetilde{FY}(\mathsf{B}_n)$ and $\widetilde{FY}^k(\mathsf{B}_n)$, respectively. 
\end{thm}

The next theorem is analogous to Theorem \ref{thm:bijectionBasis1}. It shows that the basis $\widetilde{FY}(\mathsf{B}_n)$ indeed has a combinatorial structure similar to the structure of the extended codes. 
\begin{thm}\label{thm:bijectionBasis2}
For $u=x_{F_1}^{a_1}\ldots x_{F_\ell}^{a_\ell}\in \widetilde{FY}(\mathsf{B}_n)$, define $\widetilde{\phi}(u)=(\alpha_1\ldots\alpha_n,f)$ by the following:
if $a_1=1$, then 
\[
    \alpha_i=\begin{cases} j-1 & \text{ if }i\in F_j-F_{j-1}\\
    \infty & \text{ if }i\in [n]-F_\ell
\end{cases} \quad \text{ for }i\in [n],
\]
and 
\[
    f(j)=a_{j+1} \text{ for }j\in [\ell];
\]
else if $a_1\ge 2$, then 
\[
    \alpha_i=\begin{cases} 
    j & \text{ if }i\in F_j-F_{j-1}\\
    \infty & \text{ if }i\in [n]-F_\ell
    \end{cases} \quad \text{ for }i\in [n], 
\]
and 
\[
    f(1)=a_1-1 \text{ and }f(j)=a_j\text{ for }2\le j\le \ell.
\]
Then $\widetilde{\phi}:\widetilde{FY}(\mathsf{B}_n)\rightarrow \widetilde{\Code}_n=\cup_{j=-1}^{n-1}\widetilde{\Code}_{n,j}$ is a bijection that satisfies $\deg(u)=\ind(\widetilde{\phi}(u))-1$ and respects the $\S_n$-actions on both sets.
\end{thm}

Clearly, the image $\widetilde{\phi}(u)=(\alpha,f)$ is an extended code and the procedure can be reversed easily. The detailed proof of Theorem \ref{thm:bijectionBasis2} is similar to the proof of Theorem \ref{thm:bijectionBasis1}, so we omit it here.

\begin{exa}
Let $u_1=x_{14}^{\mathbf{1}}x_{1247}x_{1245679}^2\in\widetilde{FY}(\mathsf{B}_{9})$, then $\widetilde{\phi}(u_1)$ is
\[
    0\underline{~~}~\underline{~~}0\underline{~~}~\underline{~~}~\underline{~~}~\underline{~~}~\underline{~~} \quad \rightarrow \quad
    01\underline{~~}0\underline{~~}~\underline{~~}\hat{1}\underline{~~}~\underline{~~} \quad \rightarrow \quad                         
    01\underline{~~}022\hat{1}\underline{~~}\hat{2} \quad \rightarrow \quad 01\infty022\hat{1}\infty\hat{2}.
\]
Let $u_2=x_{14}^{\mathbf{2}}x_{1247}x_{1245679}^2\in\widetilde{FY}(\mathsf{B}_{9})$, then  $\widetilde{\phi}(u_2)$ is
\[
    1\underline{~~}~\underline{~~}\hat1\underline{~~}~\underline{~~}~\underline{~~}~\underline{~~}~\underline{~~} \quad \rightarrow \quad  12\underline{~~}1\underline{~~}~\underline{~~}\hat{2}\underline{~~}~\underline{~~} \quad \rightarrow \quad                         12\underline{~~}133\hat{2}\underline{~~}\hat{3}\quad \rightarrow \quad 12\infty133\hat{2}\infty\hat{3}.
\]
The following are the extended codes of length $3$ and their corresponding monomial in $\widetilde{FY}(\mathsf{B}_3):$
\begin{align*}
&\begin{tikzcd}[ampersand replacement=\&, column sep=tiny]
  1  \& x_1 \& x_2 \& x_3  \&  x_{12}  \& x_{13}  \& x_{23}  \& x_{123}   \\
  \infty\infty\infty \arrow[u,leftrightarrow] \& 		  0\infty\infty \arrow[u,leftrightarrow] \& 
  \infty 0\infty \arrow[u,leftrightarrow] \& 
  \infty\infty 0 \arrow[u,leftrightarrow] \& 
  00\infty \arrow[u,leftrightarrow] \& 
  0\infty 0 \arrow[u,leftrightarrow] \&
  \infty00 \arrow[u,leftrightarrow]  \&
  000 \arrow[u,leftrightarrow]
\end{tikzcd}\\
&\begin{tikzcd}[ampersand replacement=\&, column sep=tiny]
  x_1x_{123}  \& x_2x_{123} \& x_3x_{123} \& x_{12}^2 \& x_{13}^2 \& x_{23}^2 \& x_{123}^2 \& x_{123}^3\\
  01\hat{1} \arrow[u,leftrightarrow] \&
  10\hat{1} \arrow[u,leftrightarrow] \&
  1\hat{1}0 \arrow[u,leftrightarrow] \&
  1\hat{1}\infty \arrow[u,leftrightarrow] \&
  1\infty\hat{1} \arrow[u,leftrightarrow] \&
  \infty 1\hat{1} \arrow[u,leftrightarrow] \&
  1\hat{1}1 \arrow[u,leftrightarrow] \&
  11\hat{1} \arrow[u,leftrightarrow] 
\end{tikzcd}.
\end{align*}
\end{exa}

Theorem \ref{thm:AugPermbasis} and Theorem \ref{thm:bijectionBasis2} give a permutation basis for the cohomology of the stellahedral variety in which the basis elements have a structure similar to the extended codes, thereby answering our question at the end of Section \ref{Sec:ExtCodes}. 

Combining Theorem \ref{thm:FcharExtCodes} and Theorem \ref{thm:bijectionBasis2} gives a combinatorial proof of Shareshian and Wachs' result in Theorem \ref{FrobStell}, which avoids the algebraic geometric machinery in the original proof.

\section{The augmented Chow ring is a Chow ring}
\label{Sec:AugIsChow}

In this section, we construct a basis for the augmented Chow ring of a matroid. We do it by showing that the augmented Chow ring of a matroid is actually a Chow ring in the sense of Feichtner and Yuzvinsky introduced in Section 2.
First, we turn our focus on another way of constructing the stellohedron $\widetilde{\Pi}_n$~\textemdash~ as the graph associahedron of the $n$-star graph $K_{1,n}$. In this case, the stellohedron is constructed in terms of building sets and nested set complexes introduced in Section \ref{sec:BS}. We compare this construction with the construction of the augmented Bergman fan of $\mathsf{B}_n$ and try to generalize the building set construction to any matroid. 

\subsection{Stellohedron as the graph associahedron of $K_{1,n}$}\label{Sec:GraphAss}
Let $G=(V,E)$ be a simple graph. The \emph{graphical building set} $\B(G)$ is the set of nonempty subsets $I$ of $V$ such that the induced subgraph of $G$ on $I$ is connected. Set $V=[n]$. For $I\subseteq V$, let $\Delta_I$ be the simplex obtained by taking the convex hull of $\{e_i\in\mathbb{R}^{n}:i\in I\}$. Then the \emph{graph associahedron} of $G$ is defined as the following Minkowski sum of scaled simplices
\[
    \mathsf{Assoc}(G)\coloneqq \sum_{I\in\B(G)}y_I\Delta_I  \quad\text{with all }y_I>0.
\]

Now let's consider the case that $G$ is the $n$-star graph. The \emph{$n$-star graph} $K_{1,n}$ has vertex set $V=[n]\cup\{*\}$ and edge set $E=\{\{i,*\}:i\in [n]\}$. Then for any $n$, the graphical building set of $K_{1,n}$ is given by the following
\[
    \B(K_{1,n})=\{\{i\}\}_{i\in [n]}\cup\{S\cup\{*\}\}_{S\subseteq [n]}.
\]

\begin{exa}
When $n=2$, the graphical building set $\mathcal{B}(K_{1,2})$ consists of:
\begin{center}
\begin{tikzpicture}[x=0.6pt,y=0.6pt,yscale=-1,xscale=1]

\draw   (44.37,31.05) .. controls (44.28,29.78) and (45.34,28.69) .. (46.74,28.61) .. controls (48.14,28.52) and (49.35,29.49) .. (49.44,30.76) .. controls (49.53,32.02) and (48.47,33.12) .. (47.07,33.2) .. controls (45.67,33.29) and (44.46,32.32) .. (44.37,31.05) -- cycle ;
\draw   (20.27,57.26) .. controls (20.17,55.72) and (21.34,54.39) .. (22.88,54.29) .. controls (24.42,54.19) and (25.75,55.36) .. (25.85,56.9) .. controls (25.95,58.44) and (24.78,59.77) .. (23.24,59.87) .. controls (21.7,59.97) and (20.37,58.81) .. (20.27,57.26) -- cycle ;
\draw   (67.93,57.08) .. controls (67.83,55.73) and (69,54.57) .. (70.54,54.48) .. controls (72.08,54.4) and (73.41,55.42) .. (73.51,56.76) .. controls (73.61,58.11) and (72.44,59.27) .. (70.9,59.36) .. controls (69.36,59.45) and (68.03,58.43) .. (67.93,57.08) -- cycle ;
\draw    (45.12,32.57) -- (24.34,55.39) ;
\draw    (48.68,32.27) -- (69.03,55.24) ;
\draw   (128.7,29.46) .. controls (128.61,28.19) and (129.67,27.09) .. (131.07,27.01) .. controls (132.47,26.93) and (133.68,27.89) .. (133.77,29.16) .. controls (133.86,30.43) and (132.8,31.52) .. (131.4,31.61) .. controls (130,31.69) and (128.79,30.73) .. (128.7,29.46) -- cycle ;
\draw   (104.6,55.67) .. controls (104.5,54.13) and (105.67,52.8) .. (107.21,52.7) .. controls (108.75,52.6) and (110.08,53.76) .. (110.18,55.3) .. controls (110.28,56.85) and (109.11,58.18) .. (107.57,58.28) .. controls (106.03,58.38) and (104.7,57.21) .. (104.6,55.67) -- cycle ;
\draw   (152.26,55.48) .. controls (152.16,54.14) and (153.33,52.97) .. (154.87,52.89) .. controls (156.41,52.8) and (157.74,53.82) .. (157.84,55.17) .. controls (157.94,56.51) and (156.77,57.68) .. (155.23,57.76) .. controls (153.69,57.85) and (152.36,56.83) .. (152.26,55.48) -- cycle ;
\draw    (129.46,30.97) -- (108.67,53.79) ;
\draw    (133.02,30.67) -- (153.36,53.64) ;
\draw   (207.37,28.46) .. controls (207.28,27.19) and (208.34,26.09) .. (209.74,26.01) .. controls (211.14,25.93) and (212.35,26.89) .. (212.44,28.16) .. controls (212.53,29.43) and (211.47,30.52) .. (210.07,30.61) .. controls (208.67,30.69) and (207.46,29.73) .. (207.37,28.46) -- cycle ;
\draw   (183.27,54.67) .. controls (183.17,53.13) and (184.34,51.8) .. (185.88,51.7) .. controls (187.42,51.6) and (188.75,52.76) .. (188.85,54.3) .. controls (188.95,55.85) and (187.78,57.18) .. (186.24,57.28) .. controls (184.7,57.38) and (183.37,56.21) .. (183.27,54.67) -- cycle ;
\draw   (230.93,54.48) .. controls (230.83,53.14) and (232,51.97) .. (233.54,51.89) .. controls (235.08,51.8) and (236.41,52.82) .. (236.51,54.17) .. controls (236.61,55.51) and (235.44,56.68) .. (233.9,56.76) .. controls (232.36,56.85) and (231.03,55.83) .. (230.93,54.48) -- cycle ;
\draw    (208.12,29.97) -- (187.34,52.79) ;
\draw    (211.68,29.67) -- (232.03,52.64) ;
\draw   (289.03,27.79) .. controls (288.94,26.52) and (290,25.42) .. (291.41,25.34) .. controls (292.81,25.26) and (294.01,26.22) .. (294.11,27.49) .. controls (294.2,28.76) and (293.13,29.86) .. (291.73,29.94) .. controls (290.33,30.02) and (289.13,29.06) .. (289.03,27.79) -- cycle ;
\draw   (264.94,54) .. controls (264.84,52.46) and (266,51.13) .. (267.54,51.03) .. controls (269.08,50.93) and (270.41,52.1) .. (270.51,53.64) .. controls (270.61,55.18) and (269.45,56.51) .. (267.91,56.61) .. controls (266.37,56.71) and (265.04,55.54) .. (264.94,54) -- cycle ;
\draw   (312.6,53.82) .. controls (312.5,52.47) and (313.66,51.31) .. (315.21,51.22) .. controls (316.75,51.13) and (318.08,52.15) .. (318.18,53.5) .. controls (318.28,54.85) and (317.11,56.01) .. (315.57,56.1) .. controls (314.03,56.18) and (312.7,55.16) .. (312.6,53.82) -- cycle ;
\draw    (289.79,29.3) -- (269,52.13) ;
\draw    (293.35,29) -- (313.69,51.98) ;
\draw   (366.37,27.46) .. controls (366.28,26.19) and (367.34,25.09) .. (368.74,25.01) .. controls (370.14,24.93) and (371.35,25.89) .. (371.44,27.16) .. controls (371.53,28.43) and (370.47,29.52) .. (369.07,29.61) .. controls (367.67,29.69) and (366.46,28.73) .. (366.37,27.46) -- cycle ;
\draw   (342.27,53.67) .. controls (342.17,52.13) and (343.34,50.8) .. (344.88,50.7) .. controls (346.42,50.6) and (347.75,51.76) .. (347.85,53.3) .. controls (347.95,54.85) and (346.78,56.18) .. (345.24,56.28) .. controls (343.7,56.38) and (342.37,55.21) .. (342.27,53.67) -- cycle ;
\draw   (389.93,53.48) .. controls (389.83,52.14) and (391,50.97) .. (392.54,50.89) .. controls (394.08,50.8) and (395.41,51.82) .. (395.51,53.17) .. controls (395.61,54.51) and (394.44,55.68) .. (392.9,55.76) .. controls (391.36,55.85) and (390.03,54.83) .. (389.93,53.48) -- cycle ;
\draw    (367.12,28.97) -- (346.34,51.79) ;
\draw    (370.68,28.67) -- (391.03,51.64) ;
\draw  [color={rgb, 255:red, 144; green, 19; blue, 254 }  ,draw opacity=1 ] (39.14,31.84) .. controls (39.14,27.58) and (42.58,24.14) .. (46.83,24.14) .. controls (51.09,24.14) and (54.53,27.58) .. (54.53,31.84) .. controls (54.53,36.09) and (51.09,39.53) .. (46.83,39.53) .. controls (42.58,39.53) and (39.14,36.09) .. (39.14,31.84) -- cycle ;
\draw  [color={rgb, 255:red, 144; green, 19; blue, 254 }  ,draw opacity=1 ] (99.69,55.49) .. controls (99.69,51.23) and (103.14,47.79) .. (107.39,47.79) .. controls (111.64,47.79) and (115.09,51.23) .. (115.09,55.49) .. controls (115.09,59.74) and (111.64,63.18) .. (107.39,63.18) .. controls (103.14,63.18) and (99.69,59.74) .. (99.69,55.49) -- cycle ;
\draw  [color={rgb, 255:red, 144; green, 19; blue, 254 }  ,draw opacity=1 ] (226.02,54.33) .. controls (226.02,50.07) and (229.47,46.63) .. (233.72,46.63) .. controls (237.97,46.63) and (241.42,50.07) .. (241.42,54.33) .. controls (241.42,58.58) and (237.97,62.02) .. (233.72,62.02) .. controls (229.47,62.02) and (226.02,58.58) .. (226.02,54.33) -- cycle ;
\draw  [color={rgb, 255:red, 144; green, 19; blue, 254 }  ,draw opacity=1 ] (262.78,55.96) .. controls (260.73,54.07) and (260.59,50.88) .. (262.48,48.82) -- (286.07,23.18) .. controls (287.96,21.13) and (291.15,21) .. (293.21,22.89) -- (296.01,25.46) .. controls (298.07,27.35) and (298.2,30.55) .. (296.31,32.61) -- (272.73,58.24) .. controls (270.84,60.3) and (267.64,60.43) .. (265.58,58.54) -- cycle ;
\draw   (443.7,26.46) .. controls (443.61,25.19) and (444.67,24.09) .. (446.07,24.01) .. controls (447.47,23.93) and (448.68,24.89) .. (448.77,26.16) .. controls (448.86,27.43) and (447.8,28.52) .. (446.4,28.61) .. controls (445,28.69) and (443.79,27.73) .. (443.7,26.46) -- cycle ;
\draw   (419.6,52.67) .. controls (419.5,51.13) and (420.67,49.8) .. (422.21,49.7) .. controls (423.75,49.6) and (425.08,50.76) .. (425.18,52.3) .. controls (425.28,53.85) and (424.11,55.18) .. (422.57,55.28) .. controls (421.03,55.38) and (419.7,54.21) .. (419.6,52.67) -- cycle ;
\draw   (467.26,52.48) .. controls (467.16,51.14) and (468.33,49.97) .. (469.87,49.89) .. controls (471.41,49.8) and (472.74,50.82) .. (472.84,52.17) .. controls (472.94,53.51) and (471.77,54.68) .. (470.23,54.76) .. controls (468.69,54.85) and (467.36,53.83) .. (467.26,52.48) -- cycle ;
\draw    (444.46,27.97) -- (423.67,50.79) ;
\draw    (448.02,27.67) -- (468.36,50.64) ;
\draw  [color={rgb, 255:red, 144; green, 19; blue, 254 }  ,draw opacity=1 ] (367.89,21.7) .. controls (370.03,19.91) and (373.22,20.19) .. (375.01,22.33) -- (397.37,49.05) .. controls (399.16,51.19) and (398.88,54.38) .. (396.74,56.17) -- (393.81,58.61) .. controls (391.67,60.41) and (388.49,60.12) .. (386.69,57.98) -- (364.34,31.26) .. controls (362.55,29.12) and (362.83,25.94) .. (364.97,24.14) -- cycle ;
\draw  [color={rgb, 255:red, 144; green, 19; blue, 254 }  ,draw opacity=1 ] (465.93,31.4) .. controls (490.6,63.4) and (489.27,67.4) .. (468.6,60.07) .. controls (447.93,52.74) and (449.27,49.4) .. (422.6,60.07) .. controls (395.93,70.74) and (419.27,41.4) .. (434.6,22.74) .. controls (449.93,4.07) and (441.27,-0.6) .. (465.93,31.4) -- cycle ;

\draw (11.21,57.44) node [anchor=north west][inner sep=0.75pt]  [font=\tiny] [align=left] {$\displaystyle 1$};
\draw (41.27,12.6) node [anchor=north west][inner sep=0.75pt]  [font=\footnotesize] [align=left] {$\displaystyle \ast $};
\draw (77.99,57.26) node [anchor=north west][inner sep=0.75pt]  [font=\tiny] [align=left] {$\displaystyle 2$};
\draw (92.54,57.51) node [anchor=north west][inner sep=0.75pt]  [font=\tiny] [align=left] {$\displaystyle 1$};
\draw (125.6,11) node [anchor=north west][inner sep=0.75pt]  [font=\footnotesize] [align=left] {$\displaystyle \ast $};
\draw (162.33,52.66) node [anchor=north west][inner sep=0.75pt]  [font=\tiny] [align=left] {$\displaystyle 2$};
\draw (174.21,53.51) node [anchor=north west][inner sep=0.75pt]  [font=\tiny] [align=left] {$\displaystyle 1$};
\draw (204.27,10) node [anchor=north west][inner sep=0.75pt]  [font=\footnotesize] [align=left] {$\displaystyle \ast $};
\draw (240.99,54.66) node [anchor=north west][inner sep=0.75pt]  [font=\tiny] [align=left] {$\displaystyle 2$};
\draw (255.87,52.84) node [anchor=north west][inner sep=0.75pt]  [font=\tiny] [align=left] {$\displaystyle 1$};
\draw (285.93,9.33) node [anchor=north west][inner sep=0.75pt]  [font=\footnotesize] [align=left] {$\displaystyle \ast $};
\draw (320.66,54) node [anchor=north west][inner sep=0.75pt]  [font=\tiny] [align=left] {$\displaystyle 2$};
\draw (334.21,53.51) node [anchor=north west][inner sep=0.75pt]  [font=\tiny] [align=left] {$\displaystyle 1$};
\draw (363.27,9) node [anchor=north west][inner sep=0.75pt]  [font=\footnotesize] [align=left] {$\displaystyle \ast $};
\draw (399.99,53.66) node [anchor=north west][inner sep=0.75pt]  [font=\tiny] [align=left] {$\displaystyle 2$};
\draw (411.54,52.51) node [anchor=north west][inner sep=0.75pt]  [font=\tiny] [align=left] {$\displaystyle 1$};
\draw (440.6,8) node [anchor=north west][inner sep=0.75pt]  [font=\footnotesize] [align=left] {$\displaystyle \ast $};
\draw (476.33,52.66) node [anchor=north west][inner sep=0.75pt]  [font=\tiny] [align=left] {$\displaystyle 2$};
\end{tikzpicture}
\end{center}
\end{exa}

We view $*$ in the vertex set $V=[n]\cup\{*\}$ as $n+1$. Then the stellohedron $\widetilde{\Pi}_n$ as the graph associahedron of $K_{1,n}$ is  the follwoing Minkowski sum of scaled simplices
\[
    \widetilde{\Pi}_n=\sum_{I\in\B(K_{1,n})}y_I\Delta_I \quad \mbox{ with all }y_I>0.
\]

In fact, the graphical building set $\B(G)$ is a special case of Feichtner and Yuzvinsky's building set of the Boolean lattice $B_{n}$ on the vertex set $V$. For a graphical building set $\B(G)$, the corresponding nested sets $N\subset\B(G)$ can be characterized by:
\begin{itemize}
\item[(1)] $\forall~I,J\in N$, either $I\subset J$, $I\supset J$ or $I\cap J=\emptyset$.
\item[(2)] $\forall~I,J\in N$ if $I\cap J=\emptyset$, then $I\cup J\notin\B(G)$ (not ``connected").
\end{itemize}
From Postnikov \cite[Theorem 7.4]{PostnikovGenPermu2005}, the face lattice of $\mathsf{Assoc}(G)$ is independent of the choice of positive value $y_I$ for each $I\in\B(G)$ and is isomorphic to the dual of the face lattice of the reduced nested set complex $\widetilde{\N}(B_{|V|},\B(G))$.

When $G$ is $K_{1,n}$, then a nested set $N\in \N(B_{n+1},\B(K_{1,n}))$ is of the form
\begin{equation}  \label{eq:NestedSetBn}
    N=\{\{i\}\}_{i\in I}\cup\{S\cup\{*\}\}_{S\in\F}
\end{equation}
where $\F=\{S_1\subsetneq \ldots\subsetneq S_\ell\}$ is any flag in the Boolean lattice $B_n$ and $I$ is a subset of $[n]$ such that $I\subseteq S_1$. Recall that we exclude those nested sets whose corresponding flag $\F$ contains $[n]$ from the nested set complex $\N(B_{n+1},\B(K_{1,n}))$ to obtain the reduced nested set complex $\widetilde{\N}(B_{n+1},\B(K_{1,n}))$.
\begin{exa} \label{ExNestedComplexSt}
The reduced nested set complex with respect to $\B(K_{1,2})$ is combinatorially equivalent to the boundary simplicial complex of $\widetilde{\Pi}_2^*$: 

\begin{center}
\begin{tikzpicture}[x=0.6pt,y=0.6pt,yscale=-1,xscale=1]

\draw   (78.48,199.98) .. controls (78.4,198.83) and (79.36,197.83) .. (80.63,197.76) .. controls (81.89,197.68) and (82.99,198.56) .. (83.07,199.71) .. controls (83.15,200.86) and (82.19,201.85) .. (80.92,201.93) .. controls (79.66,202) and (78.56,201.13) .. (78.48,199.98) -- cycle ;
\draw   (56.66,223.76) .. controls (56.57,222.36) and (57.63,221.15) .. (59.02,221.06) .. controls (60.42,220.97) and (61.62,222.03) .. (61.71,223.43) .. controls (61.8,224.82) and (60.74,226.03) .. (59.35,226.12) .. controls (57.95,226.21) and (56.75,225.15) .. (56.66,223.76) -- cycle ;
\draw   (99.81,223.59) .. controls (99.72,222.37) and (100.78,221.31) .. (102.18,221.23) .. controls (103.57,221.15) and (104.77,222.08) .. (104.87,223.3) .. controls (104.96,224.52) and (103.9,225.58) .. (102.5,225.66) .. controls (101.11,225.74) and (99.9,224.81) .. (99.81,223.59) -- cycle ;
\draw    (79.16,201.35) -- (60.34,222.05) ;
\draw    (82.39,201.08) -- (100.81,221.92) ;
\draw  [color={rgb, 255:red, 144; green, 19; blue, 254 }  ,draw opacity=1 ] (73.74,200.69) .. controls (73.74,196.83) and (76.86,193.7) .. (80.71,193.7) .. controls (84.56,193.7) and (87.68,196.83) .. (87.68,200.69) .. controls (87.68,204.55) and (84.56,207.67) .. (80.71,207.67) .. controls (76.86,207.67) and (73.74,204.55) .. (73.74,200.69) -- cycle ;

\draw   (237.4,88.49) .. controls (237.31,87.34) and (238.27,86.35) .. (239.54,86.27) .. controls (240.81,86.2) and (241.9,87.07) .. (241.99,88.22) .. controls (242.07,89.37) and (241.11,90.37) .. (239.84,90.44) .. controls (238.57,90.52) and (237.48,89.64) .. (237.4,88.49) -- cycle ;
\draw   (215.58,112.27) .. controls (215.49,110.87) and (216.54,109.66) .. (217.94,109.57) .. controls (219.33,109.48) and (220.54,110.54) .. (220.63,111.94) .. controls (220.72,113.34) and (219.66,114.54) .. (218.27,114.63) .. controls (216.87,114.73) and (215.67,113.67) .. (215.58,112.27) -- cycle ;
\draw   (258.73,112.1) .. controls (258.64,110.88) and (259.7,109.83) .. (261.09,109.75) .. controls (262.49,109.67) and (263.69,110.59) .. (263.78,111.81) .. controls (263.87,113.04) and (262.82,114.09) .. (261.42,114.17) .. controls (260.03,114.25) and (258.82,113.32) .. (258.73,112.1) -- cycle ;
\draw    (238.08,89.86) -- (219.26,110.57) ;
\draw    (241.3,89.59) -- (259.72,110.43) ;
\draw  [color={rgb, 255:red, 144; green, 19; blue, 254 }  ,draw opacity=1 ] (211.13,112.1) .. controls (211.13,108.25) and (214.25,105.12) .. (218.1,105.12) .. controls (221.95,105.12) and (225.07,108.25) .. (225.07,112.1) .. controls (225.07,115.96) and (221.95,119.09) .. (218.1,119.09) .. controls (214.25,119.09) and (211.13,115.96) .. (211.13,112.1) -- cycle ;

\draw   (132.69,17.54) .. controls (132.61,16.39) and (133.57,15.4) .. (134.84,15.32) .. controls (136.11,15.25) and (137.2,16.12) .. (137.29,17.27) .. controls (137.37,18.42) and (136.41,19.42) .. (135.14,19.49) .. controls (133.87,19.57) and (132.78,18.69) .. (132.69,17.54) -- cycle ;
\draw   (110.87,41.32) .. controls (110.78,39.92) and (111.84,38.72) .. (113.24,38.63) .. controls (114.63,38.54) and (115.83,39.59) .. (115.93,40.99) .. controls (116.02,42.39) and (114.96,43.6) .. (113.56,43.69) .. controls (112.17,43.78) and (110.96,42.72) .. (110.87,41.32) -- cycle ;
\draw   (154.03,41.15) .. controls (153.94,39.93) and (155,38.88) .. (156.39,38.8) .. controls (157.79,38.72) and (158.99,39.64) .. (159.08,40.87) .. controls (159.17,42.09) and (158.11,43.14) .. (156.72,43.22) .. controls (155.32,43.3) and (154.12,42.38) .. (154.03,41.15) -- cycle ;
\draw    (133.38,18.92) -- (114.56,39.62) ;
\draw    (136.6,18.64) -- (155.02,39.48) ;
\draw  [color={rgb, 255:red, 144; green, 19; blue, 254 }  ,draw opacity=1 ] (149.58,41.01) .. controls (149.58,37.15) and (152.7,34.03) .. (156.55,34.03) .. controls (160.4,34.03) and (163.52,37.15) .. (163.52,41.01) .. controls (163.52,44.87) and (160.4,47.99) .. (156.55,47.99) .. controls (152.7,47.99) and (149.58,44.87) .. (149.58,41.01) -- cycle ;

\draw   (200.81,199.14) .. controls (200.73,197.99) and (201.69,197) .. (202.96,196.92) .. controls (204.23,196.85) and (205.32,197.72) .. (205.41,198.87) .. controls (205.49,200.02) and (204.53,201.02) .. (203.26,201.09) .. controls (201.99,201.17) and (200.9,200.29) .. (200.81,199.14) -- cycle ;
\draw   (178.99,222.92) .. controls (178.9,221.52) and (179.96,220.32) .. (181.36,220.23) .. controls (182.75,220.14) and (183.95,221.19) .. (184.05,222.59) .. controls (184.14,223.99) and (183.08,225.2) .. (181.68,225.29) .. controls (180.29,225.38) and (179.08,224.32) .. (178.99,222.92) -- cycle ;
\draw   (222.15,222.75) .. controls (222.06,221.53) and (223.12,220.48) .. (224.51,220.4) .. controls (225.91,220.32) and (227.11,221.24) .. (227.2,222.47) .. controls (227.29,223.69) and (226.23,224.74) .. (224.84,224.82) .. controls (223.44,224.9) and (222.24,223.98) .. (222.15,222.75) -- cycle ;
\draw    (201.5,200.52) -- (182.67,221.22) ;
\draw    (204.72,200.24) -- (223.14,221.08) ;
\draw  [color={rgb, 255:red, 144; green, 19; blue, 254 }  ,draw opacity=1 ] (177.04,224.7) .. controls (175.18,222.99) and (175.06,220.09) .. (176.77,218.22) -- (198.13,194.97) .. controls (199.84,193.1) and (202.73,192.98) .. (204.59,194.7) -- (207.13,197.03) .. controls (208.99,198.75) and (209.11,201.65) .. (207.4,203.51) -- (186.05,226.77) .. controls (184.33,228.63) and (181.44,228.75) .. (179.58,227.04) -- cycle ;

\draw   (35.58,85.34) .. controls (35.5,84.19) and (36.46,83.19) .. (37.73,83.12) .. controls (39,83.04) and (40.09,83.92) .. (40.18,85.07) .. controls (40.26,86.22) and (39.3,87.21) .. (38.03,87.29) .. controls (36.76,87.36) and (35.67,86.49) .. (35.58,85.34) -- cycle ;
\draw   (13.76,109.12) .. controls (13.67,107.72) and (14.73,106.51) .. (16.13,106.42) .. controls (17.52,106.33) and (18.73,107.39) .. (18.82,108.79) .. controls (18.91,110.19) and (17.85,111.39) .. (16.45,111.48) .. controls (15.06,111.57) and (13.86,110.51) .. (13.76,109.12) -- cycle ;
\draw   (56.92,108.95) .. controls (56.83,107.73) and (57.89,106.67) .. (59.28,106.59) .. controls (60.68,106.51) and (61.88,107.44) .. (61.97,108.66) .. controls (62.06,109.88) and (61,110.94) .. (59.61,111.02) .. controls (58.21,111.1) and (57.01,110.17) .. (56.92,108.95) -- cycle ;
\draw    (36.27,86.71) -- (17.45,107.42) ;
\draw    (39.49,86.44) -- (57.91,107.28) ;
\draw  [color={rgb, 255:red, 144; green, 19; blue, 254 }  ,draw opacity=1 ] (36.97,80.12) .. controls (38.9,78.49) and (41.79,78.75) .. (43.41,80.69) -- (63.66,104.93) .. controls (65.28,106.87) and (65.02,109.76) .. (63.08,111.39) -- (60.44,113.6) .. controls (58.5,115.23) and (55.61,114.97) .. (53.99,113.03) -- (33.75,88.79) .. controls (32.12,86.85) and (32.38,83.96) .. (34.32,82.33) -- cycle ;

\draw    (55.82,86.89) -- (101.09,47.58) ;
\draw    (169.9,45.16) -- (225.74,91.12) ;
\draw    (52.2,119.55) -- (66.08,189.7) ;
\draw    (233.29,117.73) -- (212.76,196.35) ;
\draw    (105.32,213.89) -- (172.32,214.5) ;
\draw   (53.75,33.37) .. controls (53.69,32.52) and (54.43,31.78) .. (55.4,31.73) .. controls (56.37,31.67) and (57.21,32.32) .. (57.27,33.17) .. controls (57.33,34.03) and (56.6,34.77) .. (55.63,34.82) .. controls (54.66,34.88) and (53.82,34.23) .. (53.75,33.37) -- cycle ;
\draw   (37.05,51.02) .. controls (36.98,49.99) and (37.79,49.09) .. (38.85,49.02) .. controls (39.92,48.95) and (40.84,49.74) .. (40.91,50.78) .. controls (40.98,51.82) and (40.17,52.71) .. (39.11,52.78) .. controls (38.04,52.85) and (37.12,52.06) .. (37.05,51.02) -- cycle ;
\draw   (70.09,50.9) .. controls (70.02,49.99) and (70.83,49.21) .. (71.9,49.15) .. controls (72.97,49.09) and (73.89,49.78) .. (73.96,50.69) .. controls (74.03,51.59) and (73.22,52.38) .. (72.15,52.43) .. controls (71.08,52.49) and (70.16,51.81) .. (70.09,50.9) -- cycle ;
\draw    (54.28,34.39) -- (39.86,49.76) ;
\draw    (56.75,34.19) -- (70.85,49.66) ;
\draw  [color={rgb, 255:red, 144; green, 19; blue, 254 }  ,draw opacity=1 ] (54.62,29.09) .. controls (56.21,27.8) and (58.58,28.01) .. (59.92,29.56) -- (77.94,50.47) .. controls (79.28,52.02) and (79.07,54.33) .. (77.48,55.62) -- (75.31,57.39) .. controls (73.72,58.68) and (71.34,58.47) .. (70.01,56.92) -- (51.99,36) .. controls (50.65,34.45) and (50.86,32.14) .. (52.45,30.85) -- cycle ;
\draw  [color={rgb, 255:red, 144; green, 19; blue, 254 }  ,draw opacity=1 ] (67.58,50.79) .. controls (67.58,48.33) and (69.57,46.34) .. (72.03,46.34) .. controls (74.48,46.34) and (76.47,48.33) .. (76.47,50.79) .. controls (76.47,53.25) and (74.48,55.25) .. (72.03,55.25) .. controls (69.57,55.25) and (67.58,53.25) .. (67.58,50.79) -- cycle ;

\draw   (215.47,38.79) .. controls (215.41,37.96) and (216.1,37.25) .. (217.02,37.19) .. controls (217.93,37.14) and (218.72,37.77) .. (218.78,38.6) .. controls (218.84,39.43) and (218.15,40.15) .. (217.23,40.2) .. controls (216.32,40.26) and (215.53,39.63) .. (215.47,38.79) -- cycle ;
\draw   (199.72,55.95) .. controls (199.65,54.95) and (200.42,54.07) .. (201.42,54.01) .. controls (202.43,53.94) and (203.3,54.71) .. (203.37,55.72) .. controls (203.43,56.73) and (202.67,57.6) .. (201.66,57.66) .. controls (200.65,57.73) and (199.79,56.96) .. (199.72,55.95) -- cycle ;
\draw   (230.86,55.83) .. controls (230.8,54.95) and (231.56,54.19) .. (232.57,54.13) .. controls (233.58,54.08) and (234.44,54.74) .. (234.51,55.63) .. controls (234.58,56.51) and (233.81,57.27) .. (232.81,57.33) .. controls (231.8,57.38) and (230.93,56.72) .. (230.86,55.83) -- cycle ;
\draw    (215.96,39.78) -- (202.38,54.73) ;
\draw    (218.29,39.59) -- (231.58,54.63) ;
\draw  [color={rgb, 255:red, 144; green, 19; blue, 254 }  ,draw opacity=1 ] (227.66,55.73) .. controls (227.66,52.95) and (229.91,50.69) .. (232.69,50.69) .. controls (235.47,50.69) and (237.72,52.95) .. (237.72,55.73) .. controls (237.72,58.51) and (235.47,60.77) .. (232.69,60.77) .. controls (229.91,60.77) and (227.66,58.51) .. (227.66,55.73) -- cycle ;
\draw  [color={rgb, 255:red, 144; green, 19; blue, 254 }  ,draw opacity=1 ] (196.51,55.84) .. controls (196.51,53.05) and (198.76,50.8) .. (201.54,50.8) .. controls (204.32,50.8) and (206.57,53.05) .. (206.57,55.84) .. controls (206.57,58.62) and (204.32,60.88) .. (201.54,60.88) .. controls (198.76,60.88) and (196.51,58.62) .. (196.51,55.84) -- cycle ;

\draw   (32.85,150.58) .. controls (32.78,149.76) and (33.5,149.04) .. (34.44,148.99) .. controls (35.38,148.93) and (36.19,149.56) .. (36.25,150.39) .. controls (36.31,151.22) and (35.6,151.93) .. (34.66,151.99) .. controls (33.72,152.04) and (32.91,151.41) .. (32.85,150.58) -- cycle ;
\draw   (16.66,167.68) .. controls (16.6,166.67) and (17.38,165.8) .. (18.41,165.74) .. controls (19.45,165.67) and (20.34,166.44) .. (20.41,167.44) .. controls (20.48,168.44) and (19.69,169.31) .. (18.66,169.38) .. controls (17.62,169.44) and (16.73,168.68) .. (16.66,167.68) -- cycle ;
\draw   (48.67,167.56) .. controls (48.6,166.68) and (49.39,165.92) .. (50.42,165.86) .. controls (51.45,165.81) and (52.35,166.47) .. (52.41,167.35) .. controls (52.48,168.23) and (51.7,168.99) .. (50.66,169.04) .. controls (49.63,169.1) and (48.73,168.43) .. (48.67,167.56) -- cycle ;
\draw    (33.35,151.57) -- (19.39,166.45) ;
\draw    (35.74,151.38) -- (49.4,166.36) ;
\draw  [color={rgb, 255:red, 144; green, 19; blue, 254 }  ,draw opacity=1 ] (31.52,145.38) .. controls (33,144.06) and (35.31,144.16) .. (36.68,145.6) -- (55.05,165.01) .. controls (56.42,166.45) and (56.32,168.69) .. (54.84,170.02) -- (52.82,171.82) .. controls (51.34,173.14) and (49.03,173.05) .. (47.67,171.61) -- (29.29,152.19) .. controls (27.93,150.75) and (28.02,148.51) .. (29.5,147.19) -- cycle ;
\draw  [color={rgb, 255:red, 144; green, 19; blue, 254 }  ,draw opacity=1 ] (30.42,150.58) .. controls (30.42,148.2) and (32.34,146.27) .. (34.72,146.27) .. controls (37.1,146.27) and (39.03,148.2) .. (39.03,150.58) .. controls (39.03,152.96) and (37.1,154.9) .. (34.72,154.9) .. controls (32.34,154.9) and (30.42,152.96) .. (30.42,150.58) -- cycle ;

\draw   (247.53,153.14) .. controls (247.47,152.25) and (248.21,151.48) .. (249.19,151.42) .. controls (250.17,151.36) and (251.02,152.04) .. (251.08,152.93) .. controls (251.15,153.82) and (250.4,154.59) .. (249.42,154.65) .. controls (248.44,154.7) and (247.59,154.03) .. (247.53,153.14) -- cycle ;
\draw   (230.65,171.54) .. controls (230.58,170.45) and (231.39,169.52) .. (232.47,169.45) .. controls (233.55,169.38) and (234.49,170.2) .. (234.56,171.28) .. controls (234.63,172.36) and (233.81,173.3) .. (232.73,173.37) .. controls (231.65,173.44) and (230.72,172.62) .. (230.65,171.54) -- cycle ;
\draw   (264.04,171.41) .. controls (263.97,170.46) and (264.79,169.64) .. (265.87,169.58) .. controls (266.95,169.52) and (267.88,170.24) .. (267.95,171.18) .. controls (268.02,172.13) and (267.2,172.95) .. (266.12,173.01) .. controls (265.04,173.07) and (264.11,172.35) .. (264.04,171.41) -- cycle ;
\draw    (248.06,154.2) -- (233.5,170.22) ;
\draw    (250.55,153.99) -- (264.81,170.11) ;
\draw  [color={rgb, 255:red, 144; green, 19; blue, 254 }  ,draw opacity=1 ] (227.55,175.05) .. controls (226.01,173.64) and (225.91,171.24) .. (227.33,169.71) -- (245.99,149.38) .. controls (247.41,147.84) and (249.79,147.74) .. (251.33,149.15) -- (253.42,151.08) .. controls (254.96,152.5) and (255.06,154.89) .. (253.65,156.43) -- (234.98,176.76) .. controls (233.57,178.3) and (231.18,178.4) .. (229.64,176.98) -- cycle ;
\draw  [color={rgb, 255:red, 144; green, 19; blue, 254 }  ,draw opacity=1 ] (228.55,171.55) .. controls (228.55,169.39) and (230.3,167.64) .. (232.46,167.64) .. controls (234.61,167.64) and (236.37,169.39) .. (236.37,171.55) .. controls (236.37,173.72) and (234.61,175.47) .. (232.46,175.47) .. controls (230.3,175.47) and (228.55,173.72) .. (228.55,171.55) -- cycle ;

\draw   (138.26,229.21) .. controls (138.19,228.34) and (138.92,227.59) .. (139.87,227.54) .. controls (140.82,227.48) and (141.64,228.14) .. (141.7,229) .. controls (141.77,229.87) and (141.04,230.61) .. (140.09,230.67) .. controls (139.14,230.73) and (138.32,230.07) .. (138.26,229.21) -- cycle ;
\draw   (121.87,247.06) .. controls (121.8,246.01) and (122.59,245.11) .. (123.64,245.04) .. controls (124.69,244.97) and (125.59,245.77) .. (125.66,246.82) .. controls (125.73,247.87) and (124.94,248.77) .. (123.89,248.84) .. controls (122.84,248.91) and (121.94,248.11) .. (121.87,247.06) -- cycle ;
\draw   (154.28,246.94) .. controls (154.21,246.02) and (155.01,245.23) .. (156.05,245.17) .. controls (157.1,245.11) and (158.01,245.8) .. (158.07,246.72) .. controls (158.14,247.64) and (157.35,248.43) .. (156.3,248.49) .. controls (155.25,248.55) and (154.35,247.86) .. (154.28,246.94) -- cycle ;
\draw    (138.77,230.24) -- (124.63,245.79) ;
\draw    (141.19,230.03) -- (155.02,245.68) ;
\draw  [color={rgb, 255:red, 144; green, 19; blue, 254 }  ,draw opacity=1 ] (119.54,249.8) .. controls (118.05,248.42) and (117.95,246.1) .. (119.32,244.61) -- (137.44,224.87) .. controls (138.81,223.38) and (141.13,223.28) .. (142.62,224.66) -- (144.66,226.53) .. controls (146.15,227.9) and (146.24,230.23) .. (144.87,231.72) -- (126.75,251.45) .. controls (125.38,252.95) and (123.06,253.04) .. (121.57,251.67) -- cycle ;
\draw  [color={rgb, 255:red, 144; green, 19; blue, 254 }  ,draw opacity=1 ] (136.15,229.37) .. controls (136.15,227.27) and (137.85,225.56) .. (139.94,225.56) .. controls (142.04,225.56) and (143.74,227.27) .. (143.74,229.37) .. controls (143.74,231.47) and (142.04,233.17) .. (139.94,233.17) .. controls (137.85,233.17) and (136.15,231.47) .. (136.15,229.37) -- cycle ;

\draw (5.99,108.65) node [anchor=north west][inner sep=0.75pt]  [font=\tiny] [align=left] {$\displaystyle 1$};
\draw (32.21,67.99) node [anchor=north west][inner sep=0.75pt]  [font=\footnotesize] [align=left] {$\displaystyle \ast $};
\draw (65.56,108.79) node [anchor=north west][inner sep=0.75pt]  [font=\tiny] [align=left] {$\displaystyle 2$};
\draw (170.31,221.55) node [anchor=north west][inner sep=0.75pt]  [font=\tiny] [align=left] {$\displaystyle 1$};
\draw (197.44,181.8) node [anchor=north west][inner sep=0.75pt]  [font=\footnotesize] [align=left] {$\displaystyle \ast $};
\draw (228.98,222.59) node [anchor=north west][inner sep=0.75pt]  [font=\tiny] [align=left] {$\displaystyle 2$};
\draw (102.19,39.95) node [anchor=north west][inner sep=0.75pt]  [font=\tiny] [align=left] {$\displaystyle 1$};
\draw (129.32,0.2) node [anchor=north west][inner sep=0.75pt]  [font=\footnotesize] [align=left] {$\displaystyle \ast $};
\draw (162.67,40.99) node [anchor=north west][inner sep=0.75pt]  [font=\tiny] [align=left] {$\displaystyle 2$};
\draw (204.18,113.62) node [anchor=north west][inner sep=0.75pt]  [font=\tiny] [align=left] {$\displaystyle 1$};
\draw (234.02,71.14) node [anchor=north west][inner sep=0.75pt]  [font=\footnotesize] [align=left] {$\displaystyle \ast $};
\draw (267.37,109.22) node [anchor=north west][inner sep=0.75pt]  [font=\tiny] [align=left] {$\displaystyle 2$};
\draw (47.98,223.59) node [anchor=north west][inner sep=0.75pt]  [font=\tiny] [align=left] {$\displaystyle 1$};
\draw (75.1,182.63) node [anchor=north west][inner sep=0.75pt]  [font=\footnotesize] [align=left] {$\displaystyle \ast $};
\draw (106.64,223.43) node [anchor=north west][inner sep=0.75pt]  [font=\tiny] [align=left] {$\displaystyle 2$};
\draw (29.92,49.77) node [anchor=north west][inner sep=0.75pt]  [font=\tiny] [align=left] {$\displaystyle 1$};
\draw (49.76,18.26) node [anchor=north west][inner sep=0.75pt]  [font=\footnotesize] [align=left] {$\displaystyle \ast $};
\draw (77.23,49.88) node [anchor=north west][inner sep=0.75pt]  [font=\tiny] [align=left] {$\displaystyle 2$};
\draw (190.76,53.99) node [anchor=north west][inner sep=0.75pt]  [font=\tiny] [align=left] {$\displaystyle 1$};
\draw (211.36,24.47) node [anchor=north west][inner sep=0.75pt]  [font=\footnotesize] [align=left] {$\displaystyle \ast $};
\draw (235.71,54.74) node [anchor=north west][inner sep=0.75pt]  [font=\tiny] [align=left] {$\displaystyle 2$};
\draw (9.61,166.36) node [anchor=north west][inner sep=0.75pt]  [font=\tiny] [align=left] {$\displaystyle 1$};
\draw (28.79,135.23) node [anchor=north west][inner sep=0.75pt]  [font=\footnotesize] [align=left] {$\displaystyle \ast $};
\draw (55.43,166.46) node [anchor=north west][inner sep=0.75pt]  [font=\tiny] [align=left] {$\displaystyle 2$};
\draw (220.7,169.68) node [anchor=north west][inner sep=0.75pt]  [font=\tiny] [align=left] {$\displaystyle 1$};
\draw (243.56,138.24) node [anchor=north west][inner sep=0.75pt]  [font=\footnotesize] [align=left] {$\displaystyle \ast $};
\draw (268.19,170.49) node [anchor=north west][inner sep=0.75pt]  [font=\tiny] [align=left] {$\displaystyle 2$};
\draw (112.06,245.16) node [anchor=north west][inner sep=0.75pt]  [font=\tiny] [align=left] {$\displaystyle 1$};
\draw (134.23,214.56) node [anchor=north west][inner sep=0.75pt]  [font=\footnotesize] [align=left] {$\displaystyle \ast $};
\draw (158.16,245.95) node [anchor=north west][inner sep=0.75pt]  [font=\tiny] [align=left] {$\displaystyle 2$};

\end{tikzpicture}
\end{center}
Recall the augmented Bergman complex of $\mathsf{B}_2$ from Example \ref{ExAugBergB2}.

\begin{center}
\begin{tikzpicture}[x=0.75pt,y=0.75pt,yscale=-1,xscale=1]

\draw   (104.38,24.15) .. controls (104.28,22.55) and (105.47,21.17) .. (107.03,21.06) .. controls (108.59,20.96) and (109.94,22.17) .. (110.05,23.78) .. controls (110.15,25.38) and (108.96,26.76) .. (107.4,26.86) .. controls (105.83,26.97) and (104.48,25.75) .. (104.38,24.15) -- cycle ;
\draw   (152.77,73.7) .. controls (152.67,72.1) and (153.86,70.72) .. (155.42,70.62) .. controls (156.99,70.51) and (158.34,71.73) .. (158.44,73.33) .. controls (158.54,74.93) and (157.36,76.31) .. (155.79,76.42) .. controls (154.23,76.52) and (152.88,75.31) .. (152.77,73.7) -- cycle ;
\draw   (56.76,72.92) .. controls (56.66,71.32) and (57.84,69.93) .. (59.41,69.83) .. controls (60.97,69.73) and (62.32,70.94) .. (62.42,72.54) .. controls (62.52,74.14) and (61.34,75.52) .. (59.77,75.63) .. controls (58.21,75.73) and (56.86,74.52) .. (56.76,72.92) -- cycle ;
\draw   (57.53,121.68) .. controls (57.42,120.08) and (58.61,118.7) .. (60.17,118.59) .. controls (61.74,118.49) and (63.09,119.7) .. (63.19,121.31) .. controls (63.29,122.91) and (62.11,124.29) .. (60.54,124.39) .. controls (58.98,124.5) and (57.63,123.28) .. (57.53,121.68) -- cycle ;
\draw   (105.15,121.68) .. controls (105.05,120.08) and (106.23,118.7) .. (107.8,118.59) .. controls (109.36,118.49) and (110.71,119.7) .. (110.81,121.31) .. controls (110.92,122.91) and (109.73,124.29) .. (108.17,124.39) .. controls (106.6,124.5) and (105.25,123.28) .. (105.15,121.68) -- cycle ;
\draw    (105.15,25.73) -- (61.65,70.97) ;
\draw    (153.54,74.49) -- (110.05,119.73) ;
\draw    (108.57,25.78) -- (153.89,70.62) ;
\draw    (59.77,75.63) -- (60.17,118.59) ;
\draw    (105.15,121.68) -- (63.19,121.31) ;

\draw (93.1,10.52) node [anchor=north west][inner sep=0.75pt]  [font=\tiny,xslant=0.02] [align=left] {$\displaystyle \{2\} \leq \emptyset $};
\draw (161.2,66.33) node [anchor=north west][inner sep=0.75pt]  [font=\tiny,xslant=0.02] [align=left] {$\displaystyle \{1\} \leq \emptyset $};
\draw (103.89,125.47) node [anchor=north west][inner sep=0.75pt]  [font=\tiny,xslant=0.02] [align=left] {$\displaystyle \emptyset \leq \{\{1\}\}$};
\draw (7,67.33) node [anchor=north west][inner sep=0.75pt]  [font=\tiny,xslant=0.02] [align=left] {$\displaystyle \emptyset \leq \{\{2\}\}$};
\draw (20.28,120.96) node [anchor=north west][inner sep=0.75pt]  [font=\tiny,xslant=0.02] [align=left] {$\displaystyle \emptyset \leq \{\emptyset \}$};
\draw (0,93.71) node [anchor=north west][inner sep=0.75pt]  [font=\tiny,xslant=0.02] [align=left] {$\displaystyle \emptyset \leq \{\emptyset ,\{2\}\}$};
\draw (60.14,134.97) node [anchor=north west][inner sep=0.75pt]  [font=\tiny,xslant=0.02] [align=left] {$\displaystyle \emptyset \leq \{\emptyset ,\{1\}\}$};
\draw (128.9,36.84) node [anchor=north west][inner sep=0.75pt]  [font=\tiny,xslant=0.02] [align=left] {$\displaystyle \{1,2\} \leq \emptyset $};
\draw (25,39.36) node [anchor=north west][inner sep=0.75pt]  [font=\tiny,xslant=0.02] [align=left] {$\displaystyle \{2\} \leq \{\{2\}\}$};
\draw (133.75,100.11) node [anchor=north west][inner sep=0.75pt]  [font=\tiny,xslant=0.02] [align=left] {$\displaystyle \{1\} \leq \{\{1\}\}$};

\end{tikzpicture}
\end{center}

\end{exa}
Comparing the two diagrams leads us to a combinatorial proof of the following result\footnote{This is a weaker result than \cite[Proposition 2.6]{BHMPW2020} which states that the augmented Bergman fan of $\mathsf{B}_n$ is the normal fan of the stellohedron.}.
\begin{prop}\label{FaceStructrueBn}
The augmented Bergman fan (complex) of $\mathsf{B}_n$ is combinatorially equivalent to the dual stellohedron $\widetilde{\Pi}_n^*$. That is, their face lattices are isomorphic.
\end{prop}
\begin{proof}
Since the indices in (\ref{eq:NestedSetBn}) form a compatible pair. It is clear that
\[
    \{\{i\}\}_{i\in I}\cup\{S\cup\{*\}\}_{S\in\F} \longmapsto I\le\F
\]
for $I\subset [n]$ and flags $\F$ in $B_n$ gives a bijection between the nested set and the  compatible pairs of the Boolean matroid $\mathsf{B}_n$. 

Consider two nested sets $\left(\{\{i\}\}_{i\in I_1}\cup\{S\cup\{*\}\}_{S\in\F_1}\right)\subseteq\left(\{\{i\}\}_{i\in I_2}\cup\{S\cup\{*\}\}_{S\in\F_2}\right)$. This implies $I_1\subseteq I_2$ and $\F_1\subseteq\F_2$. That is, the corresponding cones satisfies $\sigma_{I_1\le\F_1}\subseteq \sigma_{I_2\le\F_2}$.

Therefore each face of $\partial\widetilde{\Pi}_n^*$ is represented by a nested set not containing $[n]\cup\{*\}$, and it corresponds to a compatible pair which represents a cone in the augmented Bergman fan of $\mathsf{B}_n$. This bijection respects the containments of faces and cones respectively. Therefore, it is a poset isomorphism between the two face lattices. 
\end{proof}

\begin{exa}
The following three nested sets with respect to $\B(K_{1,6})$ and their corresponding compatible pairs illustrate the bijection in different situations.
\[
\begin{tikzpicture}[x=0.75pt,y=0.75pt,yscale=-1,xscale=1]

\draw   (69.26,80.19) .. controls (69.22,79.45) and (69.77,78.81) .. (70.51,78.77) .. controls (71.25,78.72) and (71.88,79.28) .. (71.93,80.01) .. controls (71.98,80.75) and (71.42,81.38) .. (70.68,81.43) .. controls (69.95,81.48) and (69.31,80.92) .. (69.26,80.19) -- cycle ;
\draw    (71.86,80.73) -- (97.52,98.08) ;
\draw    (69.26,80.55) -- (43.52,98.3) ;
\draw    (70.51,78.77) -- (70.47,47.5) ;
\draw    (69.26,79.1) -- (43.06,62.29) ;
\draw    (71.57,79.29) -- (98.13,62.3) ;
\draw    (70.68,81.43) -- (70.53,113.56) ;
\draw   (97.89,61.66) .. controls (97.84,60.93) and (98.4,60.29) .. (99.13,60.24) .. controls (99.87,60.2) and (100.5,60.75) .. (100.55,61.49) .. controls (100.6,62.23) and (100.04,62.86) .. (99.31,62.91) .. controls (98.57,62.96) and (97.93,62.4) .. (97.89,61.66) -- cycle ;
\draw   (97.36,98.78) .. controls (97.31,98.05) and (97.87,97.41) .. (98.6,97.36) .. controls (99.34,97.31) and (99.98,97.87) .. (100.02,98.61) .. controls (100.07,99.34) and (99.51,99.98) .. (98.78,100.03) .. controls (98.04,100.08) and (97.41,99.52) .. (97.36,98.78) -- cycle ;
\draw   (40.55,61.41) .. controls (40.51,60.67) and (41.06,60.04) .. (41.8,59.99) .. controls (42.54,59.94) and (43.17,60.5) .. (43.22,61.23) .. controls (43.27,61.97) and (42.71,62.61) .. (41.97,62.65) .. controls (41.24,62.7) and (40.6,62.14) .. (40.55,61.41) -- cycle ;
\draw   (69.03,46.05) .. controls (68.98,45.32) and (69.54,44.68) .. (70.27,44.63) .. controls (71.01,44.59) and (71.64,45.14) .. (71.69,45.88) .. controls (71.74,46.62) and (71.18,47.25) .. (70.45,47.3) .. controls (69.71,47.35) and (69.08,46.79) .. (69.03,46.05) -- cycle ;
\draw   (69.28,114.98) .. controls (69.24,114.24) and (69.79,113.6) .. (70.53,113.56) .. controls (71.27,113.51) and (71.9,114.07) .. (71.95,114.8) .. controls (72,115.54) and (71.44,116.17) .. (70.7,116.22) .. controls (69.97,116.27) and (69.33,115.71) .. (69.28,114.98) -- cycle ;
\draw   (41.21,99.2) .. controls (41.17,98.46) and (41.72,97.83) .. (42.46,97.78) .. controls (43.2,97.73) and (43.83,98.29) .. (43.88,99.02) .. controls (43.93,99.76) and (43.37,100.4) .. (42.63,100.44) .. controls (41.9,100.49) and (41.26,99.93) .. (41.21,99.2) -- cycle ;
\draw  [color={rgb, 255:red, 144; green, 19; blue, 254 }  ,draw opacity=1 ] (64.4,46.54) .. controls (64.4,43.17) and (67.14,40.43) .. (70.52,40.43) .. controls (73.9,40.43) and (76.64,43.17) .. (76.64,46.54) .. controls (76.64,49.92) and (73.9,52.66) .. (70.52,52.66) .. controls (67.14,52.66) and (64.4,49.92) .. (64.4,46.54) -- cycle ;
\draw  [color={rgb, 255:red, 144; green, 19; blue, 254 }  ,draw opacity=1 ] (36.43,99.11) .. controls (36.43,95.73) and (39.17,92.99) .. (42.55,92.99) .. controls (45.93,92.99) and (48.66,95.73) .. (48.66,99.11) .. controls (48.66,102.49) and (45.93,105.23) .. (42.55,105.23) .. controls (39.17,105.23) and (36.43,102.49) .. (36.43,99.11) -- cycle ;
\draw  [color={rgb, 255:red, 144; green, 19; blue, 254 }  ,draw opacity=1 ] (59.4,57.3) .. controls (60.07,31.3) and (74.07,3.96) .. (78.73,32.63) .. controls (83.4,61.3) and (75.4,60.63) .. (111.4,49.96) .. controls (147.4,39.3) and (26.07,131.3) .. (22.73,111.96) .. controls (19.4,92.63) and (58.73,83.3) .. (59.4,57.3) -- cycle ;
\draw  [color={rgb, 255:red, 144; green, 19; blue, 254 }  ,draw opacity=1 ] (56.07,50.63) .. controls (56.73,24.63) and (63.73,14.3) .. (73.07,14.3) .. controls (82.4,14.3) and (80.3,41.26) .. (91.73,44.3) .. controls (103.17,47.33) and (140.4,39.63) .. (120.4,61.63) .. controls (100.4,83.63) and (83.73,84.3) .. (103.73,86.96) .. controls (123.73,89.63) and (109.89,118) .. (95.73,104.3) .. controls (81.58,90.6) and (25.07,129.63) .. (15.73,117.63) .. controls (6.4,105.63) and (55.4,76.63) .. (56.07,50.63) -- cycle ;
\draw  [color={rgb, 255:red, 144; green, 19; blue, 254 }  ,draw opacity=1 ] (50.73,52.3) .. controls (51.4,26.3) and (62.4,8.96) .. (71.73,8.96) .. controls (81.07,8.96) and (82.21,14.17) .. (86.4,22.96) .. controls (90.59,31.76) and (91.82,34.41) .. (96.4,35.63) .. controls (100.98,36.85) and (147.73,40.96) .. (127.73,62.96) .. controls (107.73,84.96) and (100.4,79.63) .. (119.07,84.3) .. controls (137.73,88.96) and (105.73,118.96) .. (89.07,122.3) .. controls (72.4,125.63) and (20.07,128.96) .. (10.73,116.96) .. controls (1.4,104.96) and (50.07,78.3) .. (50.73,52.3) -- cycle ;

\draw    (138.4,74.3) -- (183.07,74.3) ;
\draw [shift={(185.07,74.3)}, rotate = 180] [color={rgb, 255:red, 0; green, 0; blue, 0 }  ][line width=0.75]    (10.93,-3.29) .. controls (6.95,-1.4) and (3.31,-0.3) .. (0,0) .. controls (3.31,0.3) and (6.95,1.4) .. (10.93,3.29)   ;
\draw [shift={(136.4,74.3)}, rotate = 360] [color={rgb, 255:red, 0; green, 0; blue, 0 }  ][line width=0.75]    (10.93,-3.29) .. controls (6.95,-1.4) and (3.31,-0.3) .. (0,0) .. controls (3.31,0.3) and (6.95,1.4) .. (10.93,3.29)   ;

\draw (191.8,64.75) node [anchor=north west][inner sep=0.75pt]  [font=\small] [align=left] {$\displaystyle \{1,3\} \leq \{\{1,3,6\} ,\{1,3,5,6\} ,\{1,3,4,5,6\}\}$};
\draw (76.04,70.93) node [anchor=north west][inner sep=0.75pt]   [align=left] {$\displaystyle \ast $};
\draw (28.59,49.16) node [anchor=north west][inner sep=0.75pt]  [font=\scriptsize] [align=left] {$\displaystyle 2$};
\draw (100.83,51.76) node [anchor=north west][inner sep=0.75pt]  [font=\scriptsize] [align=left] {$\displaystyle 6$};
\draw (58.67,115.82) node [anchor=north west][inner sep=0.75pt]  [font=\scriptsize] [align=left] {$\displaystyle 4$};
\draw (65.12,27) node [anchor=north west][inner sep=0.75pt]  [font=\scriptsize] [align=left] {$\displaystyle 1$};
\draw (101.59,92.24) node [anchor=north west][inner sep=0.75pt]  [font=\scriptsize] [align=left] {$\displaystyle 5$};
\draw (28.59,97.06) node [anchor=north west][inner sep=0.75pt]  [font=\scriptsize] [align=left] {$\displaystyle 3$};

\end{tikzpicture}
\]

\[
\begin{tikzpicture}[x=0.75pt,y=0.75pt,yscale=-1,xscale=1]

\draw   (70.13,80.1) .. controls (70.08,79.37) and (70.64,78.73) .. (71.38,78.69) .. controls (72.11,78.64) and (72.75,79.2) .. (72.8,79.93) .. controls (72.84,80.67) and (72.29,81.3) .. (71.55,81.35) .. controls (70.81,81.4) and (70.18,80.84) .. (70.13,80.1) -- cycle ;
\draw    (72.72,80.65) -- (98.39,98) ;
\draw    (70.13,80.47) -- (44.38,98.22) ;
\draw    (71.38,78.69) -- (71.34,47.42) ;
\draw    (70.13,79.02) -- (43.92,62.21) ;
\draw    (72.43,79.21) -- (99,62.22) ;
\draw    (71.55,81.35) -- (71.4,113.48) ;
\draw   (98.75,61.58) .. controls (98.71,60.85) and (99.26,60.21) .. (100,60.16) .. controls (100.74,60.12) and (101.37,60.67) .. (101.42,61.41) .. controls (101.47,62.15) and (100.91,62.78) .. (100.17,62.83) .. controls (99.44,62.88) and (98.8,62.32) .. (98.75,61.58) -- cycle ;
\draw   (98.23,98.7) .. controls (98.18,97.96) and (98.74,97.33) .. (99.47,97.28) .. controls (100.21,97.23) and (100.84,97.79) .. (100.89,98.53) .. controls (100.94,99.26) and (100.38,99.9) .. (99.64,99.95) .. controls (98.91,99.99) and (98.27,99.44) .. (98.23,98.7) -- cycle ;
\draw   (41.42,61.33) .. controls (41.37,60.59) and (41.93,59.96) .. (42.67,59.91) .. controls (43.4,59.86) and (44.04,60.42) .. (44.09,61.15) .. controls (44.13,61.89) and (43.58,62.53) .. (42.84,62.57) .. controls (42.1,62.62) and (41.47,62.06) .. (41.42,61.33) -- cycle ;
\draw   (69.89,45.97) .. controls (69.85,45.24) and (70.4,44.6) .. (71.14,44.55) .. controls (71.88,44.51) and (72.51,45.06) .. (72.56,45.8) .. controls (72.61,46.54) and (72.05,47.17) .. (71.31,47.22) .. controls (70.58,47.27) and (69.94,46.71) .. (69.89,45.97) -- cycle ;
\draw   (70.15,114.89) .. controls (70.1,114.16) and (70.66,113.52) .. (71.4,113.48) .. controls (72.13,113.43) and (72.77,113.99) .. (72.82,114.72) .. controls (72.86,115.46) and (72.31,116.09) .. (71.57,116.14) .. controls (70.83,116.19) and (70.2,115.63) .. (70.15,114.89) -- cycle ;
\draw   (42.08,99.12) .. controls (42.03,98.38) and (42.59,97.75) .. (43.33,97.7) .. controls (44.06,97.65) and (44.7,98.21) .. (44.75,98.94) .. controls (44.79,99.68) and (44.24,100.32) .. (43.5,100.36) .. controls (42.76,100.41) and (42.13,99.85) .. (42.08,99.12) -- cycle ;
\draw  [color={rgb, 255:red, 144; green, 19; blue, 254 }  ,draw opacity=1 ] (65.27,80.46) .. controls (65.27,77.08) and (68.01,74.35) .. (71.38,74.35) .. controls (74.76,74.35) and (77.5,77.08) .. (77.5,80.46) .. controls (77.5,83.84) and (74.76,86.58) .. (71.38,86.58) .. controls (68.01,86.58) and (65.27,83.84) .. (65.27,80.46) -- cycle ;
\draw  [color={rgb, 255:red, 144; green, 19; blue, 254 }  ,draw opacity=1 ] (60.27,57.22) .. controls (60.93,31.22) and (74.93,3.88) .. (79.6,32.55) .. controls (84.27,61.22) and (76.27,60.55) .. (112.27,49.88) .. controls (148.27,39.22) and (64.4,112.29) .. (61.07,92.96) .. controls (57.73,73.63) and (59.6,83.22) .. (60.27,57.22) -- cycle ;
\draw  [color={rgb, 255:red, 144; green, 19; blue, 254 }  ,draw opacity=1 ] (56.93,50.55) .. controls (57.6,24.55) and (64.6,14.22) .. (73.93,14.22) .. controls (83.27,14.22) and (81.16,41.18) .. (92.6,44.22) .. controls (104.04,47.25) and (141.27,39.55) .. (121.27,61.55) .. controls (101.27,83.55) and (84.6,84.22) .. (104.6,86.88) .. controls (124.6,89.55) and (110.76,117.92) .. (96.6,104.22) .. controls (82.44,90.52) and (25.93,129.55) .. (16.6,117.55) .. controls (7.27,105.55) and (56.27,76.55) .. (56.93,50.55) -- cycle ;
\draw  [color={rgb, 255:red, 144; green, 19; blue, 254 }  ,draw opacity=1 ] (51.6,52.22) .. controls (52.27,26.22) and (63.27,8.88) .. (72.6,8.88) .. controls (81.93,8.88) and (83.08,14.09) .. (87.27,22.88) .. controls (91.46,31.68) and (92.68,34.33) .. (97.27,35.55) .. controls (101.85,36.77) and (148.6,40.88) .. (128.6,62.88) .. controls (108.6,84.88) and (101.27,79.55) .. (119.93,84.22) .. controls (138.6,88.88) and (106.6,118.88) .. (89.93,122.22) .. controls (73.27,125.55) and (20.93,128.88) .. (11.6,116.88) .. controls (2.27,104.88) and (50.93,78.22) .. (51.6,52.22) -- cycle ;
\draw    (139.27,74.22) -- (183.93,74.22) ;
\draw [shift={(185.93,74.22)}, rotate = 180] [color={rgb, 255:red, 0; green, 0; blue, 0 }  ][line width=0.75]    (10.93,-3.29) .. controls (6.95,-1.4) and (3.31,-0.3) .. (0,0) .. controls (3.31,0.3) and (6.95,1.4) .. (10.93,3.29)   ;
\draw [shift={(137.27,74.22)}, rotate = 360] [color={rgb, 255:red, 0; green, 0; blue, 0 }  ][line width=0.75]    (10.93,-3.29) .. controls (6.95,-1.4) and (3.31,-0.3) .. (0,0) .. controls (3.31,0.3) and (6.95,1.4) .. (10.93,3.29)   ;

\draw (192.67,64.67) node [anchor=north west][inner sep=0.75pt]  [font=\small] [align=left] {$\displaystyle \emptyset \leq \{\emptyset ,\{1,6\} ,\{1,3,5,6\} ,\{1,3,4,5,6\}\}$};
\draw (29.46,96.98) node [anchor=north west][inner sep=0.75pt]  [font=\scriptsize] [align=left] {$\displaystyle 3$};
\draw (102.45,92.16) node [anchor=north west][inner sep=0.75pt]  [font=\scriptsize] [align=left] {$\displaystyle 5$};
\draw (65.99,26.92) node [anchor=north west][inner sep=0.75pt]  [font=\scriptsize] [align=left] {$\displaystyle 1$};
\draw (59.54,115.74) node [anchor=north west][inner sep=0.75pt]  [font=\scriptsize] [align=left] {$\displaystyle 4$};
\draw (101.69,51.68) node [anchor=north west][inner sep=0.75pt]  [font=\scriptsize] [align=left] {$\displaystyle 6$};
\draw (29.46,49.08) node [anchor=north west][inner sep=0.75pt]  [font=\scriptsize] [align=left] {$\displaystyle 2$};
\draw (76.9,70.85) node [anchor=north west][inner sep=0.75pt]   [align=left] {$\displaystyle \ast $};

\end{tikzpicture}
\]

\[
\begin{tikzpicture}[x=0.75pt,y=0.75pt,yscale=-1,xscale=1]

\draw   (68.13,80.1) .. controls (68.08,79.37) and (68.64,78.73) .. (69.38,78.69) .. controls (70.11,78.64) and (70.75,79.2) .. (70.8,79.93) .. controls (70.84,80.67) and (70.29,81.3) .. (69.55,81.35) .. controls (68.81,81.4) and (68.18,80.84) .. (68.13,80.1) -- cycle ;
\draw    (70.72,80.65) -- (96.39,98) ;
\draw    (68.13,80.47) -- (42.38,98.22) ;
\draw    (69.38,78.69) -- (69.34,47.42) ;
\draw    (68.13,79.02) -- (41.92,62.21) ;
\draw    (70.43,79.21) -- (97,62.22) ;
\draw    (69.55,81.35) -- (69.4,113.48) ;
\draw   (96.75,61.58) .. controls (96.71,60.85) and (97.26,60.21) .. (98,60.16) .. controls (98.74,60.12) and (99.37,60.67) .. (99.42,61.41) .. controls (99.47,62.15) and (98.91,62.78) .. (98.17,62.83) .. controls (97.44,62.88) and (96.8,62.32) .. (96.75,61.58) -- cycle ;
\draw   (96.23,98.7) .. controls (96.18,97.96) and (96.74,97.33) .. (97.47,97.28) .. controls (98.21,97.23) and (98.84,97.79) .. (98.89,98.53) .. controls (98.94,99.26) and (98.38,99.9) .. (97.64,99.95) .. controls (96.91,99.99) and (96.27,99.44) .. (96.23,98.7) -- cycle ;
\draw   (39.42,61.33) .. controls (39.37,60.59) and (39.93,59.96) .. (40.67,59.91) .. controls (41.4,59.86) and (42.04,60.42) .. (42.09,61.15) .. controls (42.13,61.89) and (41.58,62.53) .. (40.84,62.57) .. controls (40.1,62.62) and (39.47,62.06) .. (39.42,61.33) -- cycle ;
\draw   (67.89,45.97) .. controls (67.85,45.24) and (68.4,44.6) .. (69.14,44.55) .. controls (69.88,44.51) and (70.51,45.06) .. (70.56,45.8) .. controls (70.61,46.54) and (70.05,47.17) .. (69.31,47.22) .. controls (68.58,47.27) and (67.94,46.71) .. (67.89,45.97) -- cycle ;
\draw   (68.15,114.89) .. controls (68.1,114.16) and (68.66,113.52) .. (69.4,113.48) .. controls (70.13,113.43) and (70.77,113.99) .. (70.82,114.72) .. controls (70.86,115.46) and (70.31,116.09) .. (69.57,116.14) .. controls (68.83,116.19) and (68.2,115.63) .. (68.15,114.89) -- cycle ;
\draw   (40.08,99.12) .. controls (40.03,98.38) and (40.59,97.75) .. (41.33,97.7) .. controls (42.06,97.65) and (42.7,98.21) .. (42.75,98.94) .. controls (42.79,99.68) and (42.24,100.32) .. (41.5,100.36) .. controls (40.76,100.41) and (40.13,99.85) .. (40.08,99.12) -- cycle ;
\draw  [color={rgb, 255:red, 144; green, 19; blue, 254 }  ,draw opacity=1 ] (58.27,57.22) .. controls (58.93,31.22) and (72.93,3.88) .. (77.6,32.55) .. controls (82.27,61.22) and (74.27,60.55) .. (110.27,49.88) .. controls (146.27,39.22) and (62.4,112.29) .. (59.07,92.96) .. controls (55.73,73.63) and (57.6,83.22) .. (58.27,57.22) -- cycle ;
\draw  [color={rgb, 255:red, 144; green, 19; blue, 254 }  ,draw opacity=1 ] (54.93,50.55) .. controls (55.6,24.55) and (62.6,14.22) .. (71.93,14.22) .. controls (81.27,14.22) and (79.16,41.18) .. (90.6,44.22) .. controls (102.04,47.25) and (139.27,39.55) .. (119.27,61.55) .. controls (99.27,83.55) and (82.6,84.22) .. (102.6,86.88) .. controls (122.6,89.55) and (108.76,117.92) .. (94.6,104.22) .. controls (80.44,90.52) and (23.93,129.55) .. (14.6,117.55) .. controls (5.27,105.55) and (54.27,76.55) .. (54.93,50.55) -- cycle ;
\draw  [color={rgb, 255:red, 144; green, 19; blue, 254 }  ,draw opacity=1 ] (49.6,52.22) .. controls (50.27,26.22) and (61.27,8.88) .. (70.6,8.88) .. controls (79.93,8.88) and (81.08,14.09) .. (85.27,22.88) .. controls (89.46,31.68) and (90.68,34.33) .. (95.27,35.55) .. controls (99.85,36.77) and (146.6,40.88) .. (126.6,62.88) .. controls (106.6,84.88) and (99.27,79.55) .. (117.93,84.22) .. controls (136.6,88.88) and (104.6,118.88) .. (87.93,122.22) .. controls (71.27,125.55) and (18.93,128.88) .. (9.6,116.88) .. controls (0.27,104.88) and (48.93,78.22) .. (49.6,52.22) -- cycle ;
\draw    (137.27,74.22) -- (181.93,74.22) ;
\draw [shift={(183.93,74.22)}, rotate = 180] [color={rgb, 255:red, 0; green, 0; blue, 0 }  ][line width=0.75]    (10.93,-3.29) .. controls (6.95,-1.4) and (3.31,-0.3) .. (0,0) .. controls (3.31,0.3) and (6.95,1.4) .. (10.93,3.29)   ;
\draw [shift={(135.27,74.22)}, rotate = 360] [color={rgb, 255:red, 0; green, 0; blue, 0 }  ][line width=0.75]    (10.93,-3.29) .. controls (6.95,-1.4) and (3.31,-0.3) .. (0,0) .. controls (3.31,0.3) and (6.95,1.4) .. (10.93,3.29)   ;

\draw (190.67,64.67) node [anchor=north west][inner sep=0.75pt]  [font=\small] [align=left] {$\displaystyle \emptyset \leq \{\{1,6\} ,\{1,3,5,6\} ,\{1,3,4,5,6\}\}$};
\draw (27.46,96.98) node [anchor=north west][inner sep=0.75pt]  [font=\scriptsize] [align=left] {$\displaystyle 3$};
\draw (100.45,92.16) node [anchor=north west][inner sep=0.75pt]  [font=\scriptsize] [align=left] {$\displaystyle 5$};
\draw (63.99,26.92) node [anchor=north west][inner sep=0.75pt]  [font=\scriptsize] [align=left] {$\displaystyle 1$};
\draw (57.54,115.74) node [anchor=north west][inner sep=0.75pt]  [font=\scriptsize] [align=left] {$\displaystyle 4$};
\draw (99.69,51.68) node [anchor=north west][inner sep=0.75pt]  [font=\scriptsize] [align=left] {$\displaystyle 6$};
\draw (27.46,49.08) node [anchor=north west][inner sep=0.75pt]  [font=\scriptsize] [align=left] {$\displaystyle 2$};
\draw (74.9,70.85) node [anchor=north west][inner sep=0.75pt]   [align=left] {$\displaystyle \ast $};

\end{tikzpicture}
\]

\end{exa}
We show in Section \ref{sec:generalization} that the isomorphism in Proposition \ref{FaceStructrueBn} holds even when the lattice is not a Boolean lattice. 

\subsection{Generalization to any matroid $M$} \label{sec:generalization}
Let $M$, $\L(M)$, $\I(M)$ be defined as in Section \ref{Sec:AugChowBn}. The collection $\I(M)$ forms an abstract simplicial complex called the \emph{independence complex}; here we identify $\I(M)$ with the face lattice of the independence complex. We construct a new poset $\widetilde{\mathcal{L}}(M)$ from $\mathcal{L}(M)$ and $\mathcal{I}(M)$ in the following way:
\begin{itemize}
\item As a set, $\widetilde{\mathcal{L}}(M)=\mathcal{I}(M)\uplus\mathcal{L}(M)$ in which we write $F\in\L(M)$ as  $F_*$.
\item For $I\in\mathcal{I}(M)$, define the cover relation $I\lessdot \mathsf{cl}_M(I)_*$, where $\mathsf{cl}_M(I)$ is the closure of $I$ in $M$.
The relations inside $\I(M)$ and $\L(M)$ stay the same.  
\end{itemize}

\begin{exa}
Consider the uniform matroid $U_{2,3}$, then the new poset is 
\begin{center}
\begin{tikzpicture}[x=0.75pt,y=0.75pt,yscale=-1.2,xscale=1.2]

\draw    (183.23,39.99) -- (210.56,17.77) ;
\draw    (215.42,18.19) -- (215.42,38.47) ;
\draw    (220.37,17.32) -- (247.11,41.5) ;
\draw    (182.73,51.32) -- (210.07,71.42) ;
\draw    (215.91,53.59) -- (215.91,69.45) ;
\draw    (221.46,69.91) -- (248.6,51.32) ;
\draw    (92.27,68.46) -- (92.27,50.23) ;
\draw    (164.08,69.07) -- (164.08,55.57) -- (164.08,50.23) ;
\draw    (96.28,50.11) -- (124.01,69.56) ;
\draw    (95.73,68.46) -- (123.07,50.59) ;
\draw    (132.88,50.23) -- (159.62,69.68) ;
\draw    (132.88,67.86) -- (159.13,50.23) ;
\draw    (95.24,77.58) -- (122.58,93.75) ;
\draw    (128.42,79.4) -- (128.42,92.17) ;
\draw    (133.97,92.53) -- (161.11,77.58) ;
\draw [color={rgb, 255:red, 144; green, 19; blue, 254 }  ,draw opacity=1 ]   (211.51,75.79) -- (133.3,96.02) ;
\draw [color={rgb, 255:red, 144; green, 19; blue, 254 }  ,draw opacity=1 ]   (245.91,48.43) -- (215.12,59.03) -- (167.7,71.79) ;
\draw [color={rgb, 255:red, 144; green, 19; blue, 254 }  ,draw opacity=1 ]   (210.98,48.91) -- (180.19,58.32) -- (133.85,72.27) ;
\draw [color={rgb, 255:red, 144; green, 19; blue, 254 }  ,draw opacity=1 ]   (176.79,49.63) -- (146,59.04) -- (97.49,71.8) ;
\draw [color={rgb, 255:red, 144; green, 19; blue, 254 }  ,draw opacity=1 ]   (207.53,15.91) -- (94.92,42.44) ;
\draw [color={rgb, 255:red, 144; green, 19; blue, 254 }  ,draw opacity=1 ]   (206.44,17.5) -- (128.59,42.44) ;
\draw [color={rgb, 255:red, 144; green, 19; blue, 254 }  ,draw opacity=1 ]   (210.43,17.06) -- (162.99,41.29) ;

\draw (209.08,9.39) node [anchor=north west][inner sep=0.75pt]  [font=\tiny] [align=left] {$\displaystyle 123_{*}$};
\draw (178.26,43.13) node [anchor=north west][inner sep=0.75pt]  [font=\tiny] [align=left] {$\displaystyle 1_{*}$};
\draw (212.92,43.31) node [anchor=north west][inner sep=0.75pt]  [font=\tiny] [align=left] {$\displaystyle 2_{*}$};
\draw (247.08,43.61) node [anchor=north west][inner sep=0.75pt]  [font=\tiny] [align=left] {$\displaystyle 3_{*}$};
\draw (213.03,68.96) node [anchor=north west][inner sep=0.75pt]  [font=\tiny] [align=left] {$\displaystyle \emptyset_{*}$};
\draw (88.21,43.07) node [anchor=north west][inner sep=0.75pt]  [font=\tiny] [align=left] {$\displaystyle 12$};
\draw (122.41,42.13) node [anchor=north west][inner sep=0.75pt]  [font=\tiny] [align=left] {$\displaystyle 13$};
\draw (156.85,41.98) node [anchor=north west][inner sep=0.75pt]  [font=\tiny] [align=left] {$\displaystyle 23$};
\draw (89.68,71.75) node [anchor=north west][inner sep=0.75pt]  [font=\tiny] [align=left] {$\displaystyle 1$};
\draw (125.43,71.53) node [anchor=north west][inner sep=0.75pt]  [font=\tiny] [align=left] {$\displaystyle 2$};
\draw (159.59,70.66) node [anchor=north west][inner sep=0.75pt]  [font=\tiny] [align=left] {$\displaystyle 3$};
\draw (124.79,96.41) node [anchor=north west][inner sep=0.75pt]  [font=\tiny] [align=left] {$\displaystyle \emptyset $};
\draw (10,46) node [anchor=north west][inner sep=0.75pt]   [align=left] {$\widetilde{\mathcal{L}}(U_{2,3}) =$};

\end{tikzpicture}.
\end{center}
\end{exa}

From its construction, $\widetilde{\L}(M)$ is clearly a graded, atomic lattice. Let $\widetilde{\rk}$ be the rank function of $\widetilde{\L}(M)$ and let $\L(M)_*$ be the copy of $\L(M)$ in $\widetilde{\L}(M)$.

\begin{prop}
The poset $\widetilde{\L}(M)$ is a geometric lattice.     
\end{prop}
\begin{proof}
We show that $\widetilde{\rk}$ satisfies the upper semimodular inequality:
\[
    \widetilde{\rk}(S)+\widetilde{\rk}(T)\ge \widetilde{\rk}(S\vee T)+\widetilde{\rk}(S\wedge T) \quad\text{ for }S, T\in\widetilde{\L}(M).
\]
It suffices to consider the cases (i) $S\in \I(M)$ and $T\in\L(M)_*$ and (ii) $S, T\in \I(M)$. Note that for any $I\in\I(M)$, we have 
\begin{equation}\label{eq:rkInd}
    \widetilde{\rk}(I)=|I|=\rk_M(\mathsf{cl}_M(I)). 
\end{equation}
\begin{itemize}
    \item[(i)] Given $I\in\I(M)$ and $F_*$ with $F\in\L(M)$, by (\ref{eq:rkInd}) and semimodularity of $\L(M)$, we have
    \begin{align}
        \widetilde{\rk}(I)+\widetilde{\rk}(F_*)
        &=\rk_M(\mathsf{cl}_M(I))+\rk_M(F)+1  \nonumber\\
        &\ge \rk_M(\mathsf{cl}_M(I)\vee F)+\rk_M(\mathsf{cl}_M(I)\wedge F)+1  
        \nonumber 
    \end{align}
    Since $\rk_M(\mathsf{cl}_M(I)\vee F)+1=\widetilde{\rk}({\mathsf{cl}_M(I)}_*\vee F_*)$ and $\mathsf{cl}_M(I)\wedge F=\mathsf{cl}_M(I\cap F)$, by (\ref{eq:rkInd}) the above inequality can be rewritten as
    \[
        \widetilde{\rk}(I)+\widetilde{\rk}(F_*)\ge \widetilde{\rk}({\mathsf{cl}_M(I)}_*\vee F_*)+\widetilde{\rk}(I\cap F)=\widetilde{\rk}(I\vee F_*)+\widetilde{\rk}(I\wedge F_*).
    \]
        
    \item[(ii)] Given $I_1, I_2\in\I(M)$, then we have
\begin{equation}\label{eq:TwoIndep}
    \widetilde{\rk}(I_1)+\widetilde{\rk}(I_2)-\widetilde{\rk}(I_1\wedge I_2)=|I_1|+|I_2|-|I_1\cap I_2|=|I_1\cup I_2|.
\end{equation}
If $I_1\cup I_2\in\I(M)$, then $\widetilde{\rk}(I_1\vee I_2)=|I_1\cup I_2|$, so the upper semimodular inequality holds. Otherwise, if $I_1\cup I_2\notin\I(M)$, then 
\begin{equation} \label{eq:TwoIndepJoin}
\widetilde{\rk}(I_1\vee I_2)=\widetilde{\rk}({\mathsf{cl}_M(I_1)}_*\vee{\mathsf{cl}_M(I_2)}_*)=\rk_M(\mathsf{cl}_M(I_1)\vee\mathsf{cl}_M(I_2))+1
\end{equation}
On the other hand, since $I_1\cup I_2$ is not independent and $\mathsf{cl}_M(I_1\cup I_2)=\mathsf{cl}_M(I_1)\vee\mathsf{cl}_M(I_2)$, we have
\[
    |I_1\cup I_2|\ge \rk_M(\mathsf{cl}_M(I_1\cup I_2))+1=\rk_M(\mathsf{cl}_M(I_1)\vee\mathsf{cl}_M(I_2))+1.
\]
Combining this inequality with (\ref{eq:TwoIndep}) and (\ref{eq:TwoIndepJoin}) gives the upper semimodular inequality.

\end{itemize}

\end{proof}
\begin{lem}
The set $\widetilde{\G}=\{\{1\},\ldots,\{n\}\}\cup\{F_{*}\}_{F\in\L(M)}$ is a building set of $\widetilde{\L}(M)$.    
\end{lem}
\begin{proof}
Pick any element from $\widetilde{\L}(M)-\{\emptyset\}$. If the element is some $I\in \I(M)\subset\widetilde{\L}(M)$, then since $\mathcal{I}(M)$ is the face lattice of a simplicial complex, it follows that the interval $[\hat{0},I]\cong \prod_{i\in I}[\hat{0},\{i\}]$ is built from $\{i\}$ in $\widetilde{\G}$ with $i\in I$. On the other hand, if the element is $F_*$ for some flat $F\in\L(M)$, then since $\max(\widetilde{\G}_{\le F_*})=\{F_*\}$, the interval $[\hat{0},F_*]\cong\prod_{i=1}^1[\hat{0},F_*]$ is built from $F_*$ itself.
\end{proof}

In what follows, we take $\widetilde{\G}=\{\{1\},\ldots,\{n\}\}\cup\{F_{*}\}_{F\in\L(M)}$ as the building set of $\widetilde{\L}(M)$.

\begin{lem} \label{AugNestedSet}
The nested sets with respect to $\widetilde{\G}$ are of the form 
\[
	\left\{\{i\}\right\}_{i\in I}\cup\left\{F_{*}\right\}_{F\in\mathcal{F}}
\]
for some compatible pair $I\le\mathcal{F}$ where $I\in\mathcal{I}(M)$ and $\mathcal{F}$ is a flag of arbitrary flats.
\end{lem}

\begin{proof}
Let $N\subset \widetilde{\G}$ be a nested set for $\widetilde{\G}$. First notice that if $\{i_1\},\ldots, \{i_\ell\}$ are all the singletons in $N$ that come from $\I(M)$ then it must be that $I=\{i_1,\ldots, i_\ell\}$ is an independent subset of $M$. i.e $I\in\I(M)$. Otherwise, if $I$ is not independent, then the join $\bigvee_{j=1}^\ell\{i_j\}=\mathsf{cl}_M(I)_*\in \widetilde{\G}$, which violates the definition of a nested set.

Now suppose that $F$ is a flat such that $F_*$ is in $N$, then we must have $F\supseteq I$. If not, say $i_j\notin F$, then $\{i_j\}$ and $F_*$ are incomparable in $\widetilde{\L}(M)$, so their join $F_*\vee\{i_j\}\in \L(M)_*\subset\widetilde{\G}$, which again violates the definition of a nested set.

Lastly, because $\widetilde{\G}$ contains $\L(M)_*$, any ${F_1}_*$ and ${F_2}_*$ such that $F_1$ and $F_2$ are incomparable in $\L(M)$ must have the join in $\widetilde{\G}$. Hence the flats appearing in $N$ form a flag of flats.

In conclusion, the nested set must be of the form $\{\{i\}\}_{i\in I}\cup\{F_*\}_{F\in\F}$ for some independent set $I$ and some flag $\F$ that are compatible with $I$.
\end{proof}

With this lemma, we can recover the augmented Bergman fan $\widetilde{\Sigma}_M$ as the reduced nested set complex with respect to the building set $\widetilde{\G}$ as shown in the following theorem.
\begin{thm} \label{thm:AugFace}
The face lattice of the reduced nested set complex $\widetilde{\N}(\widetilde{\mathcal{L}}(M),\widetilde{\G})$ is isomorphic to the face lattice of the augmented Bergman fan (complex) of $M$. The poset isomorphism between their face lattices is given by
\[
	\left\{\{i\}\right\}_{i\in I}\cup\left\{F_{*}\right\}_{F\in\mathcal{F}}\longleftrightarrow \sigma_{I\le\mathcal{F}}
\]
for compatible pair $I\le\mathcal{F}$ where $I\in\mathcal{I}(M)$ and flag $\mathcal{F}\subset\L(M)-\{[n]\}$ of \emph{proper} flats.
\end{thm}

Since $\widetilde{A}(M)$ is the Chow ring corresponding to the augmented Bergman fan $\widetilde{\Sigma}_M$, and $D(\widetilde{\L}(M),\G)$ is the Chow ring corresponding to the fan $\Sigma(\widetilde{\L}(M),\widetilde{\G})$ (see Remark \ref{rem:FYChowRing}) whose face structure is the same as $\widetilde{\N}(\widetilde{\mathcal{L}}(M),\widetilde{\G})$, Theorem \ref{thm:AugFace} suggests that we might be able to realize the augmented Chow ring $\widetilde{A}(M)$ as the Feichtner-Yuzvinsky Chow ring $D(\L,\G)$. The following theorem states that this is indeed the case.\footnote{See Remark \ref{rem:Eur}.}

\begin{thm}  \label{thm:AugIsChow}
The Feichtner-Yuzvinsky Chow ring $D(\widetilde{\L}(M),\widetilde{\G})$ is isomorphic to the augmented Chow ring $\widetilde{A}(M)$ for any matroid $M$.
\end{thm}

\begin{proof}
Recall from Definition \ref{ChowRingLattice} that the Chow ring $D(\widetilde{\L}(M),\widetilde{\G})$ is defined as 
\[
    D(\widetilde{\L}(M),\widetilde{\G})=S/(I+J)
\]
where $S=\mathbb{Q}[x_G]$ is a polynomial ring with one variable $x_G$ for each element $G\in\widetilde{\G}$, and where $I$, $J$ are the following ideals of $S$:
\begin{itemize}
    \item $I$ is the ideal of $S$ generated by nonface monomials $\prod_{i=1}^t x_{G_i}$ for $\{G_1,\ldots,G_t\}\notin\N(\widetilde{\L}(M),\widetilde{\G})$,
    \item $J$ is the ideal of $S$ generated by the linear element $\sum_{G\ge a}x_G$ for each atom $a$ of $\widetilde{\L}(M)$.
\end{itemize}

To match with the definition of $\widetilde{A}(M)$, we set $x_G=x_F$ for $G=F_*\in\widetilde{\G}$ and $x_G=y_i$ for $G=\{i\}\in\widetilde{\G}$ so that $S=\mathbb{Q}\left[\{x_F\}_{F\in\mathcal{F}(M)}\cup\{y_i\}_{i\in [n]}\right]$. By Lemma \ref{AugNestedSet}, the nested set $N\in \N(\widetilde{\L}(M),\widetilde{\G})$ is of the form 
\[
    N=\{\{i\}\}_{i\in T}\cup\{F_*\}_{F\in\mathcal{F}}
\]
for $T\in\I(M)$ and flag $\F\in\L(M)$ such that $T\le\F$. Hence, the ideal $I$ is the sum of two ideals $I_1+I_2$ where $I_1$ is generated by $x_Fx_{F'}$ such that $F$ and $F'$ are incomparable in $\L(M)$, and where $I_2$ is generated by $y_ix_F$ for $\{i\}$ not compatible with $\{F\}$, i.e. $i\notin F$.

Consider the linear elements that generate the ideal $J$. Recall the atoms of $\widetilde{\L}(M)$ are $\{1\},\ldots,\{n\}$ and $\emptyset_*$. For the atoms $a=\{i\}$ and $a=\emptyset_*$, the linear elements are
\[
    \sum_{G\ge\{i\}}x_G=y_i+\sum_{F:i\in F}x_F \quad\text{and}\quad \sum_{G\ge\emptyset_*}x_G=\sum_{F\in\L(M)}x_F \quad\text{respectively.}
\]
Hence, $J=\left\langle y_i+\sum_{F:i\in F}x_F\right\rangle_{i\in [n]}+\left\langle\sum_{F\in\F(M)}x_F\right\rangle$. 

Then
\begin{align*}
	D(\widetilde{\mathcal{L}}(M),\widetilde{\G})
        &=\frac{\mathbb{Q}\left[\{x_F\}_{F\in\mathcal{F}(M)}\cup\{y_i\}_{i\in [n]}\right]/\left(I_1+I_2+\left\langle y_i+\sum_{F:i\in F}x_F\right\rangle_{i\in[n]}\right)}{\left\langle\sum_{F\in\F(M)}x_F\right\rangle}\\
	&\cong \mathbb{Q}\left[\{x_F\}_{F\in\mathcal{F}(M)\setminus [n]}\cup\{y_i\}_{i\in [n]}\right]/\left(I_1+I_2+\langle y_i-\sum_{F:i\notin F}x_F\rangle_{i\in[n]}\right)\\
	&=\widetilde{A}(M)
\end{align*}
where the isomorphism is obtained by eliminating $x_{[n]}$ via expressing $x_{[n]}$ as $-\sum_{[n]\neq F\in\F(M)}x_F$ to obtain
\[
    y_i+\sum_{F:i\in F}x_F=y_i+\sum_{F:i\in F\neq [n]}x_F-\sum_{F\in\F(M)\setminus\{[n]\}}x_F=y_i-\sum_{F:i\notin F}x_F.
\]
\end{proof}
 
Since the augmented Chow ring is a Feichtner-Yuzvinsky Chow ring, we can apply Proposition \ref{FYbasis} to obtain a basis for $\widetilde{A}(M)$. 
\begin{cor}\label{AugBasis}
The augmented Chow ring $\widetilde{A}(M)$ of $M$ has the following basis
\[
	\widetilde{FY}(M)\coloneqq\left\{x_{F_1}^{a_1}x_{F_2}^{a_2}\ldots x_{F_k}^{a_k}: \substack{\emptyset\subsetneq F_1\subsetneq F_2\subsetneq\ldots\subsetneq F_k \\
    1\le a_1\le\rk_M(F_1),~ a_i\le\rk_M(F_i)-\rk_M(F_{i-1})-1 \text{ for }i\ge 2}
    \right\}.
\]
where $\rk_M$ is the rank function of the matroid $M$.
\end{cor}
\begin{proof} 
Recall that we denote $\rk_M$ as the rank function of the matroid $M$ and $\widetilde{\rk}$ as the rank function of the lattice $\widetilde{\L}(M)$. Following from Theorem \ref{thm:AugIsChow}, we have $\widetilde{A}(M)=D(\widetilde{\L}(M),\widetilde{\G})$. Then Theorem \ref{FYbasis} gives a basis for the augmented Chow ring $\widetilde{A}(M)$ that is constructed as follows. For each nested set $N=\{\{i\}\}_{i\in I}\cup\{{F_1}_*, \ldots, {F_k}_*\}$ (see Lemma \ref{AugNestedSet}), we get basis elements of the form $\prod_{i\in I}y_i^{b_i}x_{F_1}^{a_1}x_{F_2}^{a_2}\ldots x_{F_k}^{a_k}$ whose exponents $b_i$ for $i\in I$ and $a_1,\ldots, a_k$ are restricted as follows:
 
\begin{itemize}
\item[(i)] The exponent $b_i<\widetilde{\rk}(\{i\})-\widetilde{\rk}(\{i\}')$ where $\{i\}'$ is the join $\bigvee N\cap\widetilde{\L}(M)_{<\{i\}}$ in $\widetilde{\L}(M)$. Since $N\cap\widetilde{\L}(M)_{<\{i\}}$ is empty, we have $\{i\}'=\emptyset$ and therefore
\[
    b_i<\widetilde{\rk}(\{i\})-\widetilde{\rk}(\emptyset)=1-0=1.
\]
Then $b_i=0$ for all $i\in I$. This means the basis elements are of the form $x_{F_1}^{a_1}x_{F_2}^{a_2}\ldots x_{F_k}^{a_k}$.
\item[(ii)] The exponent $a_1<\widetilde{\rk}({F_1}_*)-\widetilde{\rk}({F_1}_*')$ where 
\[
    {F_1}_*'=\bigvee(N\cap\widetilde{\L}(M)_{<{F_1}_*})=\bigvee\{\{i\}:~i\in I\}=I.
\]
Then we have
\[
    a_1<\widetilde{\rk}({F_1}_{*})-\widetilde{\rk}(I)=\rk_M(F_1)+1-\rk_M(I)=\rk_M(F_1)+1-|I|.
\]
Consider any independent subset $I$ compatible with the fixed flag $F_1\subsetneq\ldots\subsetneq F_k$, by (i) the corresponding basis elements are of the form $x_{F_1}^{a_1}x_{F_2}^{a_2}\ldots x_{F_k}^{a_k}$ with $a_1<\rk_M(F_1)+1-|I|$. Hence we may take $I=\emptyset$ and obtain
\[
    a_1<\rk_M(F_1)+1.
\]
\item[(iii)] The exponent $a_i<\widetilde{\rk}({F_i}_*)-\widetilde{\rk}({F_i}_*')$ ($i\ge 2$) where
\[
    {F_i}_*'=\bigvee(N\cap\widetilde{\L}(M)_{<{F_i}_*})={F_{i-1}}_*.
\]
Then we have
\[
    a_i<\widetilde{\rk}({F_i}_*)-\widetilde{\rk}({F_{i-1}}_*)=(\rk_M(F_i)+1)-(\rk_M(F_{i-1})+1)=\rk_M(F_i)-\rk_M(F_{i-1}).
\]
\end{itemize}
Combining these three cases gives us the expression of the basis.
\end{proof}

\begin{rem} \label{rem:Eur}
After this work was completed, we learned from  \cite[Setion 5.1]{Mastroeni2022Koszul} that Theorem \ref{thm:AugIsChow} had also been discovered independently by Eur, who further noticed that $\widetilde{\L}(M)$ is the lattice of flats of the \emph{free coextension} $(M^*+e)^*$ of $M$. Theorem \ref{thm:AugIsChow} and Corollary \ref{AugBasis} were later included in Eur, Huh, and Larson's paper \cite[Lemma 5.14, Section 7.2]{EHL2022stellahedral}. See \cite{Mastroeni2022Koszul}, \cite{Ferroni2022hilbert} for further discussion of this construction.
\end{rem}

\section{Combinatorics of the two stories} \label{Sec:CombTwoStories}

\subsection{Background on permutations and symmetric functions} \label{subsec:Permutations and Symmetric functions}
For a positive integer $n$, define  $[n]_q\coloneqq 1+q+\ldots+q^{n-1}$ and $[n]_q!=\prod_{i=1}^n [i]_q$. Then for any $k_1+\ldots+k_m=n$, the \emph{$q$-multinomial coefficient} is defined to be
\[
    \qbin{n}{k_1,k_2,\ldots,k_m}=\frac{[n]_q!}{[k_1]_q![k_2]_q!\ldots[k_m]_q!}.
\]
In particular, the special case $\qbin{n}{k}\coloneqq\qbin{n}{k,n-k}$ is the \emph{$q$-binomial coefficient}.

Let $(\mathcal{A}, \prec)$ be a set equipped with a total order $\prec$. For a word $w=w_1w_2\ldots w_n$ over $\mathcal{A}$, one can define the following statistics:
\begin{itemize}
    \item \emph{Descent set} $\DES(w)\coloneqq\{i\in[n-1]: w_i\succ w_{i+1}\}$ and \emph{descent number} $\des(w)\coloneqq |\DES(w)|$.
    \item \emph{Inversion number} $\inv(w)\coloneqq |\{(i,j)\in[n]\times[n]: i<j,~w_i\succ w_j\}|$.
    \item \emph{Major index} $\maj(w)\coloneqq \sum_{i\in\DES(w)}i$.
\end{itemize}

A classical result of MacMahon states that statistics $\inv$ and $\maj$ are equidistributed over all permutations of a given multiset and have the distribution generating function as follows.
\begin{thm}[\cite{MacMahon1916}] \label{Multi_inv_maj}
Let $M$ be a multiset $\{a_1^{k_1},a_2^{k_2},\ldots,a_m^{k_m}\}$ with $a_1\lneq a_2\lneq\ldots\lneq a_m$ in  $\mathcal{A}$. Let $\S_M$ denote the set of permutations of $M$. Then 
\[
    \sum_{\sigma\in\S_M}q^{\inv(\sigma)}=\sum_{\sigma\in\S_M}q^{\maj(\sigma)}={k_1+\ldots+k_m \brack k_1,\ldots,k_m}_q.
\]
\end{thm}
It is well known and easy to see that Theorem \ref{Multi_inv_maj} holds if for all $\sigma\in\S_M$ we replace $a_i^{k_i}$ in $\sigma$ by a weakly increasing word $W_i$ of length $k_i$ so that all the letters in $W_i$ are less than those in $W_j$ whenever $i<j$. We call the resulting words shuffles of $W_1,\ldots, W_m$. More precisely, we say a word $\sigma$ of length $n+m$ is a \emph{shuffle} of the word $W_1$ of length $n$ and the word $W_2$ of length $m$, denoted by $\sigma\in W_1\shuffle W_2$, if $W_1$ and $W_2$ are subsequences of $\sigma$.
\begin{cor} \label{IncreasingShuffle}
Let $W_i$ be a weakly increasing word of length $k_i$ over $\mathcal{A}$ for $1\le i\le m$. If the letters in $W_i$ are all less than the letters in $W_{i+1}$ under the total order on $\mathcal{A}$ for all $1\le i\le m-1$, then
\[
    \sum_{\sigma\in W_1\shuffle \ldots\shuffle W_m}q^{\inv(\sigma)}=\sum_{\sigma\in W_1\shuffle \ldots\shuffle W_m}q^{\maj(\sigma)}={k_1+\ldots+k_m \brack k_1,\ldots,k_m}_q.
\]
where $W_1\shuffle\ldots\shuffle W_m$ are the set of all shuffles of words $W_1,W_2,\ldots, W_m$.
\end{cor}
Let $\sigma\in W_1\shuffle\ldots\shuffle W_m$. For each $i$, we permute the letters in $W_i$ and obtain $W_i'$ so that $W_i'$ is not necessarily weakly increasing. Call the resulting word $\sigma'\in W_1'\shuffle\ldots\shuffle W_m'$. Since the letters in $W_i$ are less than the letters in $W_{i+1}$ for every $i$, the inversions of $\sigma$ between $W_i$ and $W_j$ for $i<j$ stay the same after we replace $W_i$ with $W_i'$ and $W_j$ with $W_j'$. But the inversions within $W_i$ increase from $\inv(W_i)=0$ to $\inv(W_i')$. This suggests the following well-known extension of Corollary \ref{IncreasingShuffle} for $\inv$.

\begin{cor} \label{ShuffleINV}
Let $W_1, \ldots, W_m$ be words of length $k_1, \ldots, k_m$ over $\mathcal{A}$ such that the letters in $W_i$ are all less than the letters in $W_{i+1}$ under the total order on $\mathcal{A}$. Then 
\[
    \sum_{\sigma\in W_1\shuffle \ldots\shuffle W_m}q^{\inv(\sigma)}=q^{\inv(W_1)+\ldots+\inv(W_m)}{k_1+\ldots+k_m \brack k_1,\ldots,k_m}_q.
\]
\end{cor}

There is also a similar extension for $\maj$, while the letters in $W_i$ are not necessarily all less than those in $W_{j}$ for $i<j$.

\begin{thm}[\cite{garsia1979permutation}] \label{thm: Shuffle_maj}
Let $W_1, \ldots, W_m$ be words of length $k_1, \ldots, k_m$ over $\mathcal{A}$ such that for any distinct $i,j$ the letter set of $W_i$ and $W_j$ are disjoint. Then
\[
    \sum_{\sigma\in W_1\shuffle \ldots\shuffle W_m}q^{\maj(\sigma)}=q^{\maj(W_1)+\ldots+\maj(W_m)}{k_1+\ldots+k_m \brack k_1,\ldots,k_m}_q.
\]
\end{thm}

Next, let us restrict our attention to the case $\mathcal{A}=[n]$ and the permutations on $[n]$. Denote by $\S_n$ the set of permutations on $[n]$. For $\sigma=\sigma_1\sigma_2\ldots\sigma_n\in \S_n$, besides the statistics $\DES$, $\des$, $\inv$, $\maj$ , other classical statistics for $\S_n$ are
\begin{itemize}
    \item \emph{Excedence set} $\EXC(\sigma)\coloneqq\{i\in[n-1]: \sigma_i>i\}$ and \emph{excedence number} $\exc(\sigma)\coloneqq |\EXC(\sigma)|$.
\end{itemize}
It is well known that $\exc$ and $\des$ are equidistributed statistics on $\S_n$ and belong to the class of \emph{Eulerian statistics}. The generating function of their distribution 
\begin{equation} \label{eq:DefEulPoly} 
    A_n(t)\coloneqq \sum_{\sigma\in\S_n}t^{\des(\sigma)}=\sum_{\sigma\in\S_n} t^{\exc(\sigma)}
\end{equation}
is known as the \emph{Eulerian polynomial} whose coefficients are called the \emph{Eulerian numbers}. Euler proved the following formula for the exponential generating function of the Eulerian polynomials (used a definition of $A_n(t)$ different from ours) 
\begin{equation} \label{eq:Euler'sFormula}
    1+\sum_{n\ge 1}A_n(t)\frac{z^n}{n!}=\frac{(1-t)e^{z}}{e^{zt}-te^{z}}.
\end{equation}

In their study of the poset homology of Rees product, Shareshian and Wachs \cite{ShareshianWachs2007} discovered a $q$-analog of Euler's formula (\ref{eq:Euler'sFormula}), which is unknown at that time.  
Establishing this $q$-analog and the corresponding symmetric function analog leads to a series of interesting works \cite{ShareshianWachs2010,SaganShareshianWachs2011,HendersonWachs2012} and further generalizations \cite{ShareshianWachs2020,ShareshianWachs2016Chromatic}. In the following,
We will introduce the $q$-analogs and the symmetric function analogs of Eulerian polynomials and Eulerian numbers that Sharashian and Wachs established in \cite{ShareshianWachs2010}. We assume some background on symmetric functions and recall some basic terminology about quasisymmetric functions here; a good reference for this is \cite[Chapter~7]{StanleyEC2}. 

A \emph{quasisymmetric function} $f(\x)=f(x_1,x_2,\ldots)$ is a formal power series of finite degree with infinitely many variables $x_1$, $x_2$, \ldots such that any two monomials $x_{i_1}^{a_1}x_{i_2}^{a_2}\ldots x_{i_k}^{a_k}$ $(i_1<\ldots<i_k)$ and $x_{j_1}^{b_1}x_{j_2}^{b_2}\ldots x_{j_k}^{b_k}$ $(j_1<\ldots<j_k)$ have the same coefficients whenever $(a_1,\ldots, a_k)=(b_1,\ldots, b_k)$. In the algebra of quasisymmetric functions, Gessel's \emph{fundamental quasisymmetric functions} form a basis that has many nice properties. For a positive integer $n$ and a subset $S\subseteq [n-1]$, the fundamental quasisymmetric function $F_{S,n}$ is defined as   
\[
    F_{S,n}(\x)\coloneqq \sum_{\substack{i_1\ge i_2\ge \ldots\ge i_n\\ j\in S\Rightarrow i_j>i_{j+1}}}x_{i_1}x_{i_2}\ldots x_{i_n}
\]
and $F_{\phi,0}\coloneqq 1$. In particular, if $S=\emptyset$ then $F_{\emptyset, n}$ is the complete homogeneous symmetric function $h_n$; if $S=[n-1]$ then $F_{[n-1],n}$ is the elementary symmetric function $e_n$. The \emph{stable principal specialization} of a (quasi)symmetric function $f(\x)$ is given by $\ps(f)\coloneqq f(1,q,q^2,\ldots)$. 
\begin{lem}[\cite{Gessel1993cycle}]\label{psF}
For all $n\ge 1$, $S\subseteq [n-1]$, we have 
\[
    \ps(F_{S,n})=\frac{q^{\sum_{i\in S}i}}{(1-q)(1-q^2)\ldots(1-q^n)}
\]
\end{lem}
It follows from Lemma \ref{psF} that for all $n$ the stable principal specialization of $h_n$ is
\[
    \ps(h_n)=\frac{1}{(1-q)(1-q^2)\ldots(1-q^n)}
\]
and equivalently, 
\[
    (1-q)^n\ps(h_n)=\frac{1}{[n]_q!}.
\]

Let $[\overline{n}]\coloneqq\{\overline{1},\overline{2},\ldots,\overline{n}\}$ and assign to $[n]\cup[\overline{n}]$ the total order $\overline{1}<\ldots<\overline{n}<1<\ldots<n$. Let $\overline{\sigma}$ be a word over $[n]\cup[\overline{n}]$ obtained from $\sigma$ by replacing $\sigma_i$ with $\overline{\sigma_i}$ whenever $i\in\EXC(\sigma)$. Shareshian and Wachs \cite{ShareshianWachs2010} defined the set statistic
\[
	\DEX(\sigma)\coloneqq\DES(\overline{\sigma}).
\]
The statistic $\DEX$ has amazing properties that lead to several important results in permutation enumeration.
\begin{thm}[\cite{ShareshianWachs2010}] \label{thm:DEXPerm}
For all $\sigma \in \S_n$, 
\[
    \sum_{i\in\DEX(\sigma)}i=\maj(\sigma)-\exc(\sigma)
\]
and 
\[
    |\DEX(\sigma)|=\begin{cases}
        \des(\sigma)  & \text{ if }\sigma(1)=1\\
        \des(\sigma)-1 &  \text{ if }\sigma(1)\neq 1
    \end{cases}.
\]
\end{thm}
Then by Lemma \ref{psF}, we have the following corollary.

\begin{cor}[\cite{ShareshianWachs2010}] \label{cor:SumDEX}
For all $\sigma\in\S_n$, 
\[
    \ps(F_{\DEX(\sigma),n})=\frac{q^{\maj(\sigma)-\exc(\sigma)}}{(1-q)(1-q^2)\ldots(1-q^n)}.
\]
\end{cor}

Shareshian and Wachs introduced the \emph{Eulerian quasisymmetric function}  
\begin{equation}\label{F-expansionQ}
    Q_{n,j}(\x)\coloneqq \sum_{\substack{\sigma\in \S_n\\ \exc(\sigma)=j}} F_{\DEX(\sigma),n}(\x)
\end{equation}
and obtained the following analog of Euler's formula (\ref{eq:Euler'sFormula}).

\begin{thm}[\cite{ShareshianWachs2010}] \label{thm:gfEulerianQsym}
We have
\[
    1+\sum_{n\ge 1}\sum_{j=0}^{n-1}Q_{n,j}(\x)t^j z^n=\frac{(1-t)H(z)}{H(tz)-tH(z)}. 
\]
\end{thm}
The right-hand side of the identity in Theorem \ref{thm:gfEulerianQsym} has the same closed form as the right-hand side of (\ref{eq:FrobPerm}).
Therefore, the Eulerian quasisymmetric function $Q_{n,j}$ is actually symmetric and is exactly the Frobenius characteristic of the $\S_n$-representation on $H^{2j}(X_{\Sigma_n})$, on $\mathbb{C}\Code_{n,j}$, and on $A^j(\mathsf{B}_n)_{\mathbb{C}}\coloneqq \mathbb{C}\otimes_{\mathbb{Q}}A^j(\mathsf{B}_n)$ for all $j$. We summarize these results obtained by Stanley \cite{Stanley1989}, Stembridge \cite{Stembridge1992}, Shareshain and Wachs \cite{ShareshianWachs2010}, and Proposition \ref{prop:FY(B_n)Perm} (or Theorem \ref{thm:bijectionBasis1}) as follows. 

\begin{thm} \label{thm:summaryQ}
For all $n\ge 1$ and $0\le j\le n-1$, we have
\[
    \ch(H^{2j}(X_{\Sigma_n}))=\ch(\mathbb{C}\Code_{n,j})=\ch(A^j(\mathsf{B}_n)_{\mathbb{C}})=Q_{n,j}(\x)=\sum_{\substack{\sigma\in\S_n\\ \exc(\sigma)=j}}F_{\DEX(\sigma),n}(\x).
\]
    
\end{thm}

Recall the definition of $Q_n(\x,t)$ given in (\ref{eq:GeoDefQ}); here we take 
\[
    Q_n(\x,t)=\sum_{j=0}^{n-1}Q_{n,j}(\x)t^j
\] 
to be the combinatorial definition of $Q_n(\x,t)$ for $n\ge 1$. It is useful to set $Q_0(\x,t)\coloneqq 1$.

By Corollary \ref{cor:SumDEX}, the stable principal specialization of $Q_n(\x,t)$ yields an interesting $q$-analog of the Eulerian polynomial. 

\begin{thm}[\cite{ShareshianWachs2010}] \label{psEulerianQ}
For all $n\ge 1$, 
\[
    \prod_{i=1}^n\left(1-q^i\right)\ps(Q_n(\mathbf{x},t))=\sum_{\sigma\in\S_n}q^{\maj(\sigma)-\exc(\sigma)}t^{\exc(\sigma)}.
\]
\end{thm}
\noindent
We denote the $q$-Eulerian polynomial on the right-hand side of the above identity as  
\begin{equation}\label{qEulerian}
	A_n(q,t)\coloneqq\sum_{\sigma\in\S_n}q^{\maj(\sigma)-\exc(\sigma)}t^{\exc(\sigma)}.
\end{equation}
Therefore, Shareshian and Wachs \cite{ShareshianWachs2010} obtained the following $q$-analog of Euler's formula (\ref{eq:Euler'sFormula}) by taking the stable principal specialization of the formula in Theorem \ref{thm:gfEulerianQsym}
\[
    1+\sum_{n\ge 1}A_n(q,t)\frac{z^n}{[n]_q!}=\frac{(1-t)\exp_q(z)}{\exp_q(zt)-t\exp_q(z)}
\]
where $\exp_q(z)\coloneqq\sum_{n\ge 0}\frac{z^n}{[n]_q!}$.

\subsection{Enumerative results related to the $q$-Eulerian polynomials} \label{Enumeration1}

In \cite{Stembridge1992}, Stembridge showed that the Eulerian polynomial can be interpreted in terms of Stembridge codes as
\[
    A_n(t)=\sum_{(\alpha,f)\in\Code_n}q^{\ind(\alpha,f)}.
\]
We give a $q$-analogue of Stembridge's result which shows that the Shareshian--Wachs $q$-Eulerian polynomial $A_n(q,t)$ can also be interpreted in terms of Stembridge codes. 

\begin{thm} \label{qCountCodes}
For any $n\ge 1$, 
\[
A_n(q,t)=\sum_{(\alpha,f)\in\Code_n}q^{\inv(\alpha)}t^{\ind(\alpha,f)}=\sum_{(\alpha,f)\in\Code_n}q^{\maj(\alpha)}t^{\ind(\alpha,f)},
\]
where $\inv(\alpha)$, $\maj(\alpha)$ are the inversion number and the major index of the words $\alpha$ over $\{0<1<2<\ldots\}$. 
\end{thm}
\begin{proof}
Note that codes in the same $\mathfrak{S}_n$-orbit in $\Code_n$ have the same contents and span a permutation module $M_\lambda$ where $\lambda$ is the content. We choose a representative from each orbit to form a set $U_n$. Let $\mathcal{O}(\alpha,f)$ be the orbit containing $(\alpha,f)$. Recall from Section \ref{Sec:RepCohPerm} that $\lambda(\alpha,f)$ is the content of the Stembridge code $(\alpha,f)$ and is independent of $f$.  Then
\[
    \ch(\mathbb{C}\Code_{n,j})=\sum_{\substack{(\alpha,f)\in U_n\\ \ind(\alpha,f)=j}}\ch\left(M_{\lambda(\alpha,f)}\right)=\sum_{\substack{(\alpha,f)\in U_n\\ \ind(\alpha,f)=j}}h_{\lambda(\alpha,f)}.
\]
We now take the stable prinicipal specialization; if $\lambda(\alpha,f)=(\lambda_1(\alpha),\ldots,\lambda_k(\alpha))$ then
\[
    \prod_{i=1}^n(1-q^i) \ps\left(h_{\lambda(\alpha,f)}\right)=[n]_q!\prod_{j=1}^k(1-q)^{\lambda_j(\alpha)}\ps\left(h_{\lambda_j(\alpha)}\right)=\qbin{n}{\lambda_1(\alpha),\ldots,\lambda_k(\alpha)}.
\]
Then we get
\begin{align*}
    \prod_{i=1}^n(1-q^i)\ps\left(\ch(\mathbb{C}\Code_{n,j})\right)
    &=\sum_{\substack{(\alpha,f)\in U_n  \\ \ind(\alpha,f)=j}}\prod_{i=1}^n(1-q^i)\ps\left(h_{\lambda(\alpha,f)}\right) \\
    &=\sum_{\substack{(\alpha,f)\in U_n\\ \ind(\alpha,f)=j }}\qbin{n}{\lambda_1(\alpha),\ldots,\lambda_k(\alpha)}  \\
    &=\sum_{\substack{(\alpha,f)\in U_n\\ \ind(\alpha,f)=j }}\sum_{(\alpha',f')\in\mathcal{O}(\alpha,f)}q^{\inv(\alpha')} \quad (\text{by Theorem } \ref{Multi_inv_maj})\\
    &=\sum_{\substack{(\alpha,f)\in\Code_n\\ \ind(\alpha,f)=j}}q^{\inv(\alpha)}.  
\end{align*}
Therefore by Theorem \ref{psEulerianQ}, since $Q_{n,j}(\x)$ is the Frobenious characteristic of $\mathbb{C}\Code_{n,j}$, we have
\[
    A_n(q,t)=\sum_{j=0}^{n-1}\left(\sum_{\substack{(\alpha,f)\in\Code_n\\ \ind(\alpha,f)=j}}q^{\inv(\alpha)}\right)t^j=\sum_{(\alpha,f)\in\Code_n}q^{\inv(\alpha)}t^{\ind(\alpha,f)}.
\]
Because $\inv$ and $\maj$ are equidistributed in multiset permutations (see Theorem \ref{Multi_inv_maj}), we can replace $\inv$ by $\maj$ to get the second equality.
\end{proof}

The Eulerian quasisymmetric function $Q_n(\x,t)$ can also be interpreted using Stembridge codes. 
But we need to take some detours to achieve it. 
In \cite[Sec. 4]{Stembridge1992}, Stembirdge introduced a \emph{mark tableau} $(T,f)$ and its index $\ind(T,f)$ in which $T$ is a semistandard Young tableau whose entries satisfy the same conditions as the entries in a Stembridge code. (See \cite{Stembridge1992} for a complete definition.) 
He used mark tableaux to give a combinatorial interpretation to the multiplicities in the decomposition of $H^*(X_{\Sigma_n})$ into irreducible representations of $\S_n$. 
Let $S^{\lambda}$ be the irreducible representation of $\S_n$ indexed by partition $\lambda$ of $n$ (written as $\lambda\vdash n$). 

\begin{thm}[Stembridge \cite{Stembridge1992}] \label{thm:InterpretationPerm}
For $n\ge 1$, $0\le j\le n-1$, if
\[
    H^{2j}(X_{\Sigma_n},\mathbb{C})\cong_{\S_n}\bigoplus_{\lambda\vdash n}T_{\lambda,j}S^{\lambda},
\]
 then
 \[
    T_{\lambda,j}=|\{(T,f): \mathrm{shape}(T)=\lambda,~\ind(T,f)=j\}|.
 \]
\end{thm}

Stembridge also considered an RSK algorithm (abbreviation of Robinson-Schensted-Knuth, see \cite[Sec 7.11]{StanleyEC2}) for Stembridge codes such that for any $(\alpha,f)\in\Code_{n,j}$,
\[
    (\alpha,f)\longmapsto ((T,f), Q)
\]
where $(T,f)$ is a marked tableau of index $j$ and $Q$ is a standard Young tableau whose shape is the same as $T$.

Using this RSK algorithm for Stembridge codes and Theorem \ref{thm:InterpretationPerm}, we can show the following result.

\begin{thm} \label{thm:QnCodes}
For all $n\ge 1$,
\[
    Q_n(\x,t)=\sum_{(\alpha,f)\in\Code_n}F_{\DES(\alpha),n}(\x)t^{\ind(\alpha,f)},
\]
or equivalently,
\[
    Q_{n,j}(\x)=\sum_{\substack{(\alpha,f)\in\Code_{n,j}}}F_{\DES(\alpha),n}(\x).
\]
\end{thm}

\begin{proof}
Let $A$ be a subset of $[n-1]$ and $H_A$ be the skew hook whose descent set is $A$. Since $Q_{n,j}(\x)$ is symmetric, by Theorem 3 in \cite{GesselSkewSchur1984}, the coefficient of $F_{A,n}(\x)$ in $Q_{n,j}(\x)$ is given by the inner product
\begin{align}
    \langle Q_{n,j}(\x), s_{H_A} \rangle
    =&\left\langle \sum_{\lambda\vdash n}T_{\lambda,j}s_\lambda, s_{H_A}\right\rangle ~(\text{by Theorem \ref{thm:InterpretationPerm}})\nonumber\\
    =&\left\langle \sum_{\lambda\vdash n}T_{\lambda,j}s_\lambda, \sum_{\mu\vdash n}c_{H_A,\mu}s_{\mu} \right\rangle ~(\text{see (2.2.4) in \cite{WachsTool2007}})\nonumber\\
    =&\sum_{\lambda\vdash n}T_{\lambda,j}c_{H_A,\lambda} \label{eq:RSK}
\end{align}
where $c_{H_A,\lambda}$ is the number of standard Young tableaux of shape $\lambda$ and descent set $A$. Therefore, (\ref{eq:RSK}) is equal to
\begin{align*}
    &|\{((T,f),Q): \mathrm{shape}(T)=\mathrm{shape}(f),~\ind(T,f)=j,~\DES(Q)=A\}|\\
    =&|\{(\alpha,f): \ind(\alpha,f)=j, \DES(\alpha)=A\}|~ (\text{by RSK for Stembridge codes}).
\end{align*}
This implies
\[
    Q_{n,j}(\x)=\sum_{\substack{(\alpha,f)\in\Code_{n,j}}}F_{\DES(\alpha),n}(\x).
\]
\end{proof}

Theorem \ref{thm:QnCodes} implies that for $A\subseteq [n-1]$,
\[
    |\{\sigma\in\S_n~:~\exc(\sigma)=j,\DEX(\sigma)=A\}|=|\{(\alpha,f): \ind(\alpha,f)=j, \DES(\alpha)=A\}|.
\]

Since now both $A_n(q,t)$ and $Q_n(\x,t)$ can be interpreted in terms of both permutations $\S_n$ and Stembridge codes $\Code_n$, we ask the following natural question.
\begin{ques} \label{Q1:BijectionCodes}
Can one find a bijection between $\S_n$ and $\Code_n$ that takes $(\maj-\exc,\exc)$ over $\S_n$ to $(\inv\text{ or }\maj,\ind)$ over $\Code_n$, or even $(\DEX,\exc)$ over $\S_n$ to $(\DES,\ind)$ over $\Code_n$? (In \cite{Stembridge1992}, Stembridge actually constructed a bijection that takes $\exc$ over $\S_n$ to $\ind$ over $\Code_n$; however, this bijection fails to preserve the joint distributions.) It is even better if such a bijection translates the conjugation action by $\langle (12\ldots n)\rangle$ on $\S_n$ to the cyclic action on $\Code_n$ (see Section \ref{CSP}).
\end{ques}


\subsection{Enumerative results related to the $q$-binomial Eulerian polynomials}
Shareshian and Wachs in \cite{ShareshianWachs2020} considered a natural $q$-analogue of $\widetilde{A}_n(t)$, the \emph{$q$-binomial Eulerian polynomial}
\begin{equation} \label{DefqBinom}
    \widetilde{A}_n(q,t)=1+t\sum_{k=1}^n\qbin{n}{k}A_k(q,t),
\end{equation}
and showed that it is given by the stable principal specialization of 
\[
    \widetilde{Q}_n(\x,t)\coloneqq h_n(\x)+t\sum_{k=1}^n h_{n-k}(\x)Q_k(\x,t), 
\]
which is the analog of Theorem \ref{psEulerianQ}.
\begin{thm}[\cite{ShareshianWachs2020}] 
 \label{thm:SpecializeBinEuler}
For all $n\ge 0$,
\[
    \prod_{i=1}^n(1-q^i)\ps(\widetilde{Q}_n(\x,t))=\widetilde{A}_n(q,t).
\]
\end{thm}

Similar to the fact that $A_n(q,t)$ can be interpreted in terms of Stembridge codes in Theorem~\ref{qCountCodes}, we also give a new combinatorial interpretation of the binomial Eulerian polynomial $\widetilde{A}_n(q,t)$ in terms of extended codes. 
\begin{prop}\label{qExtendedCodes}
For any integer $n\ge 0$,
\[
\widetilde{A}_n(q,t)=\sum_{(\alpha,f)\in\widetilde{\Code}_{n}}q^{\inv(\alpha)}t^{\ind(\alpha,f)+1}=\sum_{(\alpha,f)\in\widetilde{\Code}_{n}}q^{\maj(\alpha)}t^{\ind(\alpha,f)+1},
\]
where $\inv(\alpha)$ and $\maj(\alpha)$ are the inversion number and major index of $\alpha$ as a word over $\{0<1<2<\ldots<\infty\}$.
\end{prop}
\begin{proof}
Starting with the definition of $\widetilde{A}_n(q,t)$ in (\ref{DefqBinom}) then applying Theorem \ref{qCountCodes}, we have
\[
    \tilde{A}_n(q,t)=1+t\sum_{m=1}^n{n \brack m}_q\sum_{j=0}^{m-1}\left(\sum_{(\alpha,f)\in\Code_{m,j}}q^{\inv(\alpha)}t^{j}\right).
\]
By Corollary \ref{ShuffleINV}, the right-hand side is equal to
\begin{align*}
    &1+t\sum_{m=1}^n\sum_{j=0}^{m-1}\left(\sum_{(\alpha,f)\in\Code_{m,j}}q^{\inv(\alpha)}{n \brack m}_q t^{j}\right)
    =1+\sum_{m=1}^n\sum_{j=0}^{m-1}\left(\sum_{(\alpha,f)\in\Code_{m,j}}\sum_{\alpha'\in\alpha\shuffle\infty^{n-m}}q^{\inv(\alpha')}t^{j+1}\right)\\
    =& 1+\sum_{m=1}^n\sum_{j=0}^{m-1}\left(\sum_{\substack{(\alpha',f)\in\widetilde{\Code}_{n,j}\\ \alpha \text{ contains $n-m$ $\infty$'s}}}q^{\inv(\alpha')}t^{j+1}\right)
    = 1+\sum_{j=0}^{n-1}\left(\sum_{m=j+1}^n\sum_{\substack{(\alpha',f)\in\widetilde{\Code}_{n,j}\\ \alpha \text{ contains $n-m$ $\infty$'s}}}q^{\inv(\alpha')}\right)t^{j+1}\\
    =& 1+\sum_{j=0}^{n-1}\left(\sum_{(\alpha',f)\in\widetilde{\Code}_{n,j}}q^{\inv(\alpha')}\right)t^{j+1}
    =\sum_{(\alpha',f)\in\widetilde{\Code}_{n}}q^{\inv(\alpha')}t^{\ind(\alpha',f)+1}.
\end{align*}
Since $A_m(q,t)$ is also the enumerator of codes $(\alpha,f)$ with respect to $\maj(\alpha)$ and $\ind(\alpha,f)$, a similar argument with $\inv$ being replaced by $\maj$ gives the second equality. 
\end{proof}

As an analog of the set of permutations $\mathfrak{S}_n$ in the Eulerian story, we consider the set $\widetilde{\mathfrak{S}}_n$ of  \emph{decorated permutations} of $[n]$. A \emph{decorated permutation} of $[n]$ can be viewed as a permutation on a subset of $[n]$. 
\begin{exa}
For example, $215\in\widetilde{\mathfrak{S}}_5$ is a permutation on $\{1,2,5\}\subset[5]$ that maps $1$ to $2$, $2$ to $1$, $5$ to $5$. We can express the map in \emph{two-line notation} and in \emph{one-line notation} respectively as the following 
\[
	215=\left(\begin{array}{ccccc}
		1 & 2 & 3 & 4 & 5\\
		2 & 1 & 0 & 0 & 5
	\end{array}\right)=21005.
\]
In particular, let $\theta$ denote the permutation on the empty set $\emptyset$ in $\widetilde{\mathfrak{S}}_5$. It has the following two-line notation and one-line notation
\[
	\theta\coloneqq\left(\begin{array}{ccccc}
		1 & 2 & 3 & 4 & 5\\
		0 & 0 & 0 & 0 & 0
	\end{array}\right)=00000.
\]
When $n=2$, the set $\widetilde{\S}_2$ is $\{\theta, 1,2,12,21\}=\{00,10,02,12,21\}$.
\end{exa}
The decorated permutations were first introduced by Postnikov \cite{Postnikov2006Positroid} to parametrize the positroid cells in the nonnegative Grassmannian. He defines decorated permutations as permutations with two kinds of fixed points. One can view the zero entries in the one-line notation in our definition as the second kind of fixed points; then our definition coincides with Postnikov's. 
Decorated permutations are also used by Postnikov, Reiner, Williams \cite[Section 10.4]{PRW2008} to give a combinatorial interpretation to the binomial Eulerian polynomial $\widetilde{A}_n(t)$, but in their paper the decorated permutations are called \emph{partial permutations} instead. They consider the binomial Eulerian polynomial as the distribution generating function of the descent number $\des$ over $\widetilde{\S}_n$, while in this paper we interpret it as the excedence number $\exc$. Recently, the term ``partial permutation'' has been used in \cite{heuer2020partial}, \cite{behrend2022partial}, and \cite{cheng2023PartialPerm}. But these ``partial permutations of length $n$'' describe a larger set containing  $\widetilde{\mathfrak{S}}_n$.

The relations which we have seen so far between permutations and Eulerian polynomials actually extends to decorated permutations and binomial Eulerian polynomials nicely. 
\begin{defi}
For $\sigma=\sigma_1\ldots\sigma_n\in\widetilde{\mathfrak{S}}_n$ in one-line notation, we define the \emph{excedence number} and \emph{major index} of $\sigma$ as the following
\[
	\exc(\sigma)\coloneqq\begin{cases}
	|\{i\in[n]: i<\sigma_i\}| & \text{, if }\sigma\neq\theta;\\
	-1  & \text{, if }\sigma=\theta
	\end{cases} ~\text{ and }~
	\maj(\sigma)\coloneqq\begin{cases}
	\sum_{i:\sigma_i>\sigma_{i+1}} i & \text{, if }\sigma\neq\theta;\\
	-1  & \text{, if }\sigma=\theta.
	\end{cases}
\]
And $\fixTwo(\sigma)$ is the number of $0$'s in the one-line notation of $\sigma$.
\end{defi}
Using these statistics, we obtain an analog of (\ref{qEulerian}) that gives another new combinatorial interpretation of $\widetilde{A}_n(q,t)$ in terms of decorated permutations $\widetilde{\mathfrak{S}_n}$.
\begin{thm}\label{qpartialpermutations}
For any integer $n\ge 0$,
\[
    \widetilde{A}_n(q,t)=\sum_{\sigma\in\widetilde{\mathfrak{S}}_n}q^{\maj(\sigma)-\exc(\sigma)}t^{\exc(\sigma)+1}.
\]
\end{thm}
\begin{proof}
From the definition and (\ref{qEulerian}) we have  
\begin{align*}
    \tilde{A}_n(q,t)
    &=1+t\sum_{m=1}^n {n \brack m}_q A_m(q,t)=1+t\sum_{m=1}^n {n \brack m}_q\sum_{\pi\in\mathfrak{S}_m}q^{\maj(\pi)-\exc(\pi)}t^{\exc(\pi)}\\
    &=1+\sum_{m=1}^n\sum_{\pi\in\mathfrak{S}_m}q^{-\exc(\pi)}t^{\exc(\pi)+1}q^{\maj(\pi)}{n \brack m}_q.
\end{align*}
By Theorem \ref{thm: Shuffle_maj}, the last equality is equal to
\[
    1+\sum_{m=1}^n\sum_{\pi\in\mathfrak{S}_m}q^{-\exc(\pi)}t^{\exc(\pi)+1}\sum_{\sigma\in\pi\shuffle 0^{n-m}}q^{\maj(\sigma)}
    =1+\sum_{m=1}^n\sum_{\substack{\pi\in\mathfrak{S}_m\\
    \sigma\in\pi\shuffle 0^{n-m}}}q^{\maj(\sigma)-\exc(\pi)}t^{\exc(\pi)+1}.
\]
Note that each shuffle of $\pi\in\S_m$ with $0^{n-m}$ corresponds to a decorated permutation $\sigma$ whose one-line notation has $n-m$ $0$'s and the nonzero part has the same relative order as in $\pi$. For example, the shuffles $213\shuffle 0$ corresponds to the decorated permutations $2130,2104,3014,0324$. Clearly, the correspondence is a bijection, so the last equality can be rewritten as
\[
    1+\sum_{\sigma\in\widetilde{\mathfrak{S}}_n\setminus\{\theta\}}q^{\maj(\sigma)-\exc(\sigma)}t^{\exc(\sigma)+1}=\sum_{\sigma\in\widetilde{\mathfrak{S}}_n}q^{\maj(\sigma)-\exc(\sigma)}t^{\exc(\sigma)+1}.
\]
\end{proof}

We can extend the set statistic $\DEX$ to be a statistic over $\widetilde{\mathfrak{S}}_n$ in the following way. 
\begin{defi} \label{def:DEXext}
Consider the set $[n]\cup[\overline{n}]\cup\{0\}$ with the total order $\overline{1}<\ldots<\overline{n}<0<1<\ldots<n$. For $\sigma\in\widetilde{\S}_n\setminus\{\theta\}$, let $\overline{\sigma}$ be a word over $[n]\cup[\overline{n}]\cup\{0\}$ obtained from $\sigma$ by replacing $\sigma_i$ with $\overline{\sigma_i}$ whenever $i\in\EXC(\sigma)\coloneqq\{i\in[n]:i<\sigma_i\}$. Define the set statistic
\[
    \DEX(\sigma)\coloneqq\begin{cases}
        \DES(\overline{\sigma}), & \mbox{ if }\sigma\neq \theta;\\
        \emptyset, & \mbox{ if }\sigma=\theta. 
    \end{cases}
\]
\end{defi}

The extension of $\DEX$ to decorated permutations $\widetilde{\S}_n$ preserves the nice property it has as a statistic over $\S_n$.
\begin{lem}
For $\sigma\in\widetilde{\S}_n$, 
\[
    \sum_{i\in\DEX(\sigma)}i=\maj(\sigma)-\exc(\sigma).
\]
And for $\sigma\in\widetilde{\S}_n\setminus\{\theta\}$,
\[
    |\DEX(\sigma)|=\begin{cases}
        \des(\sigma) & \mbox{, if }\sigma(1)=0 \text{ or }1;\\
        \des(\sigma)-1 & \mbox{, if }\sigma(1)>1.
    \end{cases}
\]
\end{lem}
The proof of this lemma is similar to the proof of Lemma 2.2 in \cite{ShareshianWachs2010}. For the sake of completeness, we also sketch the proof here.

\begin{proof}
Given $\sigma\in\widetilde{\S}_n$, compare the two sets $\DEX(\sigma)$ and $\DES(\sigma)$. Let 
\[
    J(\sigma)=\{i\in[n-1]: i\notin \EXC(\sigma), i+1\in\EXC(\sigma)\}
\]
and 
\[
    K(\sigma)=\{i\in[n-1]: i\in\EXC(\sigma), i+1\notin\EXC(\sigma)\}.
\]
One can check that if $i\in J(\sigma)$, then $i\notin\DES(\sigma)$ but $i\in\DEX(\sigma)$. If $i\in K(\sigma)$, then $i\in\DES(\sigma)$ but $i\notin\DEX(\sigma)$. Hence $K(\sigma)\subseteq\DES(\sigma)$ and  
\[
    \DEX(\sigma)=\DES(\sigma)-K(\sigma)\uplus J(\sigma) \quad\text{ and }\quad |\DEX(\sigma)|=\des(\sigma)-|K(\sigma)|+|J(\sigma)|.
\]
Let $J(\sigma)=\{j_1<j_2<\ldots<j_t\}$ and $K(\sigma)=\{k_1<k_2<\ldots<k_s\}$. 

Consider the case that $\sigma_1=0$ or $1$. Since neither $1$ nor $n$ is in $\EXC(\sigma)$, we should have $s=t$ and 
\[
    1\le j_1<k_1<j_2<k_2<\ldots<j_t<k_t\le n-1.
\]
Therefore, $|\DEX(\sigma)|=\des(\sigma)$. Observe that 
\[
    \EXC(\sigma)=\biguplus_{i=1}^t\{j_i+1,j_i+2,\ldots,k_i\}.
\]
This shows that $\sum_{i=1}^t(k_i-j_i)=\exc(\sigma)$ and hence 
\[
    \sum_{i\in\DEX(\sigma)}i=\sum_{i\in\DES(\sigma)}i-\sum_{i=1}^t k_i+\sum_{i=1}^t j_i=\maj(\sigma)-\exc(\sigma).
\]

Consider the other case that $\sigma(1)>1$, i.e. $1\in\EXC(\sigma)$. We have $s=t+1$ and
\[
    1\le k_1<j_1<\ldots<k_t<j_t<k_{t+1}\le n-1.
\]
Then $|\DEX(\sigma)|=\des(\sigma)-(t+1)+t=\des(\sigma)-1$. Observe that in this case 
\[
    \EXC(\sigma)=\{1,2,\ldots,k_1\}\cup\biguplus_{i=1}^t\{j_i+1,j_i+2,\ldots k_{i+1}\}.
\]
This shows that $\sum_{i=1}^{t+1}k_i-\sum_{i=1}^t j_i=\exc(\sigma)$ and hence we still have
\[
    \sum_{i\in\DEX(\sigma)}i=\sum_{i\in\DES(\sigma)}i-\sum_{i=1}^{t+1}k_i+\sum_{i=1}^t j_i=\maj(\sigma)-\exc(\sigma).
\]
\end{proof}

Furthermore, an $F$-expansion of $\widetilde{Q}_{n,j}(\x)$ analogous to (\ref{F-expansionQ}) also follows from our extension. 
\begin{thm} \label{thm:FexpWideQ}
For any integer $n\ge 0$, 
\[
	\widetilde{Q}_n(\x,t)=\sum_{\sigma\in\widetilde{\mathfrak{S}}_n}F_{\DEX(\sigma),n}(\x)t^{\exc(\sigma)+1}.
\]
Or equivalently, for all $0\le j\le n$,
\[
    \widetilde{Q}_{n,j}(\x)=\sum_{\substack{\sigma\in\widetilde{\S}_n\\ \exc(\sigma)=j-1}} F_{\DEX(\sigma),n}(\x).
\]
\end{thm}
\begin{proof}
Starting from the definition of $\widetilde{Q}_n(\x,t)$ then applying (\ref{F-expansionQ}), we have 
\begin{align}
    \widetilde{Q}_n(\x,t)
    =& h_n(\x)+t\sum_{k=1}^n h_{n-k}(\x)Q_k(\x,t) \nonumber\\
    =& F_{\emptyset,n}(\x)+\sum_{k=1}^n \sum_{\sigma\in\S_k}F_{\emptyset,n-k}(\x)F_{\DEX(\sigma),k}(\x)t^{\exc(\sigma)+1}.  \label{eq:1}
\end{align}
Since 
\[
    F_{\emptyset,n-k}(\x)F_{\DEX(\sigma),k}(\x)=F_{\DES(0^{n-m}),n-k}(\x)F_{\DES(\overline{\sigma}),k}(\x)=\sum_{\overline{\pi}\in 0^{n-k}\shuffle\overline{\sigma}}F_{\DES(\overline{\pi}),n}(\x).
\]
Same as the bijection in the proof of Theorem \ref{qpartialpermutations}, each shuffle of $\overline{\sigma}$ and $0^{n-k}$ corresponds to a decorated permutation $\pi$ with $\pi(i)$ being replaced by $\overline{\sigma(i)}$ for every excedence $i$. Therefore (\ref{eq:1}) can be rephrased as 
\begin{align*}
      &F_{\emptyset,n}(\x)+\sum_{k=1}^{n}\sum_{\substack{\pi\in\widetilde{\S}_n\\ \fixTwo(\pi)=n-k}}F_{\DEX(\pi),n}(\x)t^{\exc(\pi)+1}\\
    =& F_{\emptyset,n}(\x)+\sum_{\pi\in\widetilde{\S}_n\setminus\{\theta\}}F_{\DEX(\pi),n}(\x)t^{\exc(\pi)+1}\\
    =&\sum_{\pi\in\widetilde{\S}_n}F_{\DEX(\pi),n}(\x)t^{\exc(\pi)+1}
\end{align*}
\end{proof}

The following theorem analogous to Theorem \ref{thm:summaryQ} summarizes Shareshian and Wachs' result in Theorem \ref{FrobStell} and our results Theorem \ref{thm:FcharExtCodes}, Theorem \ref{thm:AugPermbasis} (or Theorem \ref{thm:bijectionBasis2}), and Theorem \ref{thm:FexpWideQ}.    

\begin{thm} \label{thm:summaryTildeQ}
For any $n\ge 0$, $0\le j\le n$, we have 
\[
    \ch(H^{2j}(X_{\widetilde{\Sigma}_n}))=\ch(\mathbb{C}\widetilde{\Code}_{n,j-1})=\ch(\widetilde{A}^j(\mathsf{B}_n)_{\mathbb{C}})=\widetilde{Q}_{n,j}(\x)=\sum_{\substack{\sigma\in\widetilde{\S}_n\\ \exc(\sigma)=j-1}} F_{\DEX(\sigma),n}(\x).
\]
\end{thm}

\begin{rem}
One can apply a procedure similar to Stembridge \cite[Sec. 4]{Stembridge1992} to define marked tableaux and the RSK algorithm for extended codes. Then results analogous to Theorem \ref{thm:InterpretationPerm} and Theorem \ref{thm:QnCodes} follow. 
\end{rem}

As a final remark for the enumerative results related to the binomial Eulerian polynomials, we propose an analogue of Question \ref{Q1:BijectionCodes}.

\begin{ques} \label{Q2:BijExtCodes}
 Can one find a bijection between the extended codes $\widetilde{\Code}_n$ and the decorated permutations $\widetilde{\mathfrak{S}}_n$ that translates $\maj-\exc$ over $\widetilde{\S}_n$ to $\maj$ over $\widetilde{\Code}_n$?   
\end{ques}

 Notice that if one already has a bijection $\varphi_n:\Code_n\rightarrow\S_n$, it can always be extended to a bijection between $\widetilde{\Code}_n$ and $\widetilde{\S}_n$. This can be easily done by viewing a decorated permutation as a permutation on the subset of nonzero entries $S\subset [n]$ and applying the bijection $\varphi_{|S|}$ on this decorated permutation. The resulting extended code $(\alpha, f)$ consists of a sequence $\alpha$ where $\alpha_i$ for $i\in S$ and $f$ determined by the image of $\varphi_{|S|}$ and $\alpha_i=\infty$ for $i\in [n]-S$. Therefore, Stembridge's bijection in \cite{Stembridge1992}, which we mentioned in the end of Section \ref{Enumeration1}, can be extended to the bijection that translates $\exc$ over $\widetilde{\S}_n$ to $\ind$ over $\widetilde{\Code}_n$, but it fails to translate $\maj-\exc$ over $\widetilde{\S}_n$ to $\maj$ over $\widetilde{\Code}_n$. 

\subsection{Cyclic sieving phenomenon on Stembridge codes and extended codes} \label{CSP}
The ``\emph{cyclic sieving phenomenon}'' (CSP) was introduced by Reiner, Stanton, and White \cite{RSW2004}, as a generalization of Stembridge's ``$q=-1"\--$phenomenon, on the generating function of a combinatorial structure on which a cyclic group acts. It is defined as follows.

\begin{defi}
Let $C=\langle c\rangle$ be a finite cyclic group, $X$ be a finite set equipped with a $C$-action, and $\omega=e^{\frac{2\pi i}{|C|}}$ be a primitive $|C|$-th root of unity. Let $X(q)$ be a polynomial with non-negative integer coefficients. Then we say the triple $(X,C,X(q))$ exhibits the cyclic sieving phenomenon if for all $r\ge 0$ we have
\[
    |X^{c^r}|=|\{x\in X: c^r\cdot x=x\}|=X(\omega^r),
\]
where $X^{c^r}$ is the set of elements in $X$ fixed by $c^r$.
\end{defi}

For a partition $\lambda$ of $n$, written as $\lambda \vdash n$, let $\S_{\lambda,j}$ be the subset of all $\sigma\in\S_n$ having cycle type $\lambda$ and  $\exc(\sigma)=j$. Let $\S_{n,j}=\bigcup_{\lambda\vdash n}\S_{\lambda,j}$ be the set of all $\sigma\in\S_n$ having $\exc(\sigma)=j$. Write $A_n(q,t)=a_{n,0}(q)+a_{n,1}(q)t+\ldots+a_{n,n-1}(q)t^{n-1}$ where $a_j(q)$'s are the corresponding $q$-Eulerian numbers. Fix a partition $\lambda\vdash n$, it is obvious that the cyclic group $C_n=\langle (12\ldots n)\rangle$ acts on $\S_{\lambda,j}$ by conjugation. Sagan, Shareshian, and Wachs \cite{SaganShareshianWachs2011} showed the following CSP result. 

\begin{thm}\cite{SaganShareshianWachs2011}\label{CSP:S_nCycle}
For any $\lambda\vdash n$ and $0\le j\le n-1$, the triple 
\[
    \left(\S_{\lambda,j},~ C_n,~ \sum_{\sigma\in\S_{\lambda,j}}q^{\maj(\sigma)-\exc(\sigma)}\right)
\]
exhibits the cyclic sieving phenomenon.
\end{thm}
Theorem \ref{CSP:S_nCycle} implies a less refined CSP. 
\begin{cor} \label{CSP:S_nExc}
For any positive integer $n$ and $0\le j\le n-1$, the triple
\[
    \left(\S_{n,j},~C_n,~a_{n,j}(q)\right)
\]
exhibits the cyclic sieving phenomenon. 
\end{cor}

Recall from Section \ref{Sec:RepCohPerm} that we have a natural $\S_n$-action on $\Code_n$ that preserves the index. 
Consider the restriction of the $\S_n$-action to the $C_n$-action on $\Code_{n,j}$.  
By Theorem \ref{qCountCodes}, for $0\le j\le n-1$, the polynomial $a_{n,j}(q)$ can also be expressed as a $q$-enumerator of $\Code_{n,j}$, i.e. 
 \[
    a_{n,j}(q)=\sum_{(\alpha,f)\in\Code_{n,j}}q^{\maj(\alpha)}.
 \]We shall show that there is a similar CSP over $\Code_{n,j}$.

\begin{thm}\label{thm:CSP1}
    For any positive integer $n$ and $0\le j\le n-1$, the triple 
    \[
        \left(\Code_{n,j},~C_n,~a_{n,j}(q)\right)
    \]
    exhibits the cyclic sieving phenomenon.
\end{thm}

Similarly, write $\widetilde{A}_n(q,t)=\widetilde{a}_{n,0}(q)+\widetilde{a}_{n,1}(q)t+\ldots+\widetilde{a}_{n,n}(q)t^n$. Recall that $\widetilde{\Code}_{n,j}$ is the set of extended codes of length $n$ with index $j$. From Proposition \ref{qExtendedCodes}, for $0\le j\le n$,
\[
    \widetilde{a}_{n,j}(q)=\sum_{(\alpha,f)\in\widetilde{\Code}_{n,j-1}}q^{\maj(\alpha)}.
\]
Then we have the following CSP over $\widetilde{\Code}_{n,j}$.
\begin{thm} \label{thm:CSP2}
For any positive integer $n$ and $0\le j\le n$, the triple 
\[
    \left(\widetilde{\Code}_{n,j-1},~ C_n,~ \widetilde{a}_{n,j}(q)\right)
\]
exhibits the cyclic sieving phenomenon. 
\end{thm}

We recall a result from \cite{SaganShareshianWachs2011} that Sagan, Shareshian and Wachs used to prove Theorem \ref{CSP:S_nCycle}. Let us fix some notations first. 
Given a positive integer $n$, let $\nu$ be a partition of $n$, denoted as $\nu\vdash n$. Let $m_j(\nu)$ the be number of parts of size $j$ of $\nu$ for $1\le j \le n$. Then define the number
\[
    z_\nu\coloneqq\prod_{j=1}^n j^{m_j(\nu)}m_j(\nu)!.
\]

\begin{thm}[\protect{\cite[Prop 3.1]{SaganShareshianWachs2011}}] \label{citeSSW}
Let $F(\x)$ be a homogeneous symmetric function of degree $n$ whose expansion in the power sum basis $\{p_\nu\}_{\nu\vdash n}$ is $F=\sum_{\nu\vdash n}\chi_{\nu}^F\frac{p_\nu}{z_{\nu}}$. Then for any positive integer $d|n$ i.e. $n=dk$ for some $k\in\mathbb{N}$, the coefficient
\[
    \chi_{d^k}^F=\left[\prod_{i=1}^n(1-q^i)\ps(F)\right]_{q=\omega_d}
\]
where $\omega_d$ is any primitive $d$th root of unity.
\end{thm}
In particular, if $F$ is a Frobenius characteristic of some representation of $\S_n$, then $\chi_{\nu}^F$ is exactly the correpsonding $\S_n$-character evaluated at any permutation of cycle type $\nu$. 
Let $c_n=(1,2,\ldots, n)\in\S_n$, and $\zeta_n=e^{\frac{2\pi i}{n}}$.
\begin{rem}
Let $S^\lambda$ be the irreducible $\S_n$-module associated to a partition $\lambda$ of $n$ and $\chi^{\lambda}$ be the corresponding character. Since $(\chi_{\nu}^F:\nu\vdash n)$ can be viewed as a complex-valued class function of $\S_n$ and $\{\chi^\lambda:\lambda\vdash n\}$ forms an orthonormal basis for the class function of $\S_n$, Theorem \ref{citeSSW} is equivalent to the speical case that
\begin{equation} \label{SpecializationSchur}  
    \chi^{\lambda}(c_n^{r})=\left[\prod_{i=1}^n(1-q^i)\ps(s_\lambda)\right]_{q=\zeta_n^r}
\end{equation}
for any $\lambda\vdash n$ and $0\le r\le n$. The identity (\ref{SpecializationSchur}) can be generalized to the setting of \emph{complex reflection groups} via Springer's theory of \emph{regular elements} (see \cite[Proposition 4.5]{Springer1974regular}).
\end{rem}

\begin{proof}[Proofs of Theorems \ref{thm:CSP1} and \ref{thm:CSP2}]
Because the space $\mathbb{C}\Code_{n,j}$ of Stembridge codes forms a permutation representation of $\S_n$ whose permutation basis are Stembridge codes. The character evaluated at $c_n^r$ is exactly the number of the Stembridge codes fixed by $c_n^r$. i.e. 
\begin{align*}
    \chi^{\mathbb{C}\Code_{n,j}}(c_n^{r})=\left|\{(\alpha,f)\in\Code_{n,j}: c_n^r\cdot(\alpha,f)=(\alpha,f)\}\right|.
\end{align*}
On the other hand, by Theorem \ref{thm:summaryQ}, $\ch(\mathbb{C}\Code_{n,j})=Q_{n,j}$. Hence by Theorem \ref{citeSSW} and Theorem \ref{psEulerianQ}, we have
 \[
    \chi^{\mathbb{C}\Code_{n,j}}(c_n^{r})=\left[\prod_{i=1}^n(1-q^i)\ps(Q_{n,j})\right]_{q=\zeta_n^r}=a_{n,j}\left(\zeta_n^r\right).
 \]
 This proves the CSP for $\left(\Code_{n,j},~C_n,~a_{n,j}(q)\right)$.

 Similarly, the CSP for the triple $\left(\widetilde{\Code}_{n,j-1},~ C_n,~ \widetilde{a}_{n,j}(q)\right)$ follows from Theorem \ref{thm:summaryTildeQ}, Theorem \ref{citeSSW}, and Theorem \ref{thm:SpecializeBinEuler}.
\end{proof}
\begin{rem}
Notice that $\Code_{n,j}$ itself forms a collection of $C_n$-orbits in $\widetilde{\Code}_{n,j}$, hence we may also obtain Theorem \ref{thm:CSP1} from Theorem \ref{thm:CSP2} by restricting to the extended codes in $\widetilde{\Code}_{n,j}$ that do not have $\infty$'s.
\end{rem}

Using the $\S_n$-equivaraint bijections $\phi:FY(\mathsf{B}_n)\rightarrow\Code_n$ and 
$\widetilde{\phi}:\widetilde{FY}(\mathsf{B}_n)\rightarrow\widetilde{\Code}_n$ in Theorem \ref{thm:bijectionBasis1} and Theorem \ref{thm:bijectionBasis2}, the CSP in Theorem \ref{thm:CSP1} and Theorem \ref{thm:CSP2} can be restated in terms of Chow rings. Let $FY^j(\mathsf{B}_n)$ and $\widetilde{FY}^j(\mathsf{B}_n)$ be the subsets  consisting of monomials of degree $j$ in $FY(\mathsf{B}_n)$ and $\widetilde{FY}(\mathsf{B}_n)$ respectively.

\begin{cor}
For any positive integer $n$ and $0\le i\le n-1$, $0\le j\le n$, the triples
\[
    \left(FY^i(\mathsf{B}_n),~C_n,~a_{n,i}(q)\right) \text{ and }\left(\widetilde{FY}^j(\mathsf{B}_n),~C_n,~\widetilde{a}_{n,j}(q)\right)
\]
both exhibit CSPs.
\end{cor}

Observing the proof of Theorems \ref{thm:CSP1} and \ref{thm:CSP2}, it is easy to see that Theorem \ref{citeSSW} implies the following general statement.

\begin{cor} \label{cor:GenerateCSP}
Let $X$ be a finite set equipped with a $C_n$-action. If the linear span $\mathbb{C}[X]$ is isomorphic to the restriction $V\downarrow_{C_n}^{\S_n}$ for some $\S_n$-module $V$. Then the triple 
\begin{equation} \label{eq:GenerateCSP}
    \left(X,~ C_n,~ \prod_{i=1}^n(1-q^i)\ps(\ch(V))\right)
\end{equation}
exhibits CSP.
\end{cor}

From the perspective of Corollary \ref{cor:GenerateCSP}, it is not surprising that there are CSPs over Stembridge codes and extended codes, but what really intrigues us is the observation that follows.

Let $\Conj_{n,j}$ be the $C_n$-module over $\mathbb{C}$ generated by $\S_{n,j}$ under conjugation by $C_n$. Then notice that Corollary \ref{CSP:S_nExc} and Theorem \ref{thm:CSP1} together actually show that the character value of the restriction of $\mathbb{C}\Code_{n,j}$ from $\S_n$ to $C_n$ is equal to the character value of $\Conj_{n,j}$. This suggests that the cyclic symmetry structure on the cohomology of the permutahedral variety and on the $\Conj_{n,j}$ are the same in the following sense:
\begin{cor} \label{cor:C_nSymmetry}
For $0\le j\le n-1$, as $C_n$-representations
\[
    H^{2j}(X_{\Sigma_n},\mathbb{C})\downarrow_{C_n}^{\S_n}\cong
    \mathbb{C}\Code_{n,j}\downarrow_{C_n}^{\S_n} \cong \Conj_{n,j}.
\]
\end{cor}
This fact was also mentioned in \cite[Remark 3.3]{SaganShareshianWachs2011}.
An intriguing observation about this is that although there is no obvious way to extend the conjugacy action of $C_n$ on $\S_{n,j}$ to an $\mathfrak{S}_n$-action, Corollary \ref{cor:C_nSymmetry} suggests that such an extension exists. In Question \ref{Q1:BijectionCodes}, we ask for a $C_n$-equivariant bijection between $\S_{n,j}$ and $\Code_{n,j}$ for all $j$. If one finds such a bijection, it will induce an isomorphism between $\Conj_{n,j}$ and $\mathbb{C}\Code_{n,j}\downarrow_{C_n}^{\S_{n}}$ as $C_n$-modules, which will offer another proof of Corollary \ref{CSP:S_nExc}. Moreover, the bijection might offer a hint of a way to extend the $C_n$-action on $\S_{n,j}$ to an $\S_n$-action. 

For $-1\le j\le n-1$, let $\widetilde{\S}_{n,j}$ be the subset of decorated permutations $\sigma\in \widetilde{\S}_n$ with $\exc(\sigma)=j$. Then by an argument akin to that of $\S_{n,j}$, a $C_n$-equivariant bijection between $\widetilde{\S}_{n,j-1}$ and $\widetilde{\Code}_{n,j}$ in Question \ref{Q2:BijExtCodes} will prove the following conjecture.
\begin{conj} For any positive integer $n$ and $0\le j\le n$, the triple
\[
    \left(\widetilde{\S}_{n,j-1},~C_n,~\widetilde{a}_{n,j}(q)\right)
\]
exhibits the cyclic sieving phenomenon.
\end{conj}

\section{Further work}
In an upcoming paper \cite{Liao2023Two}, we apply the bases $FY(M)$ and $\widetilde{FY}(M)$ in Corollary \ref{AugBasis} to uniform matroids $U_{r,n}$ and their $q$-analogs $U_{r,n}(q)$ to study the $\mathbb{C}\S_n$-representations carried by $A(U_{r,n})_{\mathbb{C}}$ and $\widetilde{A}(U_{r,n})_{\mathbb{C}}$, and the $\mathbb{C}GL_n(\mathbb{F}_q)$-representations carried by $A(U_{r,n}(q))_{\mathbb{C}}$ and $\widetilde{A}(U_{r,n}(q))_{\mathbb{C}}$; this not only enables us to recover Hameister, Rao, and Simpson's results on Chow rings in \cite{HameisterRaoSimpson2021} but also obtain the counterpart of it for the case of the augmented Chow rings.

\section*{Acknowledgements}
The author is very grateful to Michelle Wachs for her guidance and encouragement at every stage of this project and carefully reading the manuscript, to Vic Reiner for making his lectures on K\"{a}hler package available online during Covid time and introducing the author to Chow Rings of matroids, which motivated this work.

\bibliographystyle{abbrv}
\bibliography{bibliography}

\end{document}

%% file: Tikzpictures/IndCplxExample.tex
\begin{tikzpicture}[x=0.6pt,y=0.6pt,yscale=-1,xscale=1]
        \small
        \draw[shift={(.5 cm, .5 cm)}]   (87.63,30.48) .. controls (87.49,28.44) and (89.04,26.69) .. (91.07,26.55) .. controls (93.11,26.42) and (94.87,27.96) .. (95,30) .. controls (95.13,32.04) and (93.59,33.79) .. (91.55,33.93) .. controls (89.52,34.06) and (87.76,32.52) .. (87.63,30.48) -- cycle ;
        \draw[shift={(.5 cm, .5 cm)}]   (154.63,125.48) .. controls (154.49,123.44) and (156.04,121.69) .. (158.07,121.55) .. controls (160.11,121.42) and (161.87,122.96) .. (162,125) .. controls (162.13,127.04) and (160.59,128.79) .. (158.55,128.93) .. controls (156.52,129.06) and (154.76,127.52) .. (154.63,125.48) -- cycle ;
        \draw[shift={(.5 cm, .5 cm)}]   (22.63,126.48) .. controls (22.49,124.44) and (24.04,122.69) .. (26.07,122.55) .. controls (28.11,122.42) and (29.87,123.96) .. (30,126) .. controls (30.13,128.04) and (28.59,129.79) .. (26.55,129.93) .. controls (24.52,130.06) and (22.76,128.52) .. (22.63,126.48) -- cycle ;
        \draw[shift={(.5 cm, .5 cm)}]    (28.31,124.24) -- (89.8,33) ;
        \draw[shift={(.5 cm, .5 cm)}]    (94.55,32.93) -- (156.07,121.55) ;
        \draw[shift={(.5 cm, .5 cm)}]    (30,126) -- (155.63,125.48) ;
            
        \draw[shift={(.5 cm, .5 cm)}] (87,9) node [anchor=north west][inner sep=0.75pt]   [align=left] {1};
        \draw[shift={(.5 cm, .5 cm)}] (166,118) node [anchor=north west][inner sep=0.75pt]   [align=left] {3};
        \draw[shift={(.5 cm, .5 cm)}] (8,117) node [anchor=north west][inner sep=0.75pt]   [align=left] {2};
        \draw[shift={(.5 cm, .5 cm)}] (39,63) node [anchor=north west][inner sep=0.75pt]   [align=left] {12};
        \draw[shift={(.5 cm, .5 cm)}] (129,63) node [anchor=north west][inner sep=0.75pt]   [align=left] {13};
        \draw[shift={(.5 cm, .5 cm)}] (85,130) node [anchor=north west][inner sep=0.75pt]   [align=left] {23};
            
        \end{tikzpicture}

%% file: Tikzpictures/LatticeFlatsExample.tex
\begin{tikzpicture}[x=0.75pt,y=0.75pt,yscale=-1,xscale=1]

        \draw    (25.88,52.79) -- (71.9,27.97) ;
        \draw    (80.08,25.78) -- (80.08,51.1) ;
        \draw    (88.41,27.46) -- (133.43,54.48) ;
        \draw    (25.05,65.45) -- (71.07,87.9) ;
        \draw    (80.91,67.98) -- (80.91,85.71) ;
        \draw    (90.25,86.21) -- (135.93,65.45) ;
        
        \draw (70.19,15.24) node [anchor=north west][inner sep=0.75pt]  [font=\scriptsize] [align=left] {123};
        \draw (16.11,56.39) node [anchor=north west][inner sep=0.75pt]  [font=\scriptsize] [align=left] {1};
        \draw (78.29,57.07) node [anchor=north west][inner sep=0.75pt]  [font=\scriptsize] [align=left] {2};
        \draw (136.8,54.87) node [anchor=north west][inner sep=0.75pt]  [font=\scriptsize] [align=left] {3};
        \draw (76.21,86.96) node [anchor=north west][inner sep=0.75pt]  [font=\scriptsize] [align=left] {$\displaystyle \emptyset $};
        
        \end{tikzpicture}

%% file: Tikzpictures/BuildingSetEX1.tex
\tikzset{every picture/.style={line width=0.75pt}} 

\begin{tikzpicture}[x=0.75pt,y=0.75pt,yscale=-1,xscale=1]

\draw    (70.31,16.09) -- (97.63,30.57) ;
\draw    (59.23,19.05) -- (59.23,31.12) ;
\draw    (20.28,30.03) -- (49.47,16.09) ;
\draw    (19.73,58) -- (19.73,41.54) ;
\draw    (99.27,58.55) -- (99.27,41.54) ;
\draw    (24.18,41.43) -- (54.9,58.99) ;
\draw    (23.57,58) -- (53.85,41.87) ;
\draw    (64.71,41.54) -- (94.34,59.1) ;
\draw    (64.71,57.45) -- (93.79,41.54) ;
\draw    (23.03,66.23) -- (53.3,80.82) ;
\draw    (59.78,67.87) -- (59.78,79.39) ;
\draw    (65.92,79.72) -- (95.98,66.23) ;
\draw  [color={rgb, 255:red, 139; green, 87; blue, 42 }  ,draw opacity=1 ][line width=1.5]  (77.1,52.8) .. controls (92.55,53.28) and (86.76,45.92) .. (88.45,38.92) .. controls (90.14,31.92) and (93.03,25.04) .. (99.79,25.77) .. controls (106.55,26.49) and (106.07,61.97) .. (101.24,65.83) .. controls (96.41,69.7) and (10.97,73.07) .. (9.52,62.46) .. controls (8.07,51.84) and (61.66,52.32) .. (77.1,52.8) -- cycle ;

\draw (47.35,7.46) node [anchor=north west][inner sep=0.75pt]  [font=\scriptsize] [align=left] {123};
\draw (12.17,31.4) node [anchor=north west][inner sep=0.75pt]  [font=\scriptsize] [align=left] {12};
\draw (51.66,31.85) node [anchor=north west][inner sep=0.75pt]  [font=\scriptsize] [align=left] {13};
\draw (90.61,31.85) node [anchor=north west][inner sep=0.75pt]  [font=\scriptsize] [align=left] {23};
\draw (16.2,57.39) node [anchor=north west][inner sep=0.75pt]  [font=\scriptsize] [align=left] {1};
\draw (55.8,57.18) node [anchor=north west][inner sep=0.75pt]  [font=\scriptsize] [align=left] {2};
\draw (93.64,56.4) node [anchor=north west][inner sep=0.75pt]  [font=\scriptsize] [align=left] {3};
\draw (53.89,78.48) node [anchor=north west][inner sep=0.75pt]  [font=\scriptsize] [align=left] {$\displaystyle \emptyset $};

\end{tikzpicture}

%% file: Tikzpictures/BuildingSetEX2.tex
\tikzset{every picture/.style={line width=0.75pt}} 

\begin{tikzpicture}[x=0.75pt,y=0.75pt,yscale=-1,xscale=1]

\draw    (69.14,11.63) -- (96.46,26.11) ;
\draw    (58.06,14.6) -- (58.06,26.66) ;
\draw    (19.11,25.57) -- (48.3,11.63) ;
\draw    (18.56,53.54) -- (18.56,37.09) ;
\draw    (98.1,54.09) -- (98.1,37.09) ;
\draw    (23.01,36.98) -- (53.73,54.53) ;
\draw    (22.4,53.54) -- (52.68,37.41) ;
\draw    (63.54,37.09) -- (93.17,54.64) ;
\draw    (63.54,52.99) -- (92.62,37.09) ;
\draw    (21.86,61.77) -- (52.13,76.36) ;
\draw    (58.61,63.42) -- (58.61,74.93) ;
\draw    (64.75,75.26) -- (94.81,61.77) ;
\draw  [color={rgb, 255:red, 139; green, 87; blue, 42 }  ,draw opacity=1 ][line width=1.5]  (19.87,21.65) .. controls (35.31,22.13) and (60.54,22.89) .. (69.7,22.6) .. controls (78.87,22.31) and (91.86,20.59) .. (98.62,21.31) .. controls (105.38,22.03) and (104.9,57.51) .. (100.07,61.37) .. controls (95.24,65.24) and (9.8,68.62) .. (8.35,58) .. controls (6.9,47.38) and (4.42,21.16) .. (19.87,21.65) -- cycle ;

\draw (46.18,3) node [anchor=north west][inner sep=0.75pt]  [font=\scriptsize] [align=left] {123};
\draw (11,26.94) node [anchor=north west][inner sep=0.75pt]  [font=\scriptsize] [align=left] {12};
\draw (50.5,27.39) node [anchor=north west][inner sep=0.75pt]  [font=\scriptsize] [align=left] {13};
\draw (89.44,27.39) node [anchor=north west][inner sep=0.75pt]  [font=\scriptsize] [align=left] {23};
\draw (15.03,52.93) node [anchor=north west][inner sep=0.75pt]  [font=\scriptsize] [align=left] {1};
\draw (54.63,52.72) node [anchor=north west][inner sep=0.75pt]  [font=\scriptsize] [align=left] {2};
\draw (92.47,51.94) node [anchor=north west][inner sep=0.75pt]  [font=\scriptsize] [align=left] {3};
\draw (52.72,74.02) node [anchor=north west][inner sep=0.75pt]  [font=\scriptsize] [align=left] {$\displaystyle \emptyset $};

\end{tikzpicture}